\documentclass{article}
\title{Quasitubal Tensor Algebra Over Separable Hilbert Spaces}

 \author{Uria Mor
\thanks{School of Mathematical Sciences, Raymond and Beverly Sackler Faculty of Exact Sciences, Tel Aviv University, Tel Aviv 6997801, Israel.}
\and
Haim Avron$^*$
}
\date{\today}

\usepackage[utf8]{inputenc}
\usepackage[skip=0pt , indent=7pt]{parskip}
\usepackage[dvipsnames]{xcolor}
\usepackage{amsmath}
\usepackage{tikz}
\usetikzlibrary{shadows.blur}
\usetikzlibrary{arrows}
\usetikzlibrary{positioning}

\usepackage{caption}
\usepackage{subcaption}
\usepackage{float}

\usepackage[a4paper]{geometry}
\usepackage{algorithm}
\usepackage{algpseudocode}
\usepackage{algorithmicx}
\usepackage{amsthm}
\usepackage{thmtools}
\usepackage{enumitem}
\usepackage{upgreek}

\usepackage{mathtools}
\usepackage{amssymb}
\usepackage[mathscr]{euscript}

\usepackage[colorlinks=true]{hyperref}
\usepackage[nameinlink]{cleveref}
\usepackage{wrapfig}

\DeclareMathAlphabet{\mathbbb}{U}{bbold}{m}{n}
\usepackage{multicol}

\crefname{subsection}{subsection}{subsections}
\crefname{subsubsection}{subsubsection}{subsubsections}
\newtheorem{theorem}{Theorem}[section]
\newtheorem{corollary}[theorem]{Corollary}
\newtheorem{lemma}[theorem]{Lemma}
\newtheorem{claim}[theorem]{Claim}
\newtheorem{proposition}[theorem]{Proposition}
\newtheorem{definition}[theorem]{Definition}

\newtheorem{remark}[theorem]{Remark}

\crefname{req}{requirement}{requirements}
\Crefname{req}{Requirement}{Requirements}
\crefname{prprt}{property}{properties}
\Crefname{prprt}{Property}{Properties}
\crefname{def}{definition}{definitions}
\Crefname{def}{Definition}{Definitions}
\crefname{claim}{claim}{claims}
\Crefname{claim}{Claim}{Claims}
\crefname{q}{question}{questions}
\crefname{lemma}{lemma}{lemmas}
\Crefname{lemma}{Lemma}{Lemmas}
\crefname{obs}{observation}{observations}
\Crefname{obs}{Observation}{Observations}
\crefname{prop}{proposition}{proposition}
\Crefname{prop}{Proposition}{Propositions}
\crefname{cor}{corollary}{corollaries}
\Crefname{cor}{Corollary}{Corollaries}
\Crefname{equation}{Eq.}{Eqs.}

\crefname{example}{example}{examples}
\Crefname{example}{Example}{Examples}

\crefname{conj}{conjecture}{conjectures}
\Crefname{conj}{Conjecture}{Conjectures}

\providecommand{\examplename}{Example}
\newlength\Colsep
\makeatletter

\newcounter{parentsubequation}

\makeatother

\makeatletter
\DeclareFontFamily{OMX}{MnSymbolE}{}
\DeclareSymbolFont{MnLargeSymbols}{OMX}{MnSymbolE}{m}{n}
\SetSymbolFont{MnLargeSymbols}{bold}{OMX}{MnSymbolE}{b}{n}
\DeclareFontShape{OMX}{MnSymbolE}{m}{n}{
    <-6>  MnSymbolE5
   <6-7>  MnSymbolE6
   <7-8>  MnSymbolE7
   <8-9>  MnSymbolE8
   <9-10> MnSymbolE9
  <10-12> MnSymbolE10
  <12->   MnSymbolE12
}{}
\DeclareFontShape{OMX}{MnSymbolE}{b}{n}{
    <-6>  MnSymbolE-Bold5
   <6-7>  MnSymbolE-Bold6
   <7-8>  MnSymbolE-Bold7
   <8-9>  MnSymbolE-Bold8
   <9-10> MnSymbolE-Bold9
  <10-12> MnSymbolE-Bold10
  <12->   MnSymbolE-Bold12
}{}

\let\llangle\@undefined
\let\rrangle\@undefined
\DeclareMathDelimiter{\llangle}{\mathopen}%
                     {MnLargeSymbols}{'164}{MnLargeSymbols}{'164}
\DeclareMathDelimiter{\rrangle}{\mathclose}%
                     {MnLargeSymbols}{'171}{MnLargeSymbols}{'171}
\makeatother

\usepackage{setspace}

\DeclareFontFamily{U}{mathb}{\hyphenchar\font45}
\DeclareFontShape{U}{mathb}{m}{n}{
<-6> mathb5 <6-7> mathb6 <7-8> mathb7
<8-9> mathb8 <9-10> mathb9
<10-12> mathb10 <12-> mathb12
}{}
\DeclareSymbolFont{mathb}{U}{mathb}{m}{n}
\DeclareMathSymbol{\llcurly}{\mathrel}{mathb}{"CE}
\DeclareMathSymbol{\ggcurly}{\mathrel}{mathb}{"CF}

\begin{document}
\maketitle

\newcommand{\noun}[1]{\textsc{#1}\xspace}
\newcommand{\bs}[1]{\ensuremath{\boldsymbol{#1}}}
\newcommand{\hbs}[1]{\ensuremath{\wh{\bs{#1}}}}
\newcommand{\rank}{\text{\ensuremath{\operatorname{rank}}}}
\newcommand{\argmin}{\text{\ensuremath{\arg\min}}}
\newcommand{\argmax}{\text{\ensuremath{\arg\max}}}
\newcommand{\stt}{\,\,\,\text{s.t.}\,\,\,}
\newcommand{\xx}{\ensuremath{{\times}}}
\newcommand{\code}[1]{\codex{#1}}
\newcommand{\pinv}{\ensuremath{\dagger}}
\renewcommand{\P}{\pinv}
\newcommand{\mpn}{\ensuremath{m \xx p \xx n}}
\newcommand{\pmn}{\ensuremath{p \xx m \xx n}}
\newcommand{\innerprod}[2]{\left\langle #1,#2 \right\rangle }
\newcommand{\mpd}{\ensuremath{m \xx p \xx \ddim}}
\newcommand{\pmd}{\ensuremath{p \xx m \xx \ddim}}
\newcommand{\dimsI}{\ensuremath{I_1 \xx \cdots \xx I_d}}
\newcommand{\dimsr}{\ensuremath{r_1 \xx \cdots \xx r_d}}
\newcommand{\scalemath}[2]{\scalebox{#1}{\mbox{\ensuremath{\displaystyle #2}}}}
\newcommand{\dbr}[1]{\ensuremath{\llbracket #1 \rrbracket}}

\newcommand{\tens}[1]{\ensuremath{\bs{\EuScript{#1}}}}
\newcommand{\tA}{\tens{A}}
\newcommand{\tAt}{\tA^{\T}}

\newcommand{\tT}{\tens{T}}

\newcommand{\btA}{\bar{\tens{A}}}
\newcommand{\ttA}{\tilde{\tens{A}}}
\newcommand{\ttAt}{\ttA^{\T}}
\newcommand{\thA}{\widehat{\tA}}
\newcommand{\tthA}{\widehat{\ttA}}
\newcommand{\thAt}{\thA^{\T}}
\newcommand{\tthAt}{\tthA^{\T}}
\newcommand{\tB}{\tens{B}}
\newcommand{\tBt}{\tB^{\T}}
\newcommand{\thB}{\widehat{\tB}}
\newcommand{\thBt}{\thB^{\T}}
\newcommand{\tR}{\tens{R}}
\newcommand{\thR}{\widehat{\tR}}
\newcommand{\tRt}{\tR^{\T}}
\newcommand{\tM}{\tens{M}}
\newcommand{\thM}{\widehat{\tM}}
\newcommand{\tMt}{\tM^{\T}}

\newcommand{\tV}{\tens{V}}
\newcommand{\tVt}{\tV^{\T}}
\newcommand{\tVH}{\tV^{\H}}
\newcommand{\thV}{\widehat{\tV}}
\newcommand{\thVt}{\thV^{\T}}
\newcommand{\tU}{\tens{U}}
\newcommand{\tUt}{\tU^{\T}}
\newcommand{\thU}{\widehat{\tU}}
\newcommand{\thUt}{\thU^{\T}}
\newcommand{\tX}{\tens{X}}
\newcommand{\thX}{\ensuremath{\widehat{\tX}}}
\newcommand{\tY}{\tens{Y}}
\newcommand{\tYt}{\tY^{\T}}
\newcommand{\thY}{\widehat{\tY}}
\newcommand{\thYt}{\thY^{\T}}

\newcommand{\qprm}{\ensuremath{q^{\prime}}}

\newcommand{\tS}{\tens{S}}
\newcommand{\tSt}{\tS^{\T}}
\newcommand{\thS}{\widehat{\tS}}
\newcommand{\thSt}{\thS^{\T}}
\newcommand{\teJ}{\tens{J}}
\newcommand{\theJ}{\widehat{\teJ}}
\newcommand{\tI}{\tens{I}}
\newcommand{\thI}{\widehat{\tI}}
\newcommand{\tC}{\tens{C}}
\newcommand{\tCt}{\tC^{\T}}
\newcommand{\thC}{\widehat{\tC}}
\newcommand{\thCt}{\thC^{\T}}
\newcommand{\tG}{\tens{G}}
\newcommand{\tGt}{\tG^{\T}}
\newcommand{\thG}{\widehat{\tG}}

\newcommand{\tVe}{\tV_{\epsilon}}
\newcommand{\tVet}{\tVe^{\T}}
\newcommand{\thVe}{\widehat{\tV}_\epsilon}
\newcommand{\thVet}{\thVe^{\T}}
\newcommand{\tUe}{\tU_{\epsilon}}
\newcommand{\tUet}{\tUe^{\T}}
\newcommand{\thUe}{\widehat{\tU}_\epsilon}
\newcommand{\thUet}{\thUe^{\T}}
\newcommand{\tSe}{\tS_{\epsilon}}
\newcommand{\tSet}{\tSe^{\T}}
\newcommand{\thSe}{\widehat{\tS}_{\epsilon}}
\newcommand{\thSet}{\thSe^{\T}}

\newcommand{\tCe}{\tC_{\epsilon}}
\newcommand{\tCet}{\tCe^{\T}}
\newcommand{\thCe}{\widehat{\tC}_\epsilon}
\newcommand{\thCet}{\thCe^{\T}}

\newcommand{\tE}{\tens{E}}
\newcommand{\thE}{\widehat{\tE}}
\newcommand{\tQ}{\tens{Q}}
\newcommand{\tQt}{\tQ^{\T}}
\newcommand{\tQh}{\tQ^{\H}}
\newcommand{\thQ}{\widehat{\tQ}}
\newcommand{\thQt}{\thQ^{\T}}
\newcommand{\hsigma}{\ensuremath{\hat{\sigma}}}
\newcommand{\tZ}{\tens{Z}}
\newcommand{\tZt}{\tZ^{\T}}
\newcommand{\thZ}{\widehat{\tZ}}
\newcommand{\thZt}{\thZ^{\T}}

\newcommand{\tW}{\tens{W}}
\newcommand{\tWt}{\tW^{\T}}
\newcommand{\thW}{\widehat{\tW}}
\newcommand{\thWt}{\thW^{\T}}

\newcommand{\tP}{\tens{P}}
\newcommand{\tPt}{\tP^{\T}}
\newcommand{\thP}{\widehat{\tP}}
\newcommand{\thPt}{\thP^{\T}}

\newcommand{\tH}{\tens{H}}
\newcommand{\tHt}{\tH^{\T}}
\newcommand{\thH}{\widehat{\tH}}
\newcommand{\thHt}{\thH^{\T}}

\newcommand{\upbb}[2]{\ensuremath{{#1}^{(#2)}}}

\newcommand{\tWh}{\upbb{\tW}{h}}
\newcommand{\tWht}{( \tWh )^{\T}}
\newcommand{\thWh}{\widehat{\upbb{\tW}{h}}}
\newcommand{\thWht}{(\thWh)^{\T}}

\newcommand{\tCh}{\upbb{\tC}{h}}
\newcommand{\tCht}{(\tCh)^{\T}}
\newcommand{\thCh}{\widehat{\upbb{\tC}{h}}}
\newcommand{\thCht}{(\thCh)^{\T}}

\newcommand{\tWl}{\upbb{\tW, \ell}}
\newcommand{\tWlt}{( \tWh )^{\T}}
\newcommand{\thWl}{\widehat{\upbb{\tW}{\ell}}}
\newcommand{\thWlt}{(\thWh)^{\T}}

\newcommand{\tCl}{\upbb{\tC}{\ell}}
\newcommand{\tClt}{(\tCh)^{\T}}
\newcommand{\thCl}{\widehat{\upbb{\tC}{\ell}}}
\newcommand{\thClt}{(\thCh)^{\T}}

\newcommand{\qb}{Q_{\tB}}
\newcommand{\qv}{Q_{\tV}}
\newcommand{\pb}{P_{\tB}}
\newcommand{\tpb}{\tilde{P}_{\tB}}
\newcommand{\pv}{P_{\tV}}
\newcommand{\rbb}{R_{\matb}}
\newcommand{\rvv}{R_{\matr}}

\newcommand{\mat}[1]{\ensuremath{\mathbf{#1}}}
\newcommand{\matA}{\mat{A}}
\newcommand{\matX}{\mat{X}}
\newcommand{\matR}{\mat{R}}
\newcommand{\matRt}{\matR^\T}
\newcommand{\matQ}{\mat{Q}}
\newcommand{\matE}{\mat{E}}
\newcommand{\matQt}{\matQ^{\T}}

\newcommand{\matQip}{\ensuremath{{\matQ_i}_\perp}}
\newcommand{\matQipt}{\ensuremath{\matQip^\T}}

\newcommand{\matP}{\mat{P}}
\newcommand{\matPt}{\matP^{\T}}
\newcommand{\mattP}{\tilde{\mat{P}}}
\newcommand{\mattPt}{\mattP^{\T}}
\newcommand{\matXt}{\matX^{\T}}
\newcommand{\matAt}{\matA^{\T}}
\newcommand{\bmatA}{\bar{\matA}}
\newcommand{\bmatAt}{\bmatA^{\T}}
\newcommand{\mata}{\mat{a}}
\newcommand{\matV}{\mat{V}}
\newcommand{\matVt}{\ensuremath{\matV^{\T}}}
\newcommand{\matU}{\mat{U}}
\newcommand{\matUt}{\matU^{\T}}
\newcommand{\matS}{\mat{S}}
\newcommand{\matSt}{\ensuremath{\matS^{\T}}}
\newcommand{\matT}{\mat{T}}
\newcommand{\matTt}{\ensuremath{\matT^{\T}}}
\newcommand{\mats}{\mat{s}}
\newcommand{\hmats}{\hat{\mats}}
\newcommand{\matr}{\mat{r}}
\newcommand{\maty}{\mat{y}}
\newcommand{\amatr}{\vec{\matr}}
\newcommand{\matC}{\mat{C}}
\newcommand{\matCt}{\matC^{\T}}

\newcommand{\matAr}{\mat{A}_r}

\newcommand{\matD}{\mat{D}}
\newcommand{\matd}{\mat{d}}
\newcommand{\matds}[1]{\matd^{(#1)}}
\newcommand{\ddi}{\matds{i}}
\newcommand{\ddim}{\ensuremath{D^{\xx}}}

\newcommand{\matG}{\mat{G}}
\newcommand{\matGt}{\ensuremath{\matG^{\T}}}

\newcommand{\matZ}{\mat{Z}}
\newcommand{\matZt}{\ensuremath{\matZ^{\T}}}
\newcommand{\matY}{\mat{Y}}
\newcommand{\matYt}{\ensuremath{\matY^{\T}}}

\newcommand{\matH}{\mat{H}}
\newcommand{\matHt}{\matH^\T}

\newcommand{\matSigma}{\boldsymbol{\Sigma}}
\newcommand{\matOmega}{\boldsymbol{\Omega}}
\newcommand{\matOmegat}{\matOmega^\T}
\newcommand{\matPsi}{\boldsymbol{\Psi}}
\newcommand{\matPsit}{\matPsi^\T}

\newcommand{\matpi}{\boldsymbol{\pi}}
\newcommand{\matpit}{\matpi^\T}

\newcommand{\matK}{\mat{K}}
\newcommand{\matM}{\mathbf{M}}
\newcommand{\matMi}{\matM^{-1}}
\newcommand{\matMt}{\matM^{\T}}
\newcommand{\matW}{\mathbf{W}}
\newcommand{\matWt}{\ensuremath{\matW^{\T}}}
\newcommand{\T}{\mat{T}}
\newcommand{\CT}{\mat{H}}
\newcommand{\bell}{{(\ell)}}

\newcommand*\wh[1]{\widehat{#1}}
\newcommand*\wt[1]{\widetilde{#1}}
\renewcommand{\u}{\mat{u}}
\newcommand{\ut}{\ensuremath{\u^{\T}}}
\newcommand{\hu}{\ensuremath{\hat{\u}}}
\renewcommand{\v}{\mat{v}}
\newcommand{\hv}{\ensuremath{\hat{\v}}}
\newcommand{\rrho}{\ensuremath{\bs{\rho}}}
\newcommand{\eeta}{\ensuremath{\bs{\eta}}}
\newcommand{\pphi}{\ensuremath{\bs{\varphi}}}
\newcommand{\rb}{\mat{r}}
\newcommand{\w}{\mat{w}}
\renewcommand{\b}{\mat{b}}
\newcommand{\bt}{\ensuremath{\b^{\T}}}
\newcommand{\x}{\mat{x}}
\newcommand{\h}{\mat{h}}
\newcommand{\g}{\mat{g}}
\newcommand{\e}{\mat{e}}
\newcommand{\xt}{\x^{\T}}
\newcommand{\gt}{\g^{\T}}
\newcommand{\vt}{\v^{\T}}
\newcommand{\q}{\mat{q}}
\newcommand{\qt}{\q^{\T}}
\newcommand{\p}{\mat{p}}
\newcommand{\y}{\mat{y}}
\newcommand{\bmaty}{\bar{\y}}
\newcommand{\bmatx}{\bar{\x}}
\newcommand{\z}{\mat{z}}
\newcommand{\matB}{\mat{B}}
\newcommand{\matI}{\mat{I}}
\newcommand{\matF}{\mat{F}}
\newcommand{\matFt}{\matF^{\T}}
\newcommand{\matBt}{\matB^{\T}}
\newcommand{\si}{\ensuremath{^{(i)}}}
\newcommand{\vpsi}{\boldsymbol{\psi}}
\newcommand{\vtau}{\boldsymbol{\tau}}
\newcommand{\vepsilon}{\boldsymbol{\epsilon}}
\newcommand{\vmu}{\boldsymbol{\mu}}
\newcommand{\vom}{\boldsymbol{\omega}}
\newcommand{\matGamma}{\boldsymbol{\Gamma}}
\newcommand{\TT}{\mathcal{T}}
\newcommand{\DD}{\mathcal{D}}
\newcommand{\HH}{\mathcal{H}}
\newcommand{\FF}{\mathcal{F}}
\newcommand{\LL}{\mathcal{L}}
\newcommand{\LT}{{\LL^2}}

\newcommand{\ovec}{\operatorname{vec}}
\newcommand{\reshape}{\operatorname{reshape}}

\newcommand{\matb}{\mat{b}}
\newcommand{\matbt}{\matb^{\T}}
\newcommand{\mate}{\mat{e}}
\newcommand{\matet}{\mate^{\T}}
\newcommand{\matXi}{\boldsymbol{\Xi}}

\newcommand{\bsc}{\bs{c}}
\newcommand{\bsct}{\bsc^{\T}}

\newcommand{\tlX}{\tilde{X}}
\newcommand{\tlY}{\tilde{Y}}

\newcommand{\RR}{\mathbb{R}}
\newcommand{\NN}{\mathbb{N}}
\newcommand{\CC}{\mathbb{C}}
\newcommand{\FFF}{\mathbb{F}}
\newcommand{\EE}{\mathbb{E}}
\newcommand{\ZZ}{\mathbb{Z}}
\newcommand{\sgnErr}{E_{SGN}}
\newcommand{\tsgnErr}{\tE_{SGN}}
\newcommand{\Ei}{E^{(i)}}
\newcommand{\spacex}{\boldsymbol{\mathcal{X}}}
\newcommand{\spacey}{\boldsymbol{\mathcal{Y}}}

\newcommand{\dotp}[1]{\langle #1 \rangle}
\newcommand{\dotps}[1]{\dotp{#1}^2}
\newcommand{\FDot}[1]{\dotp{#1}_{F}}
\newcommand{\FDotS}[1]{\dotps{#1}_{F}}
\newcommand{\LTdotp}[1]{\dotp{#1}_{\LL^2}}
\newcommand{\LTdotps}[1]{\LTdotp{#1}^2}

\newcommand{\HNorm}[1]{{\left\vert\kern-0.25ex\left\vert\kern-0.25ex\left\vert #1 \right\vert\kern-0.25ex\right\vert\kern-0.25ex\right\vert}}
\newcommand{\HNormS}[1]{\HNorm{#1}}

\newcommand{\TNorm}[1]{\|#1\|_{2}}
\newcommand{\TNormS}[1]{\TNorm{#1}^2}
\newcommand{\GNorm}[1]{\|#1\|}
\newcommand{\GNormS}[1]{\GNorm{#1}^2}
\newcommand{\GNormd}[1]{\left\|#1\right\|}
\newcommand{\GNormSd}[1]{\GNormd{#1}^2}
\newcommand{\FNorm}[1]{\|#1\|_{F}}
\newcommand{\FNormS}[1]{\FNorm{#1}^2}
\newcommand{\FNormd}[1]{\lVert#1\rVert_{F}}
\newcommand{\FNormSd}[1]{\FNormd{#1}^2}
\newcommand{\NNorm}[1]{\|#1\|_{*}}

\newcommand{\LTNorm}[1]{\|#1\|_{\LT}}
\newcommand{\LTNormS}[1]{\LTNorm{#1}^2}

\newcommand{\tNorm}[1]{\|#1\|_{top}}
\newcommand{\tNormS}[1]{\tNorm{#1}^2}

\newcommand{\war}[1]{\ensuremath{\overrightarrow{\vphantom{A}{#1}}}}
\newcommand{\bwar}[1]{\war{\bs{#1}}}

\newcommand{\trc}{\ensuremath{\operatorname{Tr}}}
\newcommand{\trace}[1]{\ensuremath{\operatorname{Tr}\left(#1 \right)}}
\newcommand{\ftr}[1]{\ensuremath{\operatorname{f-Tr}(#1)}}

\newcommand{\clr}{\noun{clr}}
\newcommand{\rclr}{\noun{rclr}}
\newcommand{\tsvdm}{\textsc{tsvdm}\xspace}
\newcommand{\tcam}{\textsc{tcam}\xspace}
\newcommand{\tca}{\textsc{tca}\xspace}
\newcommand{\tsvdmii}{\ensuremath{t}-\textsc{svdmii}\xspace}

\newcommand{\calf}{{\cal F}}
\newcommand{\mx}[1]{\ensuremath{\times_{#1}}}
\newcommand{\tsub}[1]{\ensuremath{\times_{#1}}}
\newcommand{\tsM}{\ensuremath{\tsub{3}\matM}}
\newcommand{\tsMinv}{\ensuremath{\tsub{3}\matM^{-1}}}
\newcommand{\ttprod}[1]{{\star}_{{\scriptscriptstyle{\hspace{-1pt}#1}\hspace{1pt} }}}
\newcommand{\Mprod}{\ttprod{\matM}}
\newcommand{\mm}{\ttprod{\matM}}
\newcommand{\ff}{\ttprod{{F}}}
\newcommand{\pp}{\ttprod{{\Phi}}}
\newcommand{\muni}{\ensuremath{\Mprod{\textnormal{-unitary}}}\xspace}
\newcommand{\morth}{\ensuremath{\Mprod{\textnormal{-orthogonal}}}\xspace}
\newcommand{\pmorth}{\ensuremath{\textnormal{pseudo }\Mprod{\textnormal{-orthogonal}}}\xspace}
\newcommand{\Pmorth}{\ensuremath{\textnormal{Pseudo }\Mprod{\textnormal{-orthogonal}}}\xspace}
\newcommand{\rnk}{\ensuremath{\operatorname{rank}}}
\newcommand{\Npr}{\ensuremath{N^{\prime}}}
\newcommand{\unf}[2]{\ensuremath{{#1}_{(#2)}} }
\newcommand{\supl}[1]{\ensuremath{{#1}^{[\ell]}} }
\newcommand{\fld}[2]{\ensuremath{\operatorname{fold}({#1}, #2)}}

\newcommand{\XX}{\hspace*{1pt}{\textstyle\bigtimes}}

\newcommand{\OX}{{\textstyle \bigotimes}}
\newcommand{\Mpinv}{\mathbf{+}}
\newcommand{\mmpinv}{\ensuremath{\Mprod}\textnormal{-pseudo inverse}\xspace}

\newcommand{\fM}{\ensuremath{\mathfrak{M}} }

\algnewcommand{\IfThenElse}[3]{
	\State \algorithmicif\ #1\ \algorithmicthen\ #2\ \algorithmicelse\ #3}
\algnewcommand\Input{\item[\textbf{Input:}]}%
\algnewcommand\algorithmicinput{\textbf{Input:}}
\algnewcommand\INPUT{\item[\algorithmicinput]}

\algnewcommand\Noln{\item[\hspace{28pt}]}%
\algnewcommand\algorithmicnoln{\hspace{28pt}}
\algnewcommand\NOLN{\item[\algorithmicnoln]}

\algnewcommand\SState{\State \hskip-1.em }
\algnewcommand\algorithmicswitch{\textbf{switch}}
\algnewcommand\algorithmiccase{\textbf{case}}
\algnewcommand\algorithmicassert{\texttt{assert}}
\algnewcommand\Assert[1]{\State \algorithmicassert(#1)}

\algnewcommand\Output{\item[\textbf{Output:}]}%

\algnewcommand\Params{\item[\textbf{Parameters:}]}%
\algnewcommand\algorithmicparams{\textbf{Parameters:}}
\algnewcommand\PARAMS{\item[\algorithmicparams]}

\algnewcommand\Outputline{\item[]}%
\newcommand{\samplemode}{\emph{sample mode} }
\newcommand{\featuremode}{\emph{feature mode} }
\newcommand{\omx}{{\it omics} }
\newcommand{\LFB}{LFB }
\newcommand{\DFB}{DFB}
\newcommand{\supp}[1]{\ensuremath{\operatorname{supp}\{#1\}}}

\newcommand{\sigmaset}{{\scalebox{1.4}{$\upsigma$}}}


\newcommand{\pbxstd}{Fig.1b\xspace}
\newcommand{\pbxtca}{Fig.1c\xspace}
\newcommand{\pbxfunnel}{Fig.1d\xspace}
\newcommand{\pbxbars}{Fig.1e\xspace}

\newcommand{\fibersctf}{Fig.fb\xspace}
\newcommand{\fiberstca}{Fig.1g\xspace}
\newcommand{\fibersfunnel}{Fig.1h\xspace}
\newcommand{\fiberstimeseries}{Fig.1i\xspace}
\newcommand{\fibersheatmap}{Fig.1j\xspace}

\newcommand{\ibdrocs}{Fig.2a\xspace}
\newcommand{\ibdloadings}{Fig.2b\xspace}
\newcommand{\ibdtca}{Fig.2c\xspace}
\newcommand{\ibdtimeseries}{Fig.2d\xspace}

\newcommand{\snydertca}{Fig.2e\xspace}
\newcommand{\snyderheatmap}{Fig.2f\xspace}

\newcommand{\Uria}[1]{\textcolor{blue}{[Uria: #1]}}
\newcommand{\Haim}[1]{\textcolor{red}{[Haim: #1]}}

\definecolor{PinkC}{HTML}{ff66ff}
\definecolor{GreenC}{HTML}{33cc33}
\definecolor{BlueC}{HTML}{0099ff}
\definecolor{PurpleC}{HTML}{cc00cc}

\newcommand{\cpink}[1]{\textcolor{PinkC}{#1}}
\newcommand{\cgreen}[1]{\textcolor{GreenC}{#1}}
\newcommand{\cblue}[1]{\textcolor{BlueC}{#1}}
\newcommand{\cpurple}[1]{\textcolor{PurpleC}{#1}}

\newcommand{\matMs}{\ensuremath{\matM}^*}
\newcommand{\YY}{\ensuremath{\mathbb{Y}}}
\newcommand{\hYY}{\ensuremath{\widehat{\mathbb{Y}}}}
\newcommand{\YYr}{\ensuremath{\YY_r}}
\newcommand{\hYYr}{\ensuremath{\hYY_r}}

\newcommand{\spc}{\ensuremath{\texttt{spec}}}
\newcommand{\wc}{\ensuremath{\rightharpoonup}}

\newcommand{\doublecref}[2]{%
  \hyperref[#2]{\namecref{#1}~\labelcref*{#1}~\ref*{#2}}%
}
\newcommand{\doubleCref}[2]{%
  \hyperref[#2]{\labelcref*{#1}~\Cref{#2}}%
}

\newif\ifuriaread 
\uriareadfalse 
\newcommand{\ureminder}[1]{%
    \Uria{%
        Set \textbackslash{}uriaread to false if you want to remove this reminder.\newline%
        #1%
    }}
\newcommand{\urif}[1]{%
    \ifuriaread%
        \ureminder{#1}%
    \else%
    \fi}

\newcommand{\alg}[1]{\ensuremath{\mathcal{#1}} }
\newcommand{\TA}{\alg{A}}
\newcommand{\TB}{\alg{B}}
\newcommand{\TM}{\alg{M}}
\newcommand{\TE}{\alg{E}}
\newcommand{\FT}{\alg{T}}
\newcommand{\FQ}{\alg{Q}}
\newcommand{\FR}{\alg{R}}

\newcommand{\ellinf}{\ell_\infty}
\newcommand{\elltwo}{\ell_2}

\newcommand{\hspc}[1]{\ensuremath{\boldsymbol{ #1}}}
\newcommand{\cH}{\ensuremath{{\alg H}}}
\newcommand{\cE}{\ensuremath{{\hspc E}}}
\newcommand{\chH}{\ensuremath{\widehat{\cH}}}
\newcommand{\chHp}{\ensuremath{\widehat{\cH}^p}}

\newcommand{\cX}{\ensuremath{{\hspc X}}}

\newcommand{\cB}{\ensuremath{{\hspc B}}}
\newcommand{\chB}{\ensuremath{\widehat{\cB}}}
\newcommand{\chBp}{\chB^p}
\newcommand{\chBpd}{(\chB^p)^*}
\newcommand{\fphi}{\FF_{\bs{\varphi}}}
\newcommand{\fphis}{{\fphi}^*}

\newcommand{\hhat}[1]{\hat{\hspace{-1pt}{#1}}}
\newcommand{\dotmp}[1]{\left\llangle #1 \right\rrangle}
\newcommand{\dotph}[1]{\dotp{#1}_{\cH}}
\newcommand{\dotphh}[1]{\dotp{#1}_{\chH}}
\newcommand{\dotphhp}[1]{\dotp{#1}_{\chH^p}}
\newcommand{\mpT}{\ensuremath{m \xx p \xx {\hspc T}}}
\newcommand{\mpzo}{\ensuremath{m \xx p \xx [0,1]}}
\newcommand{\cHNorm}[1]{\GNorm{#1}_{\cH}}
\newcommand{\cHNormS}[1]{\cHNorm{#1}^2}
\newcommand{\cHsNorm}[1]{\GNorm{#1}_{\cHs}}
\newcommand{\cHsNormS}[1]{\cHsNorm{#1}^2}
\newcommand{\cHsNormd}[1]{\GNormd{#1}_{\cHs}}
\newcommand{\cHsNormSd}[1]{\cHsNormd{#1}^2}

\newcommand{\dotpx}[1]{\dotp{#1}_{\cX}}
\newcommand{\cXNorm}[1]{\GNorm{#1}_{\cX}}
\newcommand{\cXNormS}[1]{\cXNorm{#1}^2}

\newcommand{\HSNorm}[1]{\GNorm{#1}_{\operatorname{HS}}}
\newcommand{\HSNormS}[1]{\HSNorm{#1}^2}
\newcommand{\dotphs}[1]{\dotp{#1}_{\operatorname{HS}}}

\newcommand{\Blt}{\ensuremath{{\hspc B}(\ell_2)}}
\newcommand{\BltNorm}[1]{\GNorm{#1}_{\Blt}}
\newcommand{\BltNormS}[1]{\BltNorm{#1}^2}

\newcommand{\ltNorm}[1]{\GNorm{#1}_{\ell_2}}
\newcommand{\ltNormS}[1]{\ltNorm{#1}^2}
\newcommand{\dotplt}[1]{\dotp{#1}_{\ell_2}}

\newcommand{\cY}{\ensuremath{{\hspc Y}}}
\newcommand{\cYd}{\cY^*}
\newcommand{\cYNorm}[1]{\GNorm{#1}_{\cY}}
\newcommand{\cYNormS}[1]{\cYNorm{#1}^2}

\newcommand{\cYp}{\cY'}
\newcommand{\cYpNorm}[1]{\GNorm{#1}_{\cYp}}
\newcommand{\cYpNormS}[1]{\cYpNorm{#1}^2}

\newcommand{\infNorm}[1]{\GNorm{#1}_{\infty}}
\newcommand{\infNormS}[1]{\infNorm{#1}^2}

\newcommand{\chHNorm}[1]{\GNorm{#1}_{\chH}}
\newcommand{\chHNormS}[1]{\chHNorm{#1}^2}

\newcommand{\chHpNorm}[1]{\GNorm{#1}_{\chHp}}
\newcommand{\chHpNormS}[1]{\chHpNorm{#1}^2}

\newcommand{\op}{\ensuremath{{\operatorname{op}}}}

\newcommand{\opNorm}[1]{\GNorm{#1}_{\op}}
\newcommand{\opNormS}[1]{\opNorm{#1}^2}
\newcommand{\opNormd}[1]{\GNormd{#1}_{\op}}
\newcommand{\opNormSd}[1]{\opNormd{#1}^2}

\newcommand{\ropNorm}[1]{\GNorm{#1}_{\FFF{-}\op}}
\newcommand{\ropNormS}[1]{\ropNorm{#1}^2}

\newcommand{\cHopNorm}[1]{\GNorm{#1}_{\cH{-}\op}}
\newcommand{\cHopNormS}[1]{\cHopNorm{#1}^2}

\newcommand{\fqopNorm}[1]{\GNorm{#1}_{\FQ{-}\op}}
\newcommand{\fqopNormS}[1]{\fqopNorm{#1}^2}

\newcommand{\cali}{\ensuremath{\hspc I}}

\newcommand{\cPo}{\ensuremath{{\hspc P}_1}}
\newcommand{\cHz}{\ensuremath{{\hspc H}_0}}
\newcommand{\cHo}{\ensuremath{{\hspc H}_1}}
\newcommand{\cHs}{\cH_{*}}
\newcommand{\cHsn}{\cH_{\divideontimes}}
\newcommand{\cHsq}[1]{\cH_{*^{#1}}}

\newcommand{\calfi}{\ensuremath{{\calf}^{-1}}}

\renewcommand{\Re}[1]{\operatorname{Re}(#1)}
\renewcommand{\Im}[1]{\operatorname{Im}(#1)}

\newcommand{\Hom}{\operatorname{Hom}}
\newcommand{\Aut}{\operatorname{Aut}}
\newcommand{\var}{\mathtt{var}}
\newcommand{\cov}{\mathtt{cov}}
\newcommand{\cor}{\mathtt{cor}}
\newcommand{\mean}{\mathtt{mean}}

\newcommand{\fqp}{\FQ^p}
\newcommand{\chp}{\cH^p}
\newcommand{\qrnk}{\ensuremath{\FQ\operatorname{{-}rank}}}

\newcommand{\fqmpNorm}[1]{\GNorm{#1}_{\FQ^{m \xx p}}}
\newcommand{\fqmpNormS}[1]{\GNormS{#1}_{\FQ^{m \xx p}}}

\begin{abstract}
The tubal tensor framework provides a clean and effective algebraic setting for tensor computations, supporting matrix-mimetic features like Singular Value Decomposition and Eckart–Young-like optimality results. Underlying the tubal tensor framework is a view of a tensor as a matrix of finite sized tubes. In this work, we lay the mathematical and computational foundations for working with tensors with infinite size tubes: matrices whose elements are elements from a separable Hilbert space. A key challenge is that existence of important desired matrix-mimetic features of tubal tensors rely on the existence of a unit element in the ring of tubes. Such unit element cannot exist for tubes which are elements of an infinite-dimensional Hilbert space. We sidestep this issue by embedding the tubal space in a commutative unital C*-algebra of bounded operators. The resulting {\em quasitubal algebra}  recovers the structural properties needed for decomposition and low-rank approximation. In addition to laying the theoretical groundwork for working with tubal tensors with infinite dimensional tubes, we discuss computational aspects of our construction, and provide a numerical illustration where we compute a finite dimensional approximation to a infinitely-sized synthetic tensor using our theory.  
We believe our theory opens new exciting avenues for applying matrix mimetic tensor framework in the context of inherently infinite dimensional problems.
\end{abstract}

\section{Introduction}
Multilinear data is ubiquitous in modern scientific and computational domains.
In many contemporary applications, such as signal processing with multi-sensor arrays~\cite{Sidiropoulos2017,Miron2020}, analysis of high dimensional dynamical systems ~\cite{Mao2024,Kutz2016}, and operator-valued regression in infinite-dimensional feature spaces~\cite{Jin2022} , one encounters data with inherent multi-dimensional structure. 
When doing computations in such settings, it is essential to employ representations and computational frameworks that preserve the dimensional integrity of the data, namely, its multilinear structure, in order to extract meaningful and computationally tractable insights.
Tensor decomposition frameworks have emerged as indispensable tools for this task, offering principled methods for compressing, analyzing, and approximating multi-way data in a manner that respects its multilinear structure~\cite{KoldaBader2009}.

However, since classical tensor algebra lacks the elegance and prowess of matrix algebra, traditionally there was a steep price in using tensor based method in lieu of matrix based methods. Initially developed in a series of seminal papers~\cite{kilmer2008, Braman10,KilmerMartin11, KBHH13,Kernfeld2015} between 2008-2015, the tubal tensor promises to alleviate this issue by offering a tensor algebra that is {\em matrix mimetic} by design: it equips higher-order tensors with an algebraic structure that closely parallels matrix algebra.

In the tubal tensor framework, higher order tensors are regarded matrices of ``tubes'', that are elements of a vector space over which, in addition to the usual vector space operations, a binary multiplication is defined -- the tubal product~\cite{Braman10}. 
The notion of tubal multiplication provides the space of tubes with commutative ring structure. 
The scalar-like role played tubes give rise to the possibility of multiplying matrices of tubes, and more importantly, factorizing such matrices. 
Importantly, this structure supports a version of the singular value decomposition (tSVD), along with an Eckart–Young-like theorem that guarantees optimality of low-rank approximations under several, natural notions ranks~\cite{KilmerPNAS}.
These optimality results, together with the matrix mimetic operations which are familiar to most, have made the tubal tensor framework a powerful and elegant model for tensor computations~\cite{siamCongratulations2025}. 

Tubal tensors, much like matrices, have been defined with finite modes: there is a finite set of rows, columns, and tubes are of finite size. However, many real-world scenarios are best modeled via continuous processes. This has promoted the introduction of {\em quasimatrices} as a useful algebra for working with data in which one mode is finite, while the other is a continuous function (typically in a Hilbert spae)~\cite{Trefethen2009,Stewart1998,FShustin2022}. 

Recently there has been interest in introducing functional modes in the context of tensors. Han et al. ~\cite{Han2023}, introduced the concept of  {\em functional tensor}, which is a tensor in which there exists a mode which is indexed by a continuous domain.  Larsen et al. in a concurrent work~\cite{LARSEN2024RKHS}, introduced the concept of {\em quasitensors} as a tensorial generalization of quasimatrices, which they define as a tensor whose one or more modes are indexed by an infinite set. Both of these work draw on classical tensor algebra, e.g., consider continuous analogs of CP-decomposition.  

The goal of this paper is present the mathematical and computational foundations of working with tubal tensors in which tubes are continuous functions. Specifically, we assume tubes are elements of a separable infinite dimensional Hilbert space. The resulting framework, which we refer to as  {\em quasitubal tensor algebra}\footnote{Why we use the term ``quasitubal tensor'' and not ``tubal quasitensor'' will become evident once we introduce our construction. A more accurate descriptive name might have been ``quasitubal quasitensor'', but we preferred to avoid the double use of ``quasi''.}, recovers the key algebraic features of the finite theory while extending its merits, namely, the optimality of low rank truncations, to infinite-dimensional settings. 
This provides a principled model for representing and decomposing infinite dimensional tensors in a way that retains dimensional integrity, orthogonality, and approximation optimality. 

The cornerstone of tubal tensor algebra appeal is it's matrix-mimetic SVD decomposition, and Eckart-Young-like optimality results.  Obviously, any continuous generalization must maintain these qualities. Here a complication occurs: SVD relies on the notion of unitarity, while, as we shall see,  an algebra over infinite tubes cannot be unital. The reason is that tubal tensor-tensor product assigns each tubal tensor with a Hilbert-Schmidt (HS) operator. So, the existence of an identity tensor would imply that the identity operator is a HS operator, a situation that can only occur in finite dimensional spaces. 

We sidestep this issue by embedding the set of tubes in a larger set, of {\em quasitubes}, in which we can define a commutative unital C*-algebra. Formally, quasitubes forms a subset of bounded linear operator on the separable Hilbert space which forms the tubes. This space is shown to be isometrically isomorphic to $\ell_\infty$ sequences, which is a Banach space and not a Hilbert space.

Once we have the basic infrastructure in place, we go ahead and prove the existence of quasitubal SVD (\Cref{thm:quasitubal.svd.exists}) and a Eckart-Young-like optimality result. We actually have two Eckart-Young-like results. The first result (\Cref{thm:qt.best.lowrank.op}), the simpler of the two, works completely in the realm of quasitubal tensors. The second result (\Cref{thm:best.finite.approx.cH}) goes further, and shows that tubal tensors over seperable Hilbert space can be well approximated using finite tubal tensors. Though both the approximated and the approximation tensors are not quasitubal, we need to go through quasitubes in the proof. 

Though our paper is mostly theoretical in nature, aimed at laying the groundwork for working with tubal tensors with infinite dimensional tubes, we discuss computational aspects of our construction, and provide a numerical illustration where we compute a finite dimensional approximation to a infinitely-sized synthetic tensor using our theory.
We believe our theory opens new exciting avenues for applying matrix mimetic tensor framework in inherently infinite dimensional problems in areas such as functional analysis, operator learning, and analysis and data driven modeling of dynamical systems.

\section{Preliminaries}
\subsection{Notations and Basic Definitions}
We use lowercase Latin and Greek letters to denote scalars in a field $\FFF$ (either $\RR$ or $\CC$), e.g., $a,\beta \in \FFF$.
The $n$-dimensional vector space over a field $\FFF$ is denoted by $\smash{\FFF^n}$, and the elements are denoted by lowercase, bold Latin letters, e.g., $\smash{\mat{x} \in \FFF^n}$.
We use uppercase boldface letters for matrices over $\FFF$, e.g., $\smash{\mat{X} \in \FFF^{m \xx p}}$, and Euler script letters for higher order tensors, e.g., $\tens{X} \in \FFF^{d_1 \xx \cdots \xx d_N}$ is an $N$-order tensor with dimensions $d_1, \ldots, d_N$.
The entries of an $N$-order tensor $\tens{X} \in \FFF^{d_1 \xx \cdots \xx d_N }$ are indexed by $N$-tuples of integers $(i_1, \ldots, i_N)$ with $i_k \in [d_k]$ for $k \in [N]$, where $[K] \coloneqq \{1, \ldots, K\}$.
Consequently, the $j$-th entry of $\mat{x} \in \FFF^n$ is $x_j \in \FFF$, the $(i,j)$-th entry of $\mat{X} \in \FFF^{m \xx p}$ is $x_{i,j} \in \FFF$, and $x_{i_1, \ldots, i_N}$ is the entry of $\tens{X}$ at the $N$-tuple $(i_1, \ldots, i_N)$.

The mode-$k$ tensor-times-matrix product  of a tensor $\smash{\tens{X} \in \FFF^{d_1 \xx \cdots \xx d_N}}$ and a matrix $\smash{\mat{M} \in \FFF^{r \xx d_k}}$ is denoted by $\smash{\tens{X} \xx_k \mat{M} \in \FFF^{d_1 \xx \cdots \xx d_{k-1} \xx r \xx d_{k+1} \xx \cdots \xx d_N}}$, that is the tensor whose fiber $(i_1, \ldots, i_{k-1}, :, i_{k+1}, \ldots, i_N)$ is given by  $\smash{\mat{M} \mat{x}_{i_1, \ldots, i_{k-1}, :, i_{k+1}, \ldots, i_N} \in \FFF^r}$.
The Frobenius norm of an $N$-order tensor $\tens{X}$ is
\[\smash{\FNorm{\tens{X}} \coloneqq \sqrt{{\sum}_{i_1=1}^{d_1} \cdots {\sum}_{i_N=1}^{d_N} |x|_{i_1, \ldots, i_N}^2}}\]
which coincides with the spectral norm when $N=1$.
To avoid confusion, we use $\FNorm{\cdot}$ to denote the Frobenius norm of tensors of any order.

Sequences over $\FFF$, which are mappings from $\ZZ$ to $\FFF$, are also denoted by boldsymbol lowercase letters, e.g., $\bs{x} \in \FFF^{\ZZ}$. We also write $x_j \coloneqq \bs{x}(j) \in \FFF$.
Unless otherwise stated, it is assumed that the sequences are indexed by integers ($\ZZ$) and that their elements are scalars in $\FFF$, where, again, unless otherwise stated $\FFF = \CC$.
The operations of addition, scalar multiplication, and complex conjugation are defined element-wise, e.g., the $j$-th element values for $\bs{x} + \bs{y}, \alpha \bs{x}, \overline{\bs{x}}$ are $x_j + y_j, \alpha x_j, \overline{x_j}$ respectively. All the sums in this paper are to be interpreted as unordered sums, i.e., given an indexed set $\{ \bs{x}_\alpha \in X | \alpha \in A  \}$ then $\sum_{\alpha \in A} \bs{x}_{\alpha}$  is the limit of finite partial sums independently of the way in which the terms are ordered and summed. 

For real positive $p$, we denote by $\ell_p(\ZZ)$ the space of $p$-power summable sequences, and the $p$-norm of $\bs{x} \in \ell_p(\ZZ)$ is $\smash{\GNorm{\bs{x}}_{\ell_p} \coloneqq \left( \sum_{j \in \ZZ} |x_j|^p \right) {}^{1/p}}$. Since the only index set we use is $\ZZ$, we drop it and just write $\ell_p$. 
Of particular interest is the space $\ell_2$, which is a Hilbert space when equipped the canonical inner-product  $\smash{\dotplt{\bs{x}, \bs{y}} \coloneqq \sum_{j \in \ZZ} x_j \overline{y_j} }$, and the space $\ell_\infty$ of sequences $\bs{x}$ such that $\smash{\infty > \infNorm{\bs{x}} \coloneqq \lim_{p \to \infty} \GNorm{\bs{x}}_{\ell_p} = \sup_{j\in \ZZ} | x_j |  } $.




General infinite dimensional vector spaces are denoted by boldsymbol, capital English letters, e.g., $\alg{H}$, and the elements are also denoted by boldface lowercase letters, e.g., $\bs{x} \in \alg{H}$.
Additional, context dependent structure and properties of a space $\alg{H}$ will be stated explicitly in the relevant sections.

\subsection{Functional Tensors and Quasitensors}
Borrowing the terminology of Han et al. ~\cite{Han2023}, a functional tensor $\tX$ is a tensor in which there exists a mode which is indexed by a continuous domain.  
For example, let ${\mathfrak J} \subset \RR$ be a connected segment, then a functional tensor $\tX \in \FFF^{m \xx p \xx {\mathfrak J}}$ is a tensor whose entries are $x_{i,j,t} \in \FFF$ for $i \in [m], j \in [p], t \in {\mathfrak J}$.
Han et al. referred to the first and second modes of $\tX$ in the above example as `tabular modes', and the third mode as a `functional mode'.
A functional mode's location in the tensor is not restricted to the last mode, and the tensor can have multiple functional modes, e.g., $\tX \in \FFF^{{\mathfrak I} \xx m \xx {\mathfrak J} }$ is also a functional tensor with two functional modes, and one tabular mode.

The term quasitensor, used by Larsen et al. in a concurrent work~\cite{LARSEN2024RKHS}, inspired by the definition of quasimatrix~\cite{Trefethen2009,Stewart1998,FShustin2022}, refers to a tensor whose one or more modes are indexed by an infinite set, e.g., $\tX \in \FFF^{m \xx p \xx \infty}$.
Both \cite{Han2023} and \cite{LARSEN2024RKHS} focus on the case where modes fibers whose index sets are infinite, represent functions in a reproducing kernel Hilbert space (RKHS), and investigate the problem of constructing functional tensors $\tX \in \RR^{m \xx p \xx \mathfrak{I}}$ from finite samples, indexed by $\Omega \subseteq [m]\xx[p]\xx {\mathfrak I}$ (without loss of generality ${\mathfrak I} = [0,1]$). 
Namely, these works consider the CP factorization as a model for generating the data by minimizing the reconstruction error $\FNorm{P_{\Omega} (\tX - \sum_{\ell=1}^r \mat{v}_\ell \circ \mat{u}_\ell \circ \bs{\xi}_\ell )}$ over $\mat{v}_{\ell}, \mat{u}_{\ell}$ in $\RR^m,\RR^p$ respectively, where $\bs{\xi}_{\ell} \colon {\mathfrak I} \to \RR$ being functions in some RKHS $\alg{H}_{K} $,   and $P_{\Omega}$ is the orthogonal projection onto the set of observed entries $\Omega$, $\circ$ denotes the outer product of arrays.

\subsection{Tubal Tensor Algebra}\label{sec:mm.tubal.tensor.algebra}
The tubal-tensor paradigm ~\cite{KilmerMartin11,Braman10,KilmerPNAS,Kernfeld2015} provides a matrix mimetic framework for tensors by viewing tensors as ``matrices of tubes'', where a tube is scalar-like object which we can scale, add another tube to it, and - most importantly - multiply by another tube to obtain a third one.
For example, a order-3 tensor $\tens{X} \in \FFF^{m \xx p \xx n}$ is considered as an $m \xx p$ matrix of vectors in $\FFF^n$, each one being a ``tube''.
The tubes form a commutative ring  that is very close to having all the properties of a field. 

We now define the ring of tubes. The elements are vectors $\FFF^n$. Addition is given using the usual vector addition. As for multiplying two tubes, given an invertible $n \xx n$ matrix $\matM$, define 
\begin{equation}\label{eq:F.sup.n.star.m}
    \mat{x} \mm \mat{y} \coloneqq \matMi(\hat{\mat{x}} \odot \hat{\mat{y}}) , \quad \mat{x}, \mat{y} \in \FFF^n
\end{equation}
where $\hat{\mat{x}} = \matM \mat{x}$, and $\odot$ denotes the Hadamard product of two vectors \cite{KilmerPNAS}. 
For every invertible $\matM$, $\FFF^n$ equipped with vector addition for addition and $\mm$ for multiplication is a commutative ring. 

We denote $\smash{\FFF^n}$ equipped with this additional ring structure  by $\FFF_n$, and the elements are denoted by lowercase boldsymbol letters, e.g., $\bs{x} \in \FFF_n$ and are thought of as scalars. We can write vectors and matrices of such scalars, e.g., $\mat{X} \in \FFF_n^p$ and $\tens{X} \in \FFF_n^{m \xx p}$. The former can be thought of as matrices twisted inwards, and the latter as matrices of vectors twisted inwards, hence the use of matrix and tensor notation.

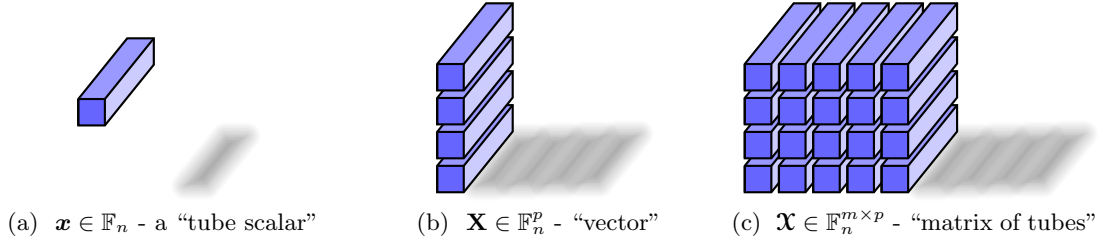
\begin{figure}[H]
    \centering
    \subcaptionbox{\label{fig:ttube.graph} $\bs{x} \in \FFF_n$ - a ``tube scalar''}[.3\linewidth]
    {

\begin{tikzpicture}
    \def\width{0.35}    
    \def\height{\width}   
    \def\depth{2.3*\width}   
    \def\scl{0.8} 
    \def\spacing{.45}     
    \def\shadowx{0.25}
    \def\shadowy{0.15}

    \foreach \y in {2} {
        \begin{scope}[shift={( \y * \spacing,0)}]
            %
            \fill[ opacity=0.0, blur shadow={shadow blur steps=10,shadow scale = .82, shadow yshift=\y*\spacing, shadow blur radius=1.5ex}] (\width, 0) -- ++(\width, 0) -- ++(\depth*\scl, \depth) -- ++(-\width, 0) -- cycle;

        \end{scope}
    }
            
    \foreach \y in {2} {
        \begin{scope}[shift={(0, \y * \spacing)}]
            \fill[blue!60] (0, 0) -- ++(\width, 0) -- ++(0, \height) -- ++(-\width, 0) -- cycle;
            
            \fill[blue!40] (0, \height) -- ++(\width, 0) -- ++(\depth*\scl, \depth) -- ++(-\width, 0) -- cycle;

            \fill[blue!20] (\width, 0) -- ++(0, \height) -- ++(\depth*\scl, \depth) -- ++(0,-\height) -- cycle;

            \draw[thick] (0, 0) -- ++(\width, 0) -- ++(\depth*\scl, \depth) -- ++(0, \height) 
                        -- ++(-\width, 0) -- ++(-\depth*\scl, -\depth) -- cycle; 
            \draw[thick] (0, 0) -- ++(0, \height);    
            \draw[thick] (\width, 0) -- ++(0, \height); 
            \draw[thick] (0, \height) -- ++(\width, 0) -- ++(\depth*\scl, \depth); 
        \end{scope}
    }
\end{tikzpicture}}
    \subcaptionbox{\label{fig:tslice.graph} $\mat{X} \in \FFF_n^p$  - ``vector''}[.3\linewidth]
    {

\begin{tikzpicture}
    \def\width{0.35}    
    \def\height{\width}   
    \def\depth{2.3*\width}   
    \def\scl{0.8} 
    \def\spacing{.45}     
    \def\shadowx{0.25}
    \def\shadowy{0.15}

    \foreach \y in {0,1,2,3} {
        \begin{scope}[shift={( \y * \spacing,0)}]
            %
            \fill[ opacity=0.0, blur shadow={shadow blur steps=10,shadow scale = .82, shadow yshift=\y*\spacing, shadow blur radius=1.5ex}] (\width, 0) -- ++(\width, 0) -- ++(\depth*\scl, \depth) -- ++(-\width, 0) -- cycle;

        \end{scope}
    }
            
    \foreach \y in {0,1,2,3} {
        \begin{scope}[shift={(0, \y * \spacing)}]
            \fill[blue!60] (0, 0) -- ++(\width, 0) -- ++(0, \height) -- ++(-\width, 0) -- cycle;
            
            \fill[blue!40] (0, \height) -- ++(\width, 0) -- ++(\depth*\scl, \depth) -- ++(-\width, 0) -- cycle;

            \fill[blue!20] (\width, 0) -- ++(0, \height) -- ++(\depth*\scl, \depth) -- ++(0,-\height) -- cycle;

            \draw[thick] (0, 0) -- ++(\width, 0) -- ++(\depth*\scl, \depth) -- ++(0, \height) 
                        -- ++(-\width, 0) -- ++(-\depth*\scl, -\depth) -- cycle; 
            \draw[thick] (0, 0) -- ++(0, \height);    
            \draw[thick] (\width, 0) -- ++(0, \height); 
            \draw[thick] (0, \height) -- ++(\width, 0) -- ++(\depth*\scl, \depth); 
        \end{scope}
    }
\end{tikzpicture}}
    \subcaptionbox{\label{fig:ttensor.graph} $\tens{X} \in \FFF_n^{m \xx p}$  - ``matrix of tubes''}[.3\linewidth]
    {

\begin{tikzpicture}
    \def\width{0.35}    
    \def\height{\width}   
    \def\depth{2.3*\width}   
    \def\scl{0.8} 
    \def\spacing{.45}     
    
    \foreach \x in {4} {
    \begin{scope}[shift={(\x*\spacing,0)}]
        \foreach \y in {0,1,2,3} {
            \begin{scope}[shift={( \y * \spacing,0)}]
                \fill[ opacity=0.0, blur shadow={shadow blur steps=10,shadow scale = .82, shadow yshift=\y*\spacing, shadow blur radius=1.5ex}] (\width, 0) -- ++(\width, 0) -- ++(\depth*\scl, \depth) -- ++(-\width, 0) -- cycle;
            \end{scope}
        }
    \end{scope}
    }
    
    \foreach \x in {0,1,2,3,4} {
    \begin{scope}[shift={(\x*\spacing,0)}]
        \foreach \y in {0,1,2,3} {
            \begin{scope}[shift={(0, \y * \spacing)}]
                \fill[blue!60] (0, 0) -- ++(\width, 0) -- ++(0, \height) -- ++(-\width, 0) -- cycle;
                
                \fill[blue!40] (0, \height) -- ++(\width, 0) -- ++(\depth*\scl, \depth) -- ++(-\width, 0) -- cycle;

                \fill[blue!20] (\width, 0) -- ++(0, \height) -- ++(\depth*\scl, \depth) -- ++(0,-\height) -- cycle;

                \draw[thick] (0, 0) -- ++(\width, 0) -- ++(\depth*\scl, \depth) -- ++(0, \height) 
                            -- ++(-\width, 0) -- ++(-\depth*\scl, -\depth) -- cycle; 
                \draw[thick] (0, 0) -- ++(0, \height);    
                \draw[thick] (\width, 0) -- ++(0, \height); 
                \draw[thick] (0, \height) -- ++(\width, 0) -- ++(\depth*\scl, \depth); 
            \end{scope}
        }
    \end{scope}
    }
\end{tikzpicture}}
    \caption{Graphical illustration of tubal-tensor as a matrix of tubes.}
    \label{fig:illustration}
\end{figure}

Going from scalars to matrices, the tensor-tensor $\mm$-product definition adheres to the usual multiplication of matrices over a field. For $\tX \in \FFF_n^{m \xx p}$ and $\tY\in \FFF^{p \xx q \xx n}$, we have $\tZ = \tX \mm \tY \in \FFF_n^{m \xx q}$, where 
    $\tZ_{k,j,:} \coloneqq {\sum}_{\ell=1}^p \tX_{k,\ell,:} \mm \tY_{\ell,j,:}$.
Equivalently, by extension of ~\Cref{eq:F.sup.n.star.m}, if we define $\wh{\tX}\coloneqq \tX \xx_3 \matM$ and $\wh{\tY}\coloneqq \tY \xx_3 \matM$, we have 
\begin{equation}\label{eq:tens.tens.mm.fwp}
    \tX \mm \tY = \left( \thX \vartriangle \thY \right) \xx_3 \matM^{-1}
\end{equation}
where $\vartriangle$ denotes the \textbf{facewise product} of tubal-tensors $(\tA \vartriangle \tB)_{:,:,j} = \tA_{:,:,j} \tB_{:,:,j}$.

The tubal-tensor $\mm$-product thus naturally views an $m \xx p \xx n$ tensor as a linear mapping from $\FFF^{p \xx 1 \xx n}$ to $\FFF^{m \xx 1 \xx n}$.
Accordingly, the \textbf{identity $m \xx m$ tensor} $\tI$ is such that $\tI \mm \tX = \tX, \tY \mm \tI = \tY$ for all $\tX \in \FFF_n^{m \xx p}, \tY \in \FFF_n^{p \xx q}$.
A tensor $\tU \in \FFF_n^{m \xx m}$ is said to be \textbf{$\mm$-unitary} tensor if $\FNorm{\tU \mm \tX} = \FNorm{\tX}$ for all $\tX \in \FFF_n^{m \xx p}$.
Equivalently, $\tU$ is $\mm$-unitary if $\tU \mm \tU^\CT = \tU^\CT \mm \tU = \tI$ where  $\tA^\CT \in \FFF_n^{p \xx m}$ denotes the \textbf{$\mm$-conjugate transpose} of $\tA \in \FFF_n^{m \xx p}$, defined by $\widehat{\tA^\CT}_{i,j,k} \coloneqq \overline{\widehat{\tA}_{j,i,k}}$. 

The space of tubes $\CC_n$ is isomorphic to a subset of linear operators on $\CC^n$.
To see this, associate with $\bs{x} \in \CC^n$ the operator $T_{\bs{x}}:\CC_n \to \CC_n$ defined by $T_{\bs{x}} \bs{y} \coloneqq \bs{x} \mm \bs{y}$. When viewing $T_{\bs{x}}$ as a mapping from $\CC^n$ to $\CC^n$, we immediately see from the ring axioms that it is linear. 
However, due to commutativity of $\CC_n$, we get that the operator $T_{\bs{x}}$ is not only linear, but in fact a \textbf{t-linear operator}~\cite{KBHH13}, i.e., $T_{\bs{x}} (\bs{a} \mm \bs{y} + \bs{b} \mm \bs{z}) = \bs{a} \mm T_{\bs{x}} \bs{y} + \bs{b} \mm T_{\bs{x}} \bs{z}$ for all $\bs{a}, \bs{b}, \bs{y}, \bs{z} \in \CC^n$.
\begin{lemma}\label{lem:tubes.tlin.121}
    Let $T \colon \CC^n \to \CC^n$ be a t-linear operator, then there exists a unique $\bs{x} \in \CC^n$ such that $T = T_{\bs{x}}$.
    As a corollary, we have a bijection between t-linear operators and tubes in $\CC^n$.
\end{lemma}
\begin{proof}
    Let $\bs{e} = \matMi \bs{1}_n$ be the multiplicative identity in $(\CC^n, \mm)$, and $\bs{x} = T \bs{e}$.
    Denote by $T_{\bs{x}}$ the operator associated with $\bs{x}$, i.e., $T_{\bs{x}} \bs{y} = \bs{x} \mm \bs{y}$.
    Since both $T$ and $T_{\bs{x}}$ are t-linear, the mapping $T - T_{\bs{x}}$ is t-linear as well.
    For all $\bs{y} \in \CC^n$, we have $(T - T_{\bs{x}}) \bs{y} = (T - T_{\bs{x}}) (\bs{e} \mm \bs{y}) $. 
    By t-linearity, $(T - T_{\bs{x}}) (\bs{e} \mm \bs{y}) = ((T - T_{\bs{x}}) \bs{e}) \mm \bs{y}$, and note that $(T - T_{\bs{x}}) \bs{e} = T \bs{e} - T_{\bs{x}} \bs{e}  = \bs{x} - \bs{x}\mm \bs{e} = 0$, thus $(T - T_{\bs{x}}) \bs{y} = 0$ for all $\bs{y} \in \CC^n$ and we have $T = T_{\bs{x}}$.
    Since $T_{\bs{x}}$ is uniquely determined by $\bs{x}$, the result follows.
\end{proof}
Importantly, the following structure emerges.
\begin{proposition}[{\cite[Proposition 4.4]{Kernfeld2015}}]\label{prop:mm.tubal.algebra}
    Let $\matM \in \FFF^{n \xx n}$ be an invertible matrix, and $\mm$ as in~\Cref{eq:F.sup.n.star.m},
    then $(\FFF^n, \mm)$ augmented with the $\mm$-conjugate transpose operation and $\GNorm{\bs{x}}_{\mm {-} max} \coloneqq \max_{d} |x_d|$ is a commutative, unital, C*-algebra over $\FFF$.
\end{proposition}

\subsubsection{Tubal Singular Value Decomposition}\label{sec:tSVD.mm}
Perhaps the most important result in the tubal-tensor framework is the tubal existence of a singular value decomposition (tSVDM) of a tensor $\tX \in \FFF^{m \xx p \xx n}~$\cite{KilmerPNAS,KBHH13}, which is a factorization of the form
\begin{equation}\label{eq:tSVDM}
    \tX = \tU \mm \tS \mm \tV^\CT
\end{equation}
where $\tU \in \FFF^{m \xx m \xx n}$ and $\tV \in \FFF^{p \xx p \xx n}$ are $\mm$-unitary tensors, and $\tS \in \FFF^{m \xx p \xx n}$ is an f-diagonal tensor, i.e., $\tS_{i,j,:} = 0$ for $i \neq j$.
The significance of the tSVDM is that it is the first, and so far the only, tensor factorization that entertains an Eckart-Young-Mirsky type results for tensors, with respect to specific tensor-rank definitions. 

\paragraph{Tensor-ranks and Truncations.}
Definitions of tubal-tensor ranks are stated for order-3 tensors; generalization to higher order is straightforward.
Let $\smash{\tX = \tU \mm \tS \mm \tV^\CT \in \FFF^{m \xx p \xx n}}$. 
The \textbf{t-rank} of $\tX$ under $\mm$, is the number of non-zero tubes in $\tS$.
The \textbf{multi-rank} of $\tX$  under $\mm$ is a vector $\rrho \in \NN^n$ where $\smash{\rho_j = \rank(\thX_{:,:,j})}$ for $j \in [n]$.
Let $\rrho$ be a vector such that $\rho_j \leq \min(m,p)$ for all $j \in [n]$, then the \textbf{multi-rank $\rrho$ of truncation} of $\tX$ is the tensor $[\tX]_{\rrho}$ such that 
$$\wh{[\tX]_{\rrho}}_{:,:,j} \coloneqq \wh{\tU}_{:,1:\rho_j,j} \, \wh{\tS}_{1:\rho_j,1:\rho_j,j} \, {\wh{\tV}_{:,1:\rho_j,j}}^\CT$$
for $j \in [n]$.
Consequently the \textbf{t-rank $r$ truncation} of $\tX$ is the tensor $[\tX]_r$ obtained by a multi-rank $\rrho$ truncation  of $\tX$ under $\mm$ where $\rho_j = r$ for all $j \in [n]$.

\paragraph{Best-Rank Approximation.}
Suppose that $\matM \in \FFF^{n \xx n}$ is a nonzero multiple of a unitary matrix and let $\tX \in \FFF^{m \xx p \xx n}$.
Then for all $\tY \in \FFF^{m \xx p \xx n}$ with t-rank $r$ (respectively, multi-rank $\rrho$) under $\mm$, we have that $\FNorm{\tX - \tY} \geq \FNorm{\tX - [\tX]_r}$ (respectively, $\FNorm{\tX - \tY} \geq \FNorm{\tX - [\tX]_{\rrho}}$)~\cite{KBHH13,Kernfeld2015,KilmerPNAS}.
Therefore, the best low t-rank (multi-rank) approximation to $\tX$ under $\mm$ is $[\tX]_r$ ($[\tX]_{\rrho}$).

\section{Infinite Dimensional Tubes and Quasitubes}\label{sec:infinite.tubes}

\subsection{From Finite Dimensional to Infinite Dimensional Tubes}

Given an infinite dimensional linear space $\cH$, we denote by $\tX \in \cH^{m \xx p}$ a $m \xx p$ tubal matrix over $\cH$, i.e., a mapping from index pairs to elements in $\cH$. 
We impose the additional requirement that $\cH$ is a separable Hilbert space over $\FFF$. 
With this requirement, the tubal product can be naturally extended to infinite dimensional Hilbert spaces by replacing the matrix $\matM$ in~\Cref{eq:F.sup.n.star.m} with a bounded, invertible linear operator $F \colon \cH \to \ell_2$.
The separability assumption guarantees the existence of such an operator $F$ and even allows us to assume that $F$ is an isometric isomorphism.
Indeed,~\Cref{def:F-transform} below of the $F$-transform is restricted to isometric isomorphisms, and all the results in this work are derived under this assumption. 
Note that unlike functional tensors~\cite{Han2023} and  quasitensors~\cite{LARSEN2024RKHS}, our theory does not require $\cH$ to be a RKHS nor even a space of functions, but rather it can be any separable Hilbert space over $\FFF$.
\begin{definition}[$F$-transform]\label{def:F-transform}
    Let $\smash{F = \{ \bs{\phi}^{(j)} \}_{j \in \ZZ}}$ be an orthonormal basis of a separable, infinite dimensional Hilbert space $\cH$ over $\FFF$. 
    Define the \textbf{$F$-transform} of $\bs{x} \in \cH$ as the sequence $\smash{F(\bs{x}) = \widehat{\bs{x}}}$ where $\smash{\wh{x}_j = \dotph{\bs{x}, \bs{\phi}^{(j)}}}$.
    Denote by $F(\cH) \coloneqq \{ F (\bs{x}) \in \FFF^\ZZ ~|~  \bs{x} \in \cH \} $ the \textbf{transform domain} of $\cH$ under $F$, and note that $\smash{\ell_2 = F(\cH)}$.
    The \textbf{inverse} $F$-transform of a sequence $\wh{\bs{x}} \in \ell_2$ is  $F^{-1}(\wh{\bs{x}}) = \sum_{j \in \ZZ} \wh{x}_j \bs{\phi}^{(j)}$.
    Since $\wh{x}_j = \dotph{\bs{x}, \bs{\phi}^{(j)} }$ for all $j \in \ZZ$, we get that $F^{-1}(\hbs{x}) = \sum_{j \in \ZZ} \dotph{\bs{x}, \bs{\phi}^{(j)} } \bs{\phi}^{(j)} = \bs{x}$.
    
\end{definition}

Overloading the symbol $F$ to denote both the basis of $\cH$ and the operator $F \colon \cH \to \ell_2$ is justified by the fact that an orthonormal basis $\{ \bs{\phi}^{(j)} \}_{j\in\ZZ}$ uniquely determines an isometric isomorphism between $\cH$ and $\smash{\ell_2}$.
Given $\{ \bs{\phi}^{(j)} \}_{j\in\ZZ}$, the mapping $\bs{\phi}^{(j)} \mapsto \bs{e}^{(j)}=F(\bs{\phi}^{(j)})$, where $\bs{e}^{(j)}$ is a sequence which is $1$ at position $j$ and zero elsewhere, is an isometric isomorphism. Conversely, given an isometric isomorphism $T \colon \cH \to \ell_2$, the basis $\{ T^{-1}(\bs{e}^{(j)}) \}_{j\in\ZZ}$ is an orthonormal basis of $\cH$.

Continuing the analogy with the finite dimensional case, the $F$ transform of a tubal tensor $\tX$ over $\cH$ is a tubal tensor $\thX = \tX \xx_3 F$ over $\ell_2$ defined by $\thX_{j,k} \coloneqq \wh{\tX_{jk}}$, and  we define the $\ff$-product for $\cH$ tubes as 
\begin{center}
\begin{minipage}{.5\linewidth}
    \begin{equation}\label{def:ff.qtm}
        \bs{x} \ff \bs{y} \coloneqq F^{-1}(\wh{\bs{x}} \odot \wh{\bs{y}}) = {\sum}_{j \in \ZZ} \wh{x}_j \wh{y}_j \bs{\phi}^{(j)}
    \end{equation}
\end{minipage}%
\begin{minipage}{.5\linewidth}
    \begin{equation}\label{def:ff.tensor}
        \tX \ff \tY \coloneqq  (\thX \vartriangle \thY) \xx_3 F^{-1}
    \end{equation}
\end{minipage}
\end{center}
where $\odot$ denotes element-wise product of two sequences. 
Note that the expression in \Cref{def:ff.qtm} defines an element in $\cH$, since element-wise product of two $\ell_2$ sequences is also square summable, hence, $\bs{z} = \sum_{j \in \ZZ} \wh{x}_j \wh{y}_j \bs{\phi}^{(j)}$ is indeed in $\cH$ as it is the limit of the sequence $\bs{z}^{(J)} = \sum_{j = -J}^{J} \wh{x}_j \wh{y}_j \bs{\phi}^{(j)}$, which is a Cauchy sequence in a complete vector space. 

While all separable Hilbert spaces are isomorphic to $\ell_2$, the choice of orthonormal basis $F$ determines the structure of the algebra on $\cH$.
To see this, consider the orthonormal basis $\smash{F = \{ \bs{\phi}_j \}_{j \in \ZZ}}$ of $\cH$, 
and set $\smash{G = \{\bs{\gamma}_j\}_{j\in \ZZ}}$ 
where 
$\bs{\gamma}_1 = \smash{(\bs{\phi}_1 + \bs{\phi}_2)/\sqrt{2}}, \smash{\bs{\gamma}_2 = (\bs{\phi}_1 - \bs{\phi}_2)/\sqrt{2}}$, 
and $\bs{\gamma}_j = \bs{\phi}_j$ for all integers $j \neq 1,2$.
Then $G$ is an orthonormal basis of $\cH$ and the corresponding $G$-transform is an isometry $\smash{G \colon \cH \to  \ell_2}$.
However,
\begin{align*}
    \bs{\phi}_1 \ttprod{G} \bs{\phi}_2 
    &= 1/2 G^{-1}  \left( (  \bs{e}_1 - \bs{e}_2  ) \odot (  \bs{e}_1 + \bs{e}_2 ) \right) \\
    &= 1/2 G^{-1} \left(  \bs{e}_1 - \bs{e}_2  \right) 
    = 1/2 (\bs{\gamma}_1 - \bs{\gamma}_2) \\
    &\neq 0 = F^{-1} \left(  \bs{e}_1 \odot \bs{e}_2 \right) = \bs{\phi}_1 \ff \bs{\phi}_2
\end{align*}

Two important properties of~\Cref{def:ff.qtm} is that the basis elements are idempotent, i.e., $\bs{\phi}^{(j)} \ff \bs{\phi}^{(j)}=\bs{\phi}^{(j)}$ for all $j\in\ZZ$, and the following identity:
$$
\bs{x} = \sum_{j \in \ZZ} \bs{\phi}^{(j)} \ff \bs{x}
$$

So far, the extension from finite tubes to infinite tubes seems trivial and not worthy of much discussion. However, a complication arises: \Cref{lem:tubes.tlin.121} no longer holds once the tubes are over infinite dimensional $\cH$.  This is due to the following \Cref{lem:cH.ffops.are.HS}. 
An important corollary of \Cref{lem:cH.ffops.are.HS} is that the resulting algebra is not unital, unlike the finite case.  The existence of $\ff$-unitary tensors depends on that of a multiplicative identity in $\cH$, thus a factorization such as~\Cref{eq:tSVDM} cannot exist for tubal tensors over infinite dimensional $\cH$.

\begin{lemma}\label{lem:cH.ffops.are.HS}
    Let $\cH$ be a separable Hilbert space
    equipped with the $\ff$-product  (\Cref{def:ff.qtm}) defined for $F=\{\bs{\phi}_j\}_{j\in\ZZ}$ as in~\Cref{def:F-transform}.
    Then, for any $\bs{x} \in \cH$, the associated operator $T_{\bs{x}}$ defined by $T_{\bs{x}}(\bs{y}) \coloneqq \bs{x}\ff \bs{y}$ for all $\bs{y} \in \cH$, is a Hilbert-Schmidt operator and $\HSNorm{T_{\bs{x}}} = \cHNorm{\bs{x}}$.

    As a corollary, if $\cH$ is infinite dimensional, then there cannot be an element $\bs{e} \in \cH$ such that $\bs{e} \ff \bs{x} = \bs{x} $ for all $\bs{x} \in \cH$. 
\end{lemma}
\begin{proof}
    For any integer $j$, we have $\smash{T_{\bs{x}} (\bs{\phi}^{(j)}) = F^{-1} (\wh{\bs{x}} \odot \bs{e}^{(j)}) = F^{-1} (\wh{x}_j \bs{e}^{(j)}) = \wh{x}_j \bs{\phi}^{(j)}}$.
    So $\cHNormS{T_{\bs{x}}( \bs{\phi}^{(j)})} = |\wh{x}_j|^2 $ 
    and $\HSNorm{T_{\bs{x}}}^2 
    =\sum_{j\in\ZZ} \cHNormS{T_{\bs{x}}(\bs{\phi}^{(j)})}
    = \sum_{j\in\ZZ} |\wh{x}_j|^2 
    = \dotph{F^{-1}\wh{\bs{x}}, F^{-1}\wh{\bs{x}}}
    = \cHNormS{\bs{x}}< \infty$, so $T_{\bs{x}}$ is Hilbert-Schmidt.

    As for the corollary, suppose that $\cH$ contained a unity $\bs{e}$ then the associated operator $T_{\bs{e}} $ would have been the identity operator, which, by the above, is a Hilbert-Schmidt operator. However, the identity operator is not Hilbert-Schmidt for any infinite dimensional Hilbert space.
\end{proof}

In the next subsection we introduce the notion of quasitubes to address the lack of multiplicative identity, and allow a infinite dimensional analog of tubal SVD. To that end, we appeal to the notion of t-linearity from finite dimensional tubal algebra, and use it as a basis for defining quasitubes. Before doing so, we first introduce the notion of $\cH$-linearity. 

We say that an operator $T:\cH\to\cH$ is $\cH$-linear with respect to $\ff$ if $T(\bs{a}\ff \bs{x}+\bs{b}\ff \bs{y}) = \bs{a} \ff  T(\bs{x}) + \bs{b} \ff T(\bs{y})$ for all $\bs{a},\bs{b},\bs{x},\bs{y}\in\cH$. The notion of ``$\cH$-linear'' replaces ``t-linear'' from the finite case, since for the infinite case it is important to specify the Hilbert space from which tubes are drawn.
Together with the $\ff$-product, the space $\cH$ forms a commutative algebra over $\FFF$, and similarly to the finite dimensional case, each $\bs{x} \in \cH$ is associated with a $\cH$-linear operator $T_{\bs{x}}$ defined by $T_{\bs{x}}(\bs{y}) \coloneqq \bs{x} \ff \bs{y}$.  

For finite dimensional tubes, it is easy to verify that every t-linear operator is also simply linear (with respect to $\FFF$). A one line argument based on the multiplicative identity $\bs{e}$ establishes it: for every t-linear operator $T$, $\bs{x}\in\FFF_n$ and $\alpha\in\CC$ we have $T(\alpha\bs{x})=T(\alpha(\bs{e}\mm\bs{x})) = T((\alpha\bs{e}) \mm \bs{x}) = (\alpha\bs{e})\mm T(\bs{x}) = \alpha T(\bs{x})$. However, the lack of multiplicative identity for infinite dimensional tubes makes that corresponding assertion that $\cH$-linearity implies $\FFF$-linearity much less trivial. Nevertheless, the following lemma shows that it still holds.
\begin{lemma}\label{lem:cHs.are.F.linear}
    An $\cH$-linear operator $\bs{q} \colon \cH \to \cH$ is $\FFF$-linear, i.e., for all $\bs{x},\bs{y} \in\cH$ and $\alpha,\beta \in \FFF$ we have $\bs{q}  (\alpha \bs{x} + \beta \bs{y}) = \alpha \bs{q} ( \bs{x}) + \beta \bs{q}( \bs{y})$.
\end{lemma}
Before proving~\Cref{lem:cHs.are.F.linear}, we need a couple of technical lemmas.
\begin{lemma}\label{lem:q.alphaphij.alpha.qphij}
    Let $\bs{q} $ be a $\cH$-linear operator. For any basis element $\bs{\phi}^{(j)} \in F$ and all $\alpha \in \FFF$ we have $\bs{q}(\alpha \bs{\phi}^{(j)}) = \alpha \bs{q}( \bs{\phi}^{(j)})$.
\end{lemma}
\begin{proof}
    Let $\bs{\phi}^{(j)} \in F$ be a basis element and write $\bs{q}^{(j)} \coloneqq  \bs{q} (\bs{\phi}^{(j)}) \in \cH$. Denote $\wh{\bs{q}}^{(j)} \coloneqq F(\bs{q}^{(j)})$. Then by \Cref{def:ff.qtm}:
    \begin{align*}
        \bs{\phi}^{(j)} \ff \bs{q}^{(j)} 
        &= \sum_{k \in \ZZ} e^{(j)}_k \wh{q}^{(j)}_k \bs{\phi}^{(k)} 
        = \wh{q}^{(j)}_j \bs{\phi}^{(j)}
    \end{align*}
    Recalling that basis elements are idempotent, we have 
    $\bs{q}^{(j)} = \bs{q}(\bs{\phi}^{(j)}) = \bs{q}(\bs{\phi}^{(j)} \ff \bs{\phi}^{(j)}) =\bs{\phi}^{(j)} \ff \bs{q}^{(j)} = \wh{q}^{(j)}_j \bs{\phi}^{(j)}$ with $\wh{q}^{(j)}_j = \dotph{\bs{q}^{(j)}, \bs{\phi}^{(j)}}$. 
    By~\Cref{def:ff.qtm}, in addition to idempotentcy, we have $(\alpha \bs{\phi}^{(j)}) \ff \bs{\phi}^{(j)}= \alpha \bs{\phi}^{(j)} =  \alpha (\bs{\phi}^{(j)} \ff \bs{\phi}^{(j)})$, so
    \begin{align*}
        \bs{q} (\alpha \bs{\phi}^{(j)}) 
        &=  \bs{q} ((\alpha \bs{\phi}^{(j)}) \ff \bs{\phi}^{(j)} )  
        =  (\alpha \bs{\phi}^{(j)}) \ff \bs{q} (  \bs{\phi}^{(j)})   
        = (\alpha \bs{\phi}^{(j)}) \ff \bs{q}^{(j)} 
        = (\alpha \bs{\phi}^{(j)}) \ff (\wh{q}^{(j)}_j \bs{\phi}^{(j)}) ~.
    \end{align*}
    and by \Cref{def:ff.qtm} we obtain that $(\alpha \bs{\phi}^{(j)}) \ff (\wh{q}^{(j)}_j \bs{\phi}^{(j)}) = \alpha \wh{q}^{(j)}_j \bs{\phi}_j = \alpha \bs{q}^{(j)} = \alpha \bs{q}(\bs{\phi}^{(j)})$.
\end{proof}
\begin{lemma}\label{lem:linearity.wrt.expansion}
    Let $\bs{q}$ be an $\cH$-linear operator, then for any $\bs{x} = \sum_{j \in \ZZ} \wh{x}_j \bs{\phi}^{(j)} \in \cH$ we have $\bs{q}(\bs{x}) = \sum_{j \in \ZZ} \wh{x}_j \bs{q}(\bs{\phi}^{(j)})$.
    As a corollary, for any basis element $\bs{\phi}^{(j)} \in F$, it holds that $\bs{q} (\bs{\phi}^{(j)}) = \wh{q}_j \bs{\phi}^{(j)}$ for some $\wh{q}_j \in \FFF$. 
\end{lemma}
\begin{proof}
    By~\Cref{def:ff.qtm}, $\bs{\phi}^{(k)} \ff \bs{x} = \sum_{j \in \ZZ}  e^{(k)}_j \wh{x}_j \bs{\phi}^{(j)} = \wh{x}_k \bs{\phi}^{(k)}$.
    Next, write $\bs{y} = \bs{q} (\bs{x}) = \sum_{j\in \ZZ} \wh{y}_j \bs{\phi}^{(j)}$, then 
    \begin{align*}
        \wh{y}_k \bs{\phi}^{(k)} 
        &= \bs{\phi}^{(k)} \ff \bs{y} 
        = \bs{\phi}^{(k)} \ff \bs{q} (\bs{x}) \\
        &=  \bs{q} (\bs{\phi}^{(k)} \ff \bs{x}) 
        = \bs{q}( \wh{x}_k \bs{\phi}^{(k)})
    \end{align*}
    Apply~\Cref{lem:q.alphaphij.alpha.qphij} on the last expression and get that $\wh{y}_k \bs{\phi}^{(k)} = \wh{x}_k \bs{q}( \bs{\phi}^{(k)})$, which, when plugged in to the series expressing $\bs{y}$, results in:
    \begin{align*}
        \bs{y} = {\sum}_{j\in \ZZ} \wh{y}_j \bs{\phi}^{(j)}
        &= {\sum}_{j\in \ZZ} \wh{x}_j \bs{q}( \bs{\phi}^{(j)})
    \end{align*}
    Now recall that $\bs{y} = \bs{q} (\bs{x})$ to conclude the proof.
    
    For the corollary, let $\bs{\phi}^{(j)} \in F$ be a basis element in $\cH$. We write, $q(\bs{\phi}^{(j)})=\sum_{k\in \ZZ} \hat{\bs{\phi}}_k^{(j)} q(\bs{\phi}^{(k)})$. Now, $\hat{\bs{\phi}}_k^{(j)}$ is the inner product between $\bs{\phi}^{(j)}$ and $\bs{\phi}^{(k)}$, which is 1 for $k$ equal to $j$, and 0 otherwise, proving the corollary.
\end{proof}

    

\begin{proof}[Proof of~\Cref{lem:cHs.are.F.linear}]
    Since $\bs{q}$ is $\cH$-linear, and  $\alpha \bs{\x} $ and $\beta \bs{y} $ are both in $\cH$, we have $\bs{q} (\alpha \bs{x} + \beta \bs{y}) = \bs{q} (\alpha \bs{x}) +  \bs{q} (\beta \bs{y})$.
    Therefore, it is enough to show that $\bs{q}  (\alpha \bs{x}) = \alpha \bs{q} (\bs{x})$ for all $\bs{x} \in \cH$ and $\alpha \in \FFF$.

    For any $\bs{x} \in \cH, \alpha \in \FFF$, write $\alpha \bs{x} = \sum_{j \in \ZZ} (\alpha \wh{x}_j) \bs{\phi}^{(j)}$.
    Apply~\Cref{lem:linearity.wrt.expansion} to $\bs{q}(\alpha \bs{x})$ and get that 
    \begin{align*}
        \bs{q}(\alpha \bs{x})
        &= \bs{q} \left( {\sum}_{j \in \ZZ} (\alpha \wh{x}_j) \bs{\phi}^{(j)} \right) 
        = {\sum}_{j \in \ZZ} (\alpha \wh{x}_j)\bs{q}(\bs{\phi}^{(j)}) 
        = \alpha  {\sum}_{j \in \ZZ} \wh{x}_j \bs{q}(\bs{\phi}^{(j)})
    \end{align*}
    Apply~\Cref{lem:linearity.wrt.expansion} once again and have that $\alpha  \sum_{j \in \ZZ} \wh{x}_j \bs{q}(\bs{\phi}^{(j)}) {=} \alpha \bs{q}(\bs{x})$ thus we conclude that 
    $\bs{q}(\alpha \bs{x}) {=} \alpha \bs{q}(\bs{x})$.
\end{proof}

\subsection{Quasitubes}
\label{subsec:quasitubes}

To allow a tubal SVD over infinite dimensional $\cH$, we introduce the notion of \textbf{quaistubes} as an extension of $(\cH, \ff)$ to a commutative unital C*-algebra. The main idea is as follows.  Recall that for any $\bs{x} \in \cH$ we can write $\bs{x} = \sum_{j\in\ZZ} \bs{\phi}^{(j)}\ff \bs{x}$. Now, if $T$ is a $\cH$-linear operator, then for any $\bs{x} \in \cH$ we have 
$$T (\bs{x}) = T \left(\sum_{j\in\ZZ} \bs{\phi}^{(j)} \ff \bs{x} \right) = \left(\sum_{j\in\ZZ} T (\bs{\phi}^{(j)})\right) \ff \bs{x},$$ 
where the last equality holds only if $\sum_{j\in\ZZ} T(\bs{\phi}^{(j)})$ converges  in $\cH$. 
Even when it does not converge, we can express $T(\bs{x})$ as the $\ff$-product of $\bs{x}$ with a certain ``object'', which we call {\em quasitubes}. 
A formal definition of quasitubes is given in~\Cref{def:quasitube}.
First, we need to define {\em dual modules}.
\begin{definition}[Dual Module]\label{def:dual.module}
    Let $\alg{E}$ be a module over a ring $\alg{A}$,
    The dual module of $\alg{E}$ is the set $\alg{E}_* = \Hom_{\alg{A}}(\alg{E}, \alg{A})$ of module homomorphisms from $\alg{E}$ to $\alg{A}$, where $\alg{A}$ is considered as a module over itself.
    Clearly $\alg{E}_*$ is an $\alg{A}$-module as well.
\end{definition}

Consider $\cHs$, the set of module homomorphisms from $\cH$ to $\cH$, where we consider $\cH$ as a $\cH$-module with respect to $\ff$. Any $\bs{q}\in \cHs$, being a module homomorphism, respects the module structure, so it must be a $\cH$-linear operator. 
The following lemma shows that any $\cH$-linear operator is also bounded.
\begin{lemma}[Quasitubes are Bounded Linear Operators]\label{lem:cHs.subset.BcH}
    Let $\bs{q} \in \cHs$. Then $\bs{q}$ is bounded linear operator. 
    More generally, $\cHs \subseteq B(\cH)$ where $B(\cH)$ denotes the space of bounded $\cH \to \cH$ linear operators. 
\end{lemma}
\begin{proof}
    Let $\bs{q} \in \cHs$ then $\bs{q} $ is ($\FFF$) linear by \Cref{lem:cHs.are.F.linear}. 

    Recall that $\bs{q} \in \cHs$ is an homomorphism, therefore $\bs{q}(\bs{x}) \in \cH$ for all $\bs{x} \in \cH$.
    By \Cref{lem:linearity.wrt.expansion}, we have $\bs{q}(\bs{\phi}^{(j)} ) = \wh{q}_{j} \bs{\phi}^{(j)}$ for all $j \in \ZZ$. 
    Suppose $\sup_{j \in \ZZ} |\wh{q}_{j}|$ does not exist, then we can assume without loss of generality that $|\wh{q}_{j}| > |j|^2 $ (otherwise take a subsequence $j_n$ such that $|\wh{q}_{j_n}| > |n|^2 $).
    Since the sequence $\{n^{-1} \}_{n=1}^\infty$ is square summable, we have that  $\bs{x} = \sum_{n=1}^\infty n^{-1} \bs{\phi}^{(n)}$ is an element of $\cH$. 
    By applying the corollary from
    \Cref{lem:linearity.wrt.expansion} to $\bs{q} (\bs{x})$ we get that 
    $$
    \cHNormS{\bs{q} (\bs{x})} = \sum_{n=1}^\infty |\wh{q}_n|^2 |\wh{x}_n|^2 \geq \sum_{n=1}^\infty n^2
    $$
    meaning that $\bs{q} (\bs{x}) \notin \cH$ in contradiction to $\bs{q}$ being a homomorphism. 
    Hence, $\sup_{j \in \ZZ} |\wh{q}_{j}| $ exists. 
    Boundedness immediately follows since $\cHNorm{\bs{q} (\bs{x})} \leq \cHNorm{\bs{x}} \sup_{j \in \ZZ} |\wh{q}_{j}| $.

\end{proof}

The set $\cHs$ may be naturally considered as a subspace of $B(\cH)$. 
Since $\cHs$ is closed to composition of operators, we get that $\cHs$ is a commutative, unital subalgebra of $B(\cH)$. Thus, we define quasitubes as bounded $\cH$-linear operators.
\begin{definition}[Quasitubes]\label{def:quasitube}
    Let $\cH$ be a separable Hilbert space, and $\ff$ as defined in \Cref{def:F-transform} and \Cref{def:ff.qtm}.
    A $\cH$-quasitube, or simply a quasitube, is an element $\bs{q} \in \cHs$, that is a bounded $\cH$-linear operator with respect to $\ff$.
    We endow quasitubes with the norm $\cHsNorm{\bs{q}}\coloneqq \sup_{\cHNorm{\bs{x}}=1}\cHNorm{\bs{q}(\bs{x})}$.

    As any quasitube $\bs{q} \in \cHs$ is a homomorphism, it is required that the domain of $\bs{q}$ is the whole space $\cH$, i.e., $\bs{q} (\bs{x}) \in \cH$ for all $\bs{q} \in \cHs$ and $\bs{x} \in \cH$.
\end{definition}

Every quasitube $\bs{q} \in \cHs$ is a bounded $\cH$-linear operator, and in particular a linear operator, so it has a unique adjoint operator with respect inner product in $\cH$, i.e., an operator $\bs{p}$ such that for every $\bs{x},\bs{y}\in\cH$ we have $\dotph{\bs{q}(\bs{x}),\bs{y}} = \dotph{\bs{x},\bs{p}(\bs{y})}$. This operator is also in $\cHs$ (this is proven in the next section). We define the conjugate of $\bs{q}$, denoted by $\bs{q}^*$, as this operator. 

\begin{remark}
In our construction, we required $F$ to be an orthonormal basis, while finite tubal algebra is defined of a matrix $\matM$ which is not necessarily orthonormal. We required $F$ to be orthnormal so that there will be meaningful relation between the adjoint of  $T_{\bs{x}}$ and the conjugate of a tube. Since $\cH$ is a Hilbert space, the adjoint to an operator $T \colon \cH \to \cH$ is defined as the operator $T^* \colon \cH \to \cH$ such that $\dotph{T \bs{y}, \bs{z}} = \dotph{\bs{y}, T^* \bs{z}}$ for all $\bs{y}, \bs{z} \in \cH$.
In the finite-dimensional case, we have that $T_{\bs{x}} = \matMi \operatorname{diag} (\wh{\bs{x}}) \matM$ and thus $T_{\bs{x}}^* = T_{\bs{x}}^\CT =  \matM^\CT \operatorname{diag} (\wh{\bs{x}})^H {\matMi}^\CT $.
Observe that $T_{\bs{x}}^\CT$ is t-linear if an only if it commutes with $T_{\bs{x}}$ which is the case if and only if $\matM \matM^\CT = \matM^\CT \matM $ is a non-zero multiple of the identity matrix.
Thus, the orthogonality assumption on the basis $F$ is crucial for having $T_{\bs{x}^\CT} = T_{\bs{x}}^\CT$, and have that $\dotph{\bs{x} \ff \bs{y}, \bs{z}} = \dotph{\bs{y}, \bs{x}^\CT \ff \bs{z}}$ for all $\bs{x}, \bs{y}, \bs{z} \in \cH$.
\end{remark}

\begin{remark}
Seemingly, the set of quasitubes depends on the particular choice of $F$. However, we later see that the set of quasitubes is the same regardless of the choice of $F$, hence our decision to refrain from making the dependence on $F$ explicit in the notation $\cHs$.
\end{remark}

\section{Quasitubal Algebra}\label{sec:quasitubal.algebra}
Before delving into the topic of this section, we need to recall some basic terms from abstract algebra. In general, a \textbf{*-ring} is a ring $\alg{R}$ with additional `star-operation' $\bs{a} \mapsto \bs{a}^*$ that is an involutive antiautomorphism ($(\bs{x} + \bs{y})^* = \bs{x}^* + \bs{y}^*$,$(\bs{xy})^* = \bs{y}^* \bs{x}^* $,$\bs{e}^* = \bs{e}$, ${\bs{x}^*}^* = \bs{x}$). 
A \textbf{*-algebra} $\alg{A}$ over $\FFF$,  is a *-ring with star operation $\bs{x} \mapsto \bs{x}^*$ such that $(\alpha \bs{x})^* = \bar{\alpha} \bs{x}^*$ for all $\bs{x} \in \alg{A} $ and $\alpha \in \FFF$, where the complex conjugation $\alpha \mapsto \bar{\alpha}$ defines a *-ring structure over $\FFF$. 
A subset $A \subseteq \alg{A}$ of a *-algebra $\alg{A}$  is called \textbf{*-closed} (or *-invariant) if $\bs{a}^* \in A$ for all $\bs{a} \in A$. 
An element $\bs{a}$ in a *-algebra $\alg{A}$ is called \textbf{Hermitian} if $\bs{a}^* = \bs{a}$.

Let $\alg{A}, \alg{B}$ be *-algebras over $\FFF$ with products and *-operations respectively denoted  by $\cdot,\bs{a} \mapsto \bs{a}^*$ for $\alg{A}$ and $\odot,\bs{b} \mapsto \bs{b}'$ for $\alg{B}$.
A mapping $f\colon \alg{A} \to \alg{B}$ is called a \textbf{*-morphism} (star-morphism) if $f(\bs{a}^*) = f(\bs{a})'$.
A \textbf{*-homomorphism} $f\colon \alg{A} \to \alg{B}$ is a *-morphism such that $f(\bs{a} \odot \bs{b}) = f(\bs{a})\odot f(\bs{b})$ and following the properties of *-operations, we have $f((\bs{a})\odot f(\bs{b})^*) = f(\bs{b})'\odot f(\bs{a})'$.
A *-homomorphism $f\colon \alg{A} \to \alg{B}$ is  \textbf{*-isomorphism} if and only if $f$ is bijective, i.e., there exists a *-homomorphism $g\colon \alg{B} \to \alg{A}$ such that $f(g(\bs{b})) = \bs{b}$ and $g(f(\bs{a})) = \bs{a}$ for all $\bs{a} \in \alg{A}$ and $\bs{b} \in \alg{B}$.
A *-isomorphism $f\colon \alg{A} \to \alg{B}$ between normed algebras is said to be \textbf{isometric} if $\GNorm{\bs{a}}_{\alg{A}} = \GNorm{f(\bs{a})}_{\alg{B}}$ for all $\bs{a} \in \alg{A}$ where $\GNorm{\cdot}_{\alg{A}},\GNorm{\cdot}_{\alg{B}}$ are the norms defined on $\alg{A}$ and $\alg{B}$ respectively.

$\alg{A}$ is a \textbf{Banach algebra} over $\FFF$, if $\alg{A}$ is an associative algebra,  equipped with a norm $\GNorm{\cdot}$ such that $(\alg{A}, \GNorm{\cdot})$ is a complete normed vector space and $\GNorm{\bs{a} \bs{b}} \leq \GNorm{\bs{a}}\GNorm{\bs{b}}$. 
An algebra $\alg{A}$ is a \textbf{C*-algebra} if $\alg{A}$ is both a star and Banach algebra, with a norm such that the \textbf{C*-identity} holds: $\GNorm{\bs{x}^* \bs{x}} = \GNorm{\bs{x}^*} \GNorm{\bs{x}} = \GNormS{\bs{x}}$.

Let $\cH$ be a Hilbert space over $\FFF$,
and consider the algebra  $B(\cH)$ of bounded linear operators defined on $\cH$, together with addition and multiplication operations defined by the usual addition and operator composition respectively.
Further equipping $B(\cH)$ with the operator norm and a *-operation defined by taking the adjoint of an operator, makes it a C*-algebra.
The Gelfand-Naimark theorem (GNT) states that any C*-algebra $\alg{A}$ is *-isomorphic to a norm-closed, *-closed  subalgebra of $B(\cal{G})$ for a suitable choice of Hilbert space $\cal{G}$.
In a sense, $B(\cH)$ is the canonical example of a C*-algebra.

Let $\bs{a}$ be an element in a C*-algebra  $\alg{A}$.  In light of GNT, we call $\bs{a}^*$ the \textbf{adjoint} element to $\bs{a}$ (due to the *-correspondence with some C*-algebra of operators).
Correspondingly, Hermitian elements in C*-algebras are called \textbf{self-adjoint}.

The representation of $\cH$ elements in the transform domain $\ell_2$ (\Cref{def:F-transform}) motivates us to seek a correspondence between $\cHs$ and some space of sequences, in which $\ell_2$  can be embedded.
The following construction and results add up to the conclusion, stated in~\Cref{lem:cHs.star.alg}, that $(\cHs, \circ)$ (where $\circ$ denotes composition) is a commutative, unital C*-algebra over $\FFF$, that is isometrically *-isomorphic to $(\ell_\infty, \odot)$.

Consider the mapping 
\begin{equation}\label{eq:pi.chs.linf}
    \uppi({\bs{q}}) \coloneqq \sum_{j \in \ZZ} \dotph{\bs{q}(\bs{\phi}^{(j)}), \bs{\phi}^{(j)}} \bs{e}^{(j)} = \sum_{j \in \ZZ} \wh{q}^{(j)}_j \bs{e}^{(j)}\quad \text{for all } \bs{q} \in \cHs
\end{equation}
where $\bs{e}^{(j)}$ is the $j$th element in the standard basis of $\ell_\infty$, and $\bs{q}^{(j)} \coloneqq \bs{q}(\bs{\phi}^{(j)}_j)$.
We have already shown in the proof of \Cref{lem:cHs.subset.BcH} that the sequence $\{ \wh{q}^{(j)}_j\}_{j\in\ZZ}$ is bounded, so we have $\uppi({\bs{q}}) \in \ell_\infty$.

\begin{lemma}\label{lem:pi.is.homo}
    The mapping $\uppi \colon \cHs \to \ell_\infty$ is a homomorphism  between the $\FFF$-algebras $(\cHs, \circ)$ and $(\ell_\infty, \odot)$.
\end{lemma}
\begin{proof}
    It is easy to see that $\uppi$ is linear: $\uppi(\alpha \bs{q} + \beta \bs{p}) = \alpha \uppi(\bs{q}) + \beta \uppi(\bs{p})$ for all $\bs{q}, \bs{p} \in \cHs$ and $\alpha, \beta \in \FFF$. 
    
    For $\bs{q}, \bs{p} \in \cHs$ we have 
    \begin{equation}\label{eq:pi.composition.q.p}
        \uppi(\bs{q} \circ \bs{p})
        = \sum_{j \in \ZZ} \dotph{\bs{q}(\bs{p} (\bs{\phi}^{(j)})), \bs{\phi}^{(j)}} \bs{e}^{(j)}
    \end{equation}
    Write 
    \begin{equation}\label{eq:p.phi_j.sum}
        \bs{p} (\bs{\phi}^{(j)}) = \sum_{k \in \ZZ} \dotph{\bs{p} (\bs{\phi}^{(j)}), \bs{\phi}^{(k)}} \bs{\phi}^{(k)}
    \end{equation}
    then since $\bs{\phi}^{(j)}$ is idempotent and $\bs{p}$ is $\cH$-linear, we have $\bs{p} (\bs{\phi}^{(j)}) = \bs{\phi}^{(j)} \ff \bs{p} (\bs{\phi}^{(j)})$, thus $\dotph{\bs{p} (\bs{\phi}^{(j)}), \bs{\phi}^{(k)}} = \dotph{\bs{\phi}^{(j)} \ff \bs{p} (\bs{\phi}^{(j)}), \bs{\phi}^{(k)}} $ and by \Cref{def:ff.qtm} 
    \begin{align*} 
    \dotph{\bs{\phi}^{(j)} \ff \bs{p} (\bs{\phi}^{(j)}), \bs{\phi}^{(k)}} &= \delta_{j,k} \dotph{\bs{\phi}^{(j)} \ff \bs{p} (\bs{\phi}^{(j)}), \bs{\phi}^{(k)}}\\ 
    \end{align*}
    Since $\bs{\phi}^{(j)} \ff \bs{z} = \wh{z}_j \bs{\phi}^{(j)}$ 
    for all $\bs{z} \in \cH$, we have that 
    $\dotph{\bs{\phi}^{(j)} \ff \bs{z}, \bs{\phi}^{(k)}} = \wh{z}_j \dotph{\bs{\phi}^{(j)}, \bs{\phi}^{(k)}} = \delta_{jk} \wh{z}_j$. 
    Therefore 
    \begin{align*} 
    \dotph{\bs{\phi}^{(j)} \ff \bs{p} (\bs{\phi}^{(j)}), \bs{\phi}^{(k)}} 
    &= \delta_{j,k} \dotph{\bs{\phi}^{(j)} \ff \bs{p} (\bs{\phi}^{(j)}), \bs{\phi}^{(k)}}  \\
    &= \delta_{j,k} \dotph{\bs{p} (\bs{\phi}^{(j)}), \bs{\phi}^{(j)}}
    \end{align*}
    and~\Cref{eq:p.phi_j.sum} becomes 
    \begin{equation}\label{eq:p.phi_j.final}
        \bs{p} (\bs{\phi}^{(j)}) = \wh{p}_j \bs{\phi}^{(j)} 
    \end{equation}
    where $\wh{\bs{p}} = \uppi(\bs{p})$.
    By noting that \Cref{eq:p.phi_j.final} holds for all quasitubes in $\cHs$, and plugging it into~\Cref{eq:pi.composition.q.p} we get
    \begin{align*}
        \uppi(\bs{q} \circ \bs{p})
        &= {\sum}_{j \in \ZZ} \dotph{\bs{q}(\wh{p}^{(j)} \bs{\phi}^{(j)}), \bs{\phi}^{(j)}} \bs{e}^{(j)} 
        = {\sum}_{j \in \ZZ} \dotph{\wh{p}_j  \bs{q}(\bs{\phi}^{(j)}), \bs{\phi}^{(j)}} \bs{e}^{(j)} \\
        &= {\sum}_{j \in \ZZ} \dotph{\wh{p}_j  \wh{q}_j \bs{\phi}^{(j)} , \bs{\phi}^{(j)}} \bs{e}^{(j)} 
        = {\sum}_{j \in \ZZ} \wh{q}_j \wh{p}_j \bs{e}^{(j)} = \wh{\bs{q}} \odot \wh{\bs{p}} = \uppi(\bs{q}) \odot \uppi(\bs{p})
    \end{align*}
    where the first and third transitions are made by applying~\Cref{eq:p.phi_j.final} to $\bs{p}$ and $\bs{q}$ respectively.
\end{proof}
The following establishes the homomorphism $\uppi $ defined in~\Cref{eq:pi.chs.linf} as an isometric isomorphism between $\cHs$ and $\ell_\infty$.
\begin{lemma}\label{lemma:hstar.ellinf.isoiso.cstar.uppi}
    $\cHs$ and $\ell_\infty$ are isometrically isomorphic via $\uppi$.
\end{lemma}
\begin{proof}
    For any sequence $\hbs{a} \in \ell_\infty$ the operator $T_{\hbs{a}}:\ell_2\to\ell_2$ defined by $T_{\hbs{a}}(\hbs{x}) \coloneqq \hbs{a} \odot \hbs{x}$ is an $\ell_2$-linear operator on $(\ell_2, \odot)$.
    Thus, the mapping 
    \begin{equation}\label{eq:eta.linf.to.chs}
        \eta[\hbs{a}] \coloneqq F^{-1} \circ T_{\hbs{a}} \circ F
    \end{equation}
    is a bounded linear operator on $\cH$.
    Let $\hbs{q} \in \ell_\infty$ and write $\bs{q} \coloneqq \eta[\hbs{q}]$.
    Then 
    \begin{align*}
        \uppi(\bs{q}) 
        &= \sum_{j \in \ZZ} \dotph{\bs{q}(\bs{\phi}^{(j)}), \bs{\phi}^{(j)}} \bs{e}^{(j)} 
        = \sum_{j \in \ZZ} \dotph{\eta[\hbs{q}](\bs{\phi}^{(j)}), \bs{\phi}^{(j)}} \bs{e}^{(j)} \\
        &= \sum_{j \in \ZZ} \dotph{F^{-1}(T_{\hbs{q}}( \bs{e}^{(j)})), \bs{\phi}^{(j)}} \bs{e}^{(j)} 
        = \sum_{j \in \ZZ} \dotph{\wh{q}_j  \bs{\phi}^{(j)}, \bs{\phi}^{(j)}} \bs{e}^{(j)} 
        = \sum_{j \in \ZZ} \wh{q}_j \bs{e}^{(j)} = \hbs{q}
    \end{align*}
    Thus $\uppi \circ \eta = \operatorname{id}_{\ell_\infty}$.
    Similarly, given $\bs{q} \in \cHs$, we write $\hbs{q}\coloneqq \uppi(\bs{q})$, and for any $\bs{x} \in \cH$ we have
    \begin{align*}
        \eta[\hbs{q}](\bs{x}) 
        &= F^{-1} (\hbs{q} \odot F(\bs{x})) 
        = F^{-1} (\uppi(\bs{q})\odot F(\bs{x}))\\
        &= F^{-1} \left(\left({\sum}_{j \in \ZZ} \dotph{\bs{q} (\bs{\phi}^{(j)}) , \bs{\phi}^{(j)}} \bs{e}^{(j)}\right) \odot F(\bs{x})\right)  \\
        &=  \sum_{j \in \ZZ} \wh{x}_j \dotph{\bs{q} (\bs{\phi}^{(j)}) , \bs{\phi}^{(j)}} F^{-1}(\bs{e}^{(j)})   \\
        &=  \sum_{j \in \ZZ}  \dotph{\bs{q} (\wh{x}_j \bs{\phi}^{(j)}) , \bs{\phi}^{(j)}}  \bs{\phi}^{(j)}     
    \end{align*}
    Use that 
    $\dotph{\bs{q} (\wh{x}_k \bs{\phi}^{(k)}) , \bs{\phi}^{(j)}} 
    = \dotph{\wh{q}_k \wh{x}_k \bs{\phi}^{(k)} , \bs{\phi}^{(j)}}
    = \delta_{j,k} \dotph{\wh{q}_k \wh{x}_k \bs{\phi}^{(k)} , \bs{\phi}^{(j)}}
    = \delta_{j,k} \dotph{\bs{q} (\wh{x}_j \bs{\phi}^{(j)}) , \bs{\phi}^{(j)}}$ 
    by \Cref{eq:p.phi_j.final} to write 
    $$\dotph{\bs{q} (\wh{x}_j \bs{\phi}^{(j)}) , \bs{\phi}^{(j)}} = \dotph{\bs{q} \left(\sum_{k \in \ZZ } \wh{x}_k \bs{\phi}^{(k)}\right) , \bs{\phi}^{(j)}} = \dotph{\bs{q} (\bs{x}) , \bs{\phi}^{(j)}}$$
    and thus 
    $\eta[\hbs{q}](\bs{x}) = {\sum}_{j \in \ZZ}  \dotph{\bs{q} (\bs{x}) , \bs{\phi}^{(j)}}  \bs{\phi}^{(j)} = \bs{q} (\bs{x})$.
    So $\eta \circ \uppi = \operatorname{id}_{\cHs}$.
    And we write $\uppi^{-1} \coloneqq \eta$.

    Next, let $\hbs{q} \in \ell_\infty$ and $\bs{q} = \uppi^{-1}_{\hbs{q}}$, then
    \begin{align*}
        \cHsNorm{\bs{q}}
        &= {\sup}_{\cHNorm{\bs{x}} = 1} \cHNorm{\bs{q} (\bs{x})} 
        = {\sup}_{\cHNorm{\bs{x}} = 1} \cHNorm{F^{-1} ( \hbs{q} \odot F( \bs{x}))} \\
        &= {\sup}_{\cHNorm{\bs{x}} = 1} \ltNorm{\hbs{q} \odot F (\bs{x})} 
        = {\sup}_{\ltNorm{\hbs{y}} = 1} \ltNorm{\hbs{q} \odot \hbs{y}} 
        = \infNorm{\hbs{q}}
    \end{align*}
    and thus, the isomorphism $\uppi$ in \Cref{eq:pi.chs.linf} is an isometry. 
\end{proof}

\begin{lemma}\label{lem:pi.star.iso}
    The isomorphism $\uppi$ defined in~\Cref{eq:pi.chs.linf} is a *-isomorphism between $\cHs$ and $\ell_\infty$.
\end{lemma}
\begin{proof}
    Given $\hbs{q} \in \ell_\infty$ we have $\dotplt{\hbs{q} \odot \hbs{x}, \hbs{y}} = \dotplt{\hbs{x}, \hbs{q}^* \odot \hbs{y}}$ for all $\hbs{x}, \hbs{y} \in \ell_2$, where $\hbs{q}^*$ is obtained by applying elementwise complex conjugation to $\hbs{q}$.
    Since the $F$-transform is an isometry, we have that $\dotph{\bs{q} (\bs{x}), \bs{y}} = \dotplt{\hbs{q} \odot \hbs{x}, \hbs{y}}$ for all $\bs{x}, \bs{y} \in \cH$,  where $\bs{q} = \eta[\hbs{q}] $. Thus $\dotph{\bs{q} (\bs{x}), \bs{y}} = \dotplt{\hbs{x}, \hbs{q}^* \odot \hbs{y}}$.
    By applying  $F^{-1}$ to both sides we obtain $\dotph{\bs{q} (\bs{x}), \bs{y}} = \dotph{\bs{x},  \eta[\hbs{q}^*](\bs{y})}$ for all $\bs{x}, \bs{y} \in \cH$.
    Thus, 
    $$
    \bs{q}^* = \eta[\hbs{q}^*]
    $$ and after applying $\uppi$ on both sides we have
    \begin{equation}\label{eq:pi.star.iso}
        \uppi(\bs{q}^*) = \uppi(\bs{q})^* \textnormal{ for all } \bs{q} \in \cHs
    \end{equation}
    so $\uppi$ is a *-isomorphism.
\end{proof}

With the above setup, we have the following result.
\begin{theorem}\label{lem:cHs.star.alg}
    $\cHs$ is a commutative, unital C*-algebra, isometrically *-isomorphic to $\ell_\infty$.
\end{theorem}
\begin{proof}
    Commutativity and unitality of $\cHs$ follow from the corresponding properties of $\ell_\infty$.
    By \Cref{lem:pi.star.iso}, we have that $\cHs$ the mapping $\bs{q} \mapsto \uppi^{-1}[\uppi [\bs{q}]^*]$, where $\hbs{q}^* \in \ell_\infty$ is the elementwise complex conjugate of $\hbs{q} \in \ell_\infty$,  is an involutive automorphism. 

    Since $\uppi$ is an isometric isomorphism, we have that $\cHsNorm{\bs{q}^* \circ \bs{q}} = \infNorm{\uppi(\bs{q}^* \circ \bs{q}) } = \infNorm{\uppi(\bs{q}) \odot \uppi(\bs{q}^*)}=\infNorm{\uppi(\bs{q}) \odot \uppi(\bs{q})^*}$ and by the C*-property of $\ell_\infty$ we have that $\infNorm{\uppi(\bs{q}) \odot \uppi(\bs{q})^*} = \infNormS{\uppi[\bs{q}]} $ hence $\cHsNorm{\bs{q}^* \circ \bs{q}} = \cHsNorm{\bs{q}}^2$. 

    It remains to show that $\cHs$ is complete with respect to $\cHsNorm{\cdot}$, which immediately follows from the fact that $\ell_\infty$ is complete with respect to $\infNorm{\cdot}$ and that $\uppi$ is an isometry. Indeed,
    for any Cauchy sequence $\{\bs{q}^{(n)}\}_{n \in \NN}$ in $\cHs$ we have that $\{\uppi(\bs{q}^{(n)})\}_{n \in \NN}$ is a Cauchy sequence in $\ell_\infty$ and thus converges to some $\hbs{q} \in \ell_\infty$, hence $\bs{q} = \uppi^{-1}[\hbs{q}] \in \cHs$ and   $\cHsNorm{\bs{q}^{(n)} - \bs{q}} \to 0 $.
\end{proof}

Recall that $\cHs$ is the dual module of $\cH$.
As such, it is a module over $\cH$ and hence, the tubal multiplication of a $\cH$-quasitubes $\bs{q} \in \cHs$  by a tube $ \bs{a} \in \cH$, denoted by $\bs{a} \ff \bs{q}$, is well-defined as the quasitube $\bs{a} \ff \bs{q} \in \cHs$ such that $(\bs{a} \ff \bs{q}) (\bs{x}) = \bs{a} \ff (\bs{q} (\bs{x})) = \bs{q}( \bs{x} \ff \bs{a})$ for all $\bs{x} \in \cH$.

On the other hand, we have the view of $\cHs$ as a commutative, unital C*-algebra over $\FFF$.
Here, $\cHs$ is a rank-1 module over itself, in which $\cH$ is embedded as a subspace $\smash{\cH  \xhookrightarrow{~\uptau~} \cHs}$ where $\uptau \colon \cH \to \cHs$ maps $\bs{x} \in \cH$ to the quasitube $\uptau[\bs{x}]\in \cHs$ such that $\uptau[\bs{x}] (\bs{y}) \coloneqq \bs{x} \ff \bs{y}$ for all $\bs{y} \in \cH$.
We can also use the fact that $\cH$ is 
isometrically isomorphic to $\ell_2$, in combination with the fact that $\ell_2 \subseteq \ell_\infty$ and the isometric isomorphism between $\cHs$ and $\ell_\infty$, to define $\uptau$ more explicitly as 
\begin{equation}\label{eq:tau.cH.to.cHs}
    \uptau[\bs{x}] \coloneqq \uppi^{-1}[F(\bs{x})]
\end{equation}
with $\uppi$ defined in~\Cref{eq:pi.chs.linf}, i.e., $\uptau 
= \uppi^{-1} \circ F$.
Denote the image of $\cH$ under $\uptau$ by $\cH' \subseteq \cHs$, and we have the following diagram:
\begin{figure}[H]
    \label{fig:commdiag}
    \centering
    \begin{tikzpicture}
        \node (cH) at (0,0) {$\cH$};
        \node (cHs) at (2.35,0) {\vphantom{$\cHs'$}$\cHs$};
        \node (cHt) at (1.65,0) {\vphantom{$\cHs'$}$\cH' \subseteq$};
        \node (ell2)  at (0.825,-1.45) {$\ell_2$};
        \node (ellinf) at (3.175,-1.45) {$\ell_\infty$};
        \draw[->] (cH) -- node[left]{$F~$} (ell2);
        \draw[->] (cHs) -- node[right]{$\uppi$} (ellinf);
        \draw[->, right hook-latex] (cH) -- node[above]{$\uptau$} (cHt);
        \draw[->, right hook-latex] (ell2) -- node[above]{$\iota$} (ellinf);
    \end{tikzpicture}
    \caption{Diagram of the mappings and embeddings between $\cH, \cHs, \ell_2, \ell_\infty$.}
\end{figure}
Moreover, we have the following result, which will come into play later when discussing identities and properties related to the optimality of approximations obtained by truncating the quasitubal tensor SVD.
\begin{lemma}\label{lem:cH.2sided.star.ideal.in.cHs}
    Consider $\cHs$ as a commutative, unital C*-algebra over $\FFF$, and let $\cH' \subseteq \cHs$ be the embedding of $\cH$ into $\cHs$ as defined in~\Cref{eq:tau.cH.to.cHs}.
    Then $\cH'$ is a two-sided *-ideal in $\cHs$, i.e., $\cH'$ is a subset closed to addition, multiplication and involution such that $\bs{q} \ff \bs{p} \in \cH'$ for any $\bs{q} \in \cH'$ and $\bs{p} \in \cHs$. 
\end{lemma}
\begin{proof}
    To show that $\cH'$ is *-invariant,
    let $\bs{x} \in \cH$ and consider $\uptau[\bs{x}]^* = \uppi^{-1}[F(\bs{x})]^*$.
    By~\Cref{eq:pi.star.iso}, we have that $\uppi^{-1}[F(\bs{x})]^* = \uppi^{-1}[\hbs{x}^*]$ where $\hbs{x}^*$ is the elementwise complex conjugate of $\hbs{x} = F(\bs{x})$. 
    Since $\smash{F^{-1}(\hbs{x}^*) \in \cH}$ we have that $\smash{\uptau[\bs{x}]^* = \uptau[F^{-1}(\hbs{x}^*)]\in \cH'}$.
    
    To establish that \smash{$\cH'$} is a two-sided ideal, we need to show that for all $\bs{x}' \in \cH'$ and $\bs{q} \in \cHs$ we have $\bs{q} \circ \bs{x}' \in \cH'$ (then $\bs{x}' \circ \bs{q} \in \cH'$ follows from commutativity).  
    For any $\bs{x}' \in \cH'$ and $\bs{q} \in \cHs$ we have that $\bs{q} \circ \bs{x}' \in \cH'$ if there exists $\bs{y} \in \cH$ such that $\uptau[\bs{y}] = \bs{q} \circ \bs{x}'$, i.e.,  $\bs{y} \ff \bs{z} = (\bs{q} \circ \bs{x}')(\bs{z})$ for all $\bs{z} \in \cH$.
    Using the associativity of linear operators composition, we get that $(\bs{q} \circ \bs{x}')(\bs{z}) = \bs{q} (\bs{x}'(\bs{z}))$, and since  $\bs{x}' \in \cH'$, there is a tube $\bs{x} \in \cH$ such that $\bs{x}'(\bs{z}) = \uptau[\bs{x}](\bs{z}) = \bs{x} \ff \bs{z}$.
    Now use the fact that $\bs{q} \in \cHs$ is a $\cH$-linear operator to write 
    $$
    \bs{q} (\bs{x}'(\bs{z})) = \bs{q} (\bs{x} \ff \bs{z}) = \bs{z} \ff \bs{q} (\bs{x})
    $$
    where $\bs{q} (\bs{x}) \in \cH$. 
    Set $\bs{y} \coloneqq \bs{q} (\bs{x})$ and observe that $\bs{y} \ff \bs{z} = (\bs{q} \circ \bs{x}')(\bs{z}) $ for all $\bs{z} \in \cH$, and thus $\uptau[\bs{y}] = \bs{q} \circ \bs{x}' $ is in $\cH'$.
    Thus, $\cH'$ is a two-sided *-ideal in $\cHs$
\end{proof}

Henceforth, the notation $\cH$ will be synonymous with $\cH'$, thus we will refer to $\cH$ as an embedded subspace of $\cHs$. To avoid confusion, unities of $\cHs$ and $\ell_\infty$ will be denoted by $\bs{e}$ and $\bs{1}$, respectively.

\begin{table}[H]
    \centering
    \begin{tabular}{c|lccc}
        \textbf{Symbol}& 
        \textbf{Definition}& 
        \textbf{Norm}& 
        \textbf{Multiplication} & 
        \textbf{Unity} \\
        \hline
        $\cH$ & Hilbert space &$\cHNorm{\cdot}$ & $\ff$ & \\
        $\ell_2$ & Square Summable Sequences & $\ltNorm{\cdot}$ & $\odot$ &  \\
        $\cHs$ & Quasitubes (subspace of $B(\cH)$) & $\cHsNorm{\cdot}$ & $\ff$ & $\bs{e}$ \\
        $\ell_\infty$ & Bounded Sequences &$\infNorm{\cdot}$ & $\odot$ & $\bs{1}$ 
    \end{tabular}
\end{table}

\paragraph{Positive elements.}
Consider the SVD of a matrix $\mat{X} = \matU \matSigma \matV^\CT \in \FFF^{m \xx p}$, and recall that $\sigma_1 \geq \sigma_2 \geq \cdots \geq \sigma_{\min(m,p)} \geq 0$. 
In order to show that a similar result hold for a factorization such as~\Cref{eq:tSVDM}, we need a notion of positiveness in $\cHs$.

We define positiveness in a general commutative, unital C*-algebra $(\alg{A}, +, \cdot)$ over $\FFF$, and then discuss its manifestation in $\cHs$.
An element $\bs{a} \in \alg{A}$ is invertible if there exists $\bs{b} \in \alg{A}$ such that $\bs{a} \cdot \bs{b} = \bs{b} \cdot \bs{a} = \bs{e}$, where $\bs{e}$ denote the unit element in $\alg{A}$. The spectrum of $\bs{a} \in \alg{A}$, denoted by $\spc(\bs{a})$, is the set of all $\lambda \in \FFF$ such that $\bs{a} - \lambda \bs{e}$ is not invertible.
Note that if $\bs{a} \in \alg{A}$ is self-adjoint then $\spc(\bs{a}) \subseteq \RR$. 
To see this, suppose that $\bs{a} = \bs{a}^*$ and $\lambda \in \spc(\bs{a})$  scalar, then there exists $\bs{z} \in \alg{A}$ such that $\bs{a} \cdot \bs{z} = \lambda \bs{z}$, hence $\bs{z}^* \cdot \bs{a} = \overline{\lambda} \bs{z}^*$. 
Following right-multiplication by $\bs{z}$, we get $\lambda \bs{z}^* \cdot \bs{z} = \overline{\lambda} \bs{z}^* \cdot \bs{z}$ thus $\lambda = \overline{\lambda}$ and conclude that $\spc(\bs{a}) \subseteq \RR$.

\begin{definition}[Positive Element]\label{def:cstar.positive}
 We say that an element $\bs{a} \in \alg{A}$ is \textbf{positive} or \textbf{non-negative}, and write $\bs{a} \geq_{\alg{A}} 0 $ if $\bs{a}$ is self adjoint and $\spc(\bs{a}) $ is contained in non-negative half of the real line.
\textbf{Strictly positive} element $\bs{a} \in \alg{A}$ is an element such that $\spc(\bs{a}) $ is contained in the positive half of the real line, and are denoted by $\bs{a} >_{\alg{A}} 0$.
For $\bs{a}, \bs{b} \in \alg{A}$ we write $\bs{a} \geq_{\alg{A}} \bs{b}$ if $\bs{a} - \bs{b} \geq_{\alg{A}} 0$.
\end{definition}
For any unital C* algebra $\alg{A}$, the relation $\geq_{\alg{A}}$ is (non-strict) partial order, or, an antisymmetric preorder on $\alg{A}$.
Note that $\bs{a} \in \alg{A}$ is positive if and only if $\bs{a} = \bs{b} \cdot \bs{b}^* $ for some $\bs{b} \in \alg{A}$ \cite[Theorem 3.6]{Conway2007}. 

For the concrete C*-algebra $\ell_\infty$, one can show that   $\spc(\hbs{a}) = \operatorname{cl} \{ \wh{a}_j \}_{j \in \ZZ}$, where $\operatorname{cl}$ denotes the closure of a set with respect to the uniform norm $\infNorm{\cdot}$.
For example, the sequence $\hbs{a} $ with $\wh{a}_j = 1/(j^{2}+1)$ for all $j \in \ZZ$ is non-negative, but not strictly positive, since $0 = \lim_{|j| \to \infty} \wh{a}_j$ thus $0 \in \operatorname{cl} \{ \wh{a}_j \}_{j \in \ZZ} = \spc(\hbs{a})$.
 $\spc(\hbs{a}) = \operatorname{cl} \{ \wh{a}_j \}_{j \in \ZZ}$ implies that a sequence $\hbs{a} \in \ell_\infty$ is non-negative if and only if all its elements are non-negative real numbers.

\begin{theorem}\label{lem:rel.invariance}
    A quasitube $\bs{q} \in \cHs$ is non-negative if and only if $\hbs{q} = \uppi(\bs{q}) \in \ell_\infty$ is non-negative.
\end{theorem}

\begin{proof}
    Let $\alpha \notin \spc(\bs{q})$ then $\bs{q}_\alpha = \bs{q} - \alpha \bs{e}$ is invertible, i.e., there is a $\bs{q}_\alpha^{-1} \in \cHs$.
    Since $\uppi$ is a *-isomorphism, $\hbs{q}_\alpha = \hbs{q} - \alpha \bs{1}$, and $\hbs{q}_\alpha^{-1} \in \ell_\infty$ is the inverse of $\hbs{q}_\alpha$.
    Therefore, $\alpha \notin \spc(\hbs{q})$, and we have $\spc(\hbs{q}) \subseteq \spc(\bs{q})$.
    By the same argument, $\spc(\bs{q}) \subseteq \spc(\hbs{q})$, so $\spc(\bs{q}) = \spc(\hbs{q})$.
    It clearly follows that $\bs{q} \geq_{\cHs} 0$ if and only if $\hbs{q} \geq_{\ell_\infty} 0$.
\end{proof}

The relation $\geq_{\cHs}$ forms a partial order that allows discussion about `best' upper and lower bounds for subsets of $\cHs$, and, in particular, state and prove a result that shows that the `singular quasitubes' in a decomposition of the form of~\Cref{eq:tSVDM} may be sorted in a decreasing order.

\section{Quasitubal Tensors}
In previous sections, we have established the algebraic structure of $\cHs$ as a commutative, unital C*-algebra over $\FFF$, in which $\cH$ is embedded as a *-ideal. 
In this section, we define quasitubal tensors and present their properties in analogy to the finite-dimensional case.
Our construction is concluded in~\Cref{sec:quasitubal.svd} where the existence of quasitubal SVD is established together with optimality results for low-rank approximations obtained by truncating the decomposition.

\subsection{Quasitubal Tensor Algebra}

\begin{definition}[Quasitubal Tensors]\label{def:quasitubal.tensor}
    A $m \xx p$ quasitubal tensor is an element $\tX \in \cH_*^{m \xx p}$, i.e., an $m \xx p$ matrix of $\cH$-quasitubes.
    A $p$-dimensional quasitubal slice $\mat{X} \in \cH_*^p$ is considered as a ``column vector'' of quasitubes, that is $\matX \in \cHs^{p \xx 1}$
\end{definition}
In analogy to the $F$-transform, we use the mapping $\uppi$ from~\Cref{eq:pi.chs.linf} and the result of~\Cref{lem:cHs.star.alg} to represent quasitubes and quasitubal tensors in the transform domain, i.e., in $\ell_\infty$: 
\begin{center}
    \setlength\belowdisplayskip{0pt}
    \setlength\abovedisplayskip{-10pt}
    \setlength\belowdisplayshortskip{0pt}
    \setlength\abovedisplayshortskip{-10pt}
    \begin{minipage}[t]{.35\textwidth}
        \begin{equation}\label{def:pi-transform.quasitube}
            \smash{\wh{\bs{q}} \coloneqq \uppi[\bs{q}] \in \ell_\infty \text{ for all } \bs{q} \in \cHs}
        \end{equation}
    \end{minipage}%
    \begin{minipage}[t]{.39\textwidth}
        \begin{equation}\label{def:pi-transform.quasitensor}
            \smash{\wh{\tX} \coloneqq \tX \xx_3\hspace{2pt}\uppi \text{ for all } \tX \in \cH_*^{m \xx p}}
        \end{equation}
    \end{minipage}
\end{center}
We call the mappings in~\Cref{def:pi-transform.quasitube,def:pi-transform.quasitensor} the $\boldsymbol{\uppi}$\textbf{-transform} of quasitubes and quasitubal tensors, respectively.
The hat notation is used both for denoting $F$-transformed tubes and the $\uppi$-transformed quasitubes. 
Even though the mappings differ in their domains, they are consistent in the sense that for a tube $\bs{x}\in\cH$ the $\uppi$-transform of the associated quasitube $\uptau[\bs{x}]$ is the $F$-transform of $\bs{x}$. So, we can  overload $\uppi$ by defining $\uppi(\bs{x}) \coloneqq \uppi(\uptau[x])$, and with this definition $\uppi$ can be viewed as extending $F$ to quasitubes that cannot be represented by tubes. Hence there is no ambiguity in using the same notation for both transforms.

For convenience and clarity, we re-define the $\ff$-product (\cref{def:ff.qtm,def:ff.tensor}), originally defined for tubes in $\cH$, to apply to quasitubes in $\cHs$.
Note that the redefined $\ff$-product is consistent with the $\circ$ operation in $\cHs$, so it is a homomorphism of algebras.
\begin{definition}[Quasitubal \texorpdfstring{$\ff$}{F}-Product]\label{def:quasi.ff.redef}
    We define the quasitubal $\ff$-product as 
    \begin{center}
        \setlength\belowdisplayskip{0pt}
        \setlength\abovedisplayskip{-10pt}
        \setlength\belowdisplayshortskip{0pt}
        \setlength\abovedisplayshortskip{-10pt}
        \begin{minipage}{.3\linewidth}
            \begin{equation}\label{def:quasi.ff.redef.scalar}
                \bs{x} \ff \bs{y} \coloneqq \uppi^{-1}(\wh{\bs{x}} \odot \wh{\bs{y}}) 
            \end{equation}
        \end{minipage}%
        \begin{minipage}{.5\linewidth}
            \begin{equation}\label{def:quasi.ff.redef.matmul}
                \tX \ff \tY \coloneqq (\thX \vartriangle \thY) \xx_3 \pi^{-1}
            \end{equation}
        \end{minipage}
    \end{center}
    for all $\bs{x},\bs{y} \in \cHs, \tX \in \cH_*^{m \xx p}$, $\tY \in \cHs^{p \xx q}$, and
     $\vartriangle$ is the facewise multiplication of matrices over $\ell_\infty$.
\end{definition}
Let $\tX \in \cHs^{m \xx p} $ then
by \Cref{lem:Rlin.iff.matmul}, the mapping $T_{\tX} \colon \cHs^p \to \cHs^m $ defined by $T_{\tX} \mat{Z} = \tX \ff \mat{Z}$  for all $\mat{Z} \in \cHs^p$, is an $\cHs$-linear (in particular $\cH$-linear) mapping from $\cHs^p $ to $ \cHs^m$.

Let  $\thX \in \ell_\infty^{m \xx p }$, then for all $k \in \ZZ$ we denote the $k$-th frontal slice of $\thX$ by $\thX_{:,:,k} \in \FFF^{m \xx p}$. 
The $\boldsymbol{\ff}$\textbf{-conjugate transpose} of $\tX \in \cH_*^{m \xx p}$ is the quasitubal tensor 
\begin{equation}\label{eq:def.ff.conj.transpose}
    \tX^* \coloneqq \thX^{\CT} \xx_3 \uppi^{-1} \in \cH_*^{p \xx m} ~~~\text{ with } (\thX^\CT)_{:,:,k} = (\thX_{:,:,k})^\CT \in \FFF^{m \xx p}  \quad \text{for all } k \in \ZZ
\end{equation}
It is easy to verify that $(\tX \ff \tY)^* = \tY^* \ff \tX^*$ for all $\tX$ and $\tY$ of compatible sizes.

A quasitubal slice $\mat{A} \in \cH_*^p$ is said to be $\boldsymbol{\ff}$\textbf{-unit-normalized} if $\mat{A}^* \ff \mat{A} = \bs{e}$, and two quasitubal slices $\mat{A}, \mat{B} \in \cH_*^p$ are said to be $\boldsymbol{\ff}$\textbf{-orthogonal} if $\mat{A}^* \ff \mat{B} = 0$.
The $p \xx p$ \textbf{identity} quasitubal tensor in $\cH_*^{p \xx p}$ is denoted by $\tens{I}_p$ (or $\tens{I}$ when the dimension is clear from the context), and is such that $\tens{I} \ff \tens{X} = \tens{X} $ and $\tens{Y} \ff \tens{I} = \tens{Y}$ for all quasitubal tensors $\tens{X}, \tens{Y} $ of appropriate dimensions.
Concretely, $\tens{I}_p$ is an f-diagonal tensor whose diagonal entries are $\bs{e}$.

The reason for denoting the $\ff$ conjugate transpose by the star symbol, usually reserved for denoting the adjoint of an operator, is clarified in~\Cref{sec:quasitubal.tensor.alg} where we study quasitubal tensors as operators and module homomorphisms, and discuss the above definitions in detail.
For now, it suffices to note that $[\tX^*]_{j,k} = \bs{x}_{kj}^*$. In the last sentence, the adjoint of the right hand side is the adjoint of a quasitube, as defined in \Cref{subsec:quasitubes}. We also see that the two definitions are consistent.

In consistency with current naming convention, 
we use the \textit{f-} prefix to describe properties of matrices that hold for all frontal slices of a quasitubal tensor in the transform domain. 
A quasitubal tensor $\tX \in \cHs^{m \xx p}$ is f-diagonal if for all $n \in \ZZ$ the matrix $\thX_{:,:,n}$ is a (possibly rectangular) diagonal matrix.
Correspondingly, a quasitubal tensor $\tU \in \cHs^{p \xx p}$ is said to be f-unitary (or f-orthogonal) if for all $n$, the $n$-th frontal slice of $\thU$, $\thU_{:,:,n} $, is a $p \xx p$ unitary matrix.

\paragraph{Norms and \texorpdfstring{$\ff$}{star F}-Unitarity.}
Since $\tA \in \cH_*^{m \xx p}$ if and only if the mapping $\mat{X} \mapsto \tA \ff \mat{X}$ is a bounded $\cHs$-linear thus, in particular, $\cH$-linear operator from $\cH^p$ to $\cH^m$,
the space $\cH_*^{m \xx p}$ is isomorphic to a subspace of bounded linear operators from $\cH^p$ to $\cH^m$. 
As such, it admits an operator norm that is induced by the respective norms of the domain and codomain spaces, namely $\cH^p$ and $\cH^m$.
The space $\cH^{m \xx p} $ is a direct sum of $mp$ copies of the space $\cH$, $\cH^{m \xx p} = \bigoplus_{j=1}^{m} \bigoplus_{k=1}^{p} \cH$, thus in addition to the operator norm of the ambient space $\cH_*^{m \xx p}$, in which $\cH^{m \xx p}$ is embedded, we have the norm induced by the direct sum structure of $\cH^{m \xx p}$ (see~\Cref{sec:quasitubal.tensor.alg}).

\begin{definition}[Quasitubal Tensor Norm]\label{def:quasitubal.tensor.norm}
    The $\bs{\cH}$\textbf{-norm} of a quasitubal tensor $\tA \in \cH^{m \xx p}$ is 
    \begin{equation}\label{eq:qt.frob.norm}
        \cHNorm{\tA} = \sqrt{{\sum}_{j=1}^{m} {\sum}_{k=1}^{p} \cHNormS{\bs{a}_{kj}}}
    \end{equation}
    And the $\bs{\cH}$\textbf{-induced}, or \textbf{operator norm}, of $\tB \in \cHs^{m \xx p}$ under $\ff$ is defined as
    \begin{equation}\label{eq:qt.op.norm}
        \opNorm{\tB} = {\sup}_{\mat{X} \in \cH^p, \cHNorm{\mat{X}} = 1} \cHNorm{\tB \ff \mat{X}}
    \end{equation}
    where in the above is based on $\cH^{m \xx p} \xhookrightarrow{~} \cHs^{m \xx p}$ being an embedding of $\cH$ into $\cHs$ as defined in~\Cref{eq:tau.cH.to.cHs}.

    When $\mat{X} \in \cH^p$ is such that $\cHNorm{\mat{X}} = 1$ we call $\mat{X}$ a \textbf{unit-length} slice (not to be confused with $\ff$-unit-normalized slices).
\end{definition}

For $\tA \in \cH^{m \xx p}$, by definition $\bs{a}_{jk} \in \cH$, hence the sum in~\Cref{eq:qt.frob.norm} is finite. 
For~\Cref{eq:qt.op.norm}, let $\tB \in \cHs^{m \xx p}$ and $\matX \in \cH^p$.
Write $\matY = \tB \ff \matX $, and observe that for all $h \in [m]$
we have $\bs{y}_h = \tB_{h,:} \ff \mat{X} = \sum_{j=1}^p \bs{b}_{hj} \ff \bs{x}_j$, then
\begin{align*}
    {\sum}_{k \in \ZZ} |\wh{y}_{h}(k)|^2 
    &=  {\sum}_{k \in \ZZ} |{\sum}_{j=1}^p \wh{b}_{hj}(k)  \wh{x}_{j}(k)|^2 \\
    &\leq {\sum}_{k \in \ZZ} |{\sum}_{j=1}^p \wh{b}_{hj}(k)|^2 |{\sum}_{j=1}^p \wh{x}_{j}(k)|^2 \\
    &\leq {\sum}_{k \in \ZZ} p \max_{j' \in [p]} | \wh{b}_{hj'}(k)|^2 p{\sum}_{j=1}^p |\wh{x}_{j}(k)|^2 \\
    &\leq p^2 \sup_{k' \in \ZZ} \max_{j' \in [p]} | \wh{b}_{hj'}(k')|^2 {\sum}_{k \in \ZZ} {\sum}_{j=1}^p |\wh{x}_{j}(k)|^2 \\
    &\leq p^2 \max_{j' \in [p],h' \in [m]} \infNorm{\hbs{b}_{h'j'}}^2 \cHNormS{\mat{X}}
\end{align*}
Hence $\bs{y}_h \in \cH$ and $\matY \in \cH^m$.
It follows that for any $\tB \in \cHs^{m \xx p}$, there exists a constant $A \in \RR$ such that $\cHNorm{\tB \ff \matX} \leq A \cHNorm{\matX}$  for all $\matX \in \cH^p$, therefore, the supremum in~\Cref{eq:qt.op.norm} exists.

Furthermore, note that for all $k \in \ZZ$, we have $\TNorm{\wh{\tB}_{:,:,k}} = \max_{\TNorm{\wh{\mat{x}}} \leq 1 } \TNorm{\wh{\tB}_{:,:,k} \wh{\mat{x}}}$, and 
denote by $\wh{\mat{x}}^{(k)} \in \FFF^p$ the (unit length)  vector such that $\TNorm{\wh{\tB}_{:,:,k}} = \TNorm{\wh{\tB}_{:,:,k} \wh{\mat{x}}^{(k)}}$.
Correspondingly, let $\matX^{(k)} $ denote the quasitubal slice such that $\wh{\mat{X}}^{(k)}_{:,:,k'} = \delta_{kk'} \wh{\mat{x}}^{(k)}$ then $\cHNorm{\matX^{(k)}} = \FNorm{\wh{\mat{x}}^{(k)}} = 1$.
By definition, $\cHNorm{\tB\ff \matX^{(k)}} \leq \opNorm{\tB}$ hence $\{ \TNorm{\wh{\tB}_{:,:,k} } ~|~ k \in \ZZ \}$ is a bounded set in $\RR$.
In particular, $\sup_{k \in \ZZ}\TNorm{\wh{\tB}_{:,:,k} } $ exists and we have the following.





\begin{lemma}\label{lem:qtopnorm.and.linf2norm}
    Let $\tA \in \cHs^{m \xx p}$, then $\opNorm{\tA} = \sup_{j \in \ZZ} \TNorm{\wh{\tA}_{:,:,j}}$.
\end{lemma}
\begin{proof}
    For all unit-length $\mat{X} \in \cH^p$ we have 
    \begin{align*}
        \cHNormS{\tA \ff \mat{X}} 
        &= \ltNormS{\wh{\tA} \vartriangle \wh{\mat{X}}}
        = {\sum}_{k \in \ZZ} \FNormS{\wh{\tA}_{:,:,k} \wh{\mat{X}}_{:,k}}\\
        &\leq {\sum}_{k \in \ZZ} \TNormS{\wh{\tA}_{:,:,k}} \FNormS{\wh{\mat{X}}_{:,k}} 
        \leq ({\sup}_{j \in \ZZ} \TNormS{\wh{\tA}_{:,:,j}}){\sum}_{k \in \ZZ} \FNormS{\wh{\mat{X}}_{:,k}} \\
        &= {\sup}_{j \in \ZZ} \TNormS{\wh{\tA}_{:,:,j}}
    \end{align*}
Hence $\opNorm{\tA} \leq {\sup}_{j \in \ZZ} \TNorm{\wh{\tA}_{:,:,j}}$.
Next, note that for all $j \in \ZZ$ the supremum of 
$\sup_{\FNorm{\wh{\mat{z}}} = 1 } \FNorm{\wh{\tA}_{:,:,j} \wh{\mat{z}}}$ 
is attained, and we have that 
$\FNorm{\wh{\tA}_{:,:,j} \wh{\mat{x}}_j} = \TNorm{\wh{\tA}_{:,:,j}}$ for some unit norm $\wh{\mat{x}}_j \in \FFF^p$.
Define $\mat{X}_{(j)} \in \cH^p$ as the tubal tensor with 
$$
{\wh{\mat{X}_{(j)}}}_{:,k} = \delta_{j,k} \wh{\mat{x}}_j
$$
Then 
$
    \cHNorm{\mat{X}_{(j)}}
    = \ltNorm{\wh{\mat{X}_{(j)}}} 
    = \FNorm{\wh{\mat{x}}_j} = 1
$.
Thus, for all $j \in \ZZ$ we have that $\TNorm{\wh{\tA}_{:,:,j}} = \cHNorm{\tA \ff \mat{X}_{(j)}} \leq \opNorm{\tA}$, and conclude that $\opNorm{\tA} = \sup_{j \in \ZZ} \TNorm{\wh{\tA}_{:,:,j}}$.

\end{proof}
Recall that in~\Cref{sec:mm.tubal.tensor.algebra} we defined the notion of unitarity for tubal tensors as the property of preserving the Frobenius norm.
For infinite dimensional quasitubal tensors, we have the following definition.
\begin{definition}
    A quasitubal tensor $\tens{U} \in \cHs^{p \xx p}$ is said to be $\boldsymbol{\ff}$\textbf{-unitary} if $\tens{U}^* \ff \tens{U} = \tens{U} \ff \tens{U}^* = \tens{I}$.
\end{definition}

The following lemma is a direct consequence of algebraic identities. 
\begin{lemma}\label{lem:ff.unitary.identities}
    Let $\tens{I}_p \in \cHs^{p \xx p}$ be the identity quasitubal tensor, then $[\wh{\tens{I}_p}]_{:,:,k} = \matI_p$ for all $k \in \ZZ$, where $\matI_p$ is the $p \xx p$ identity matrix. 
    As a consequence, a quasitubal tensor $\tens{U} \in \cHs^{p \xx p}$ is $\ff$-unitary if and only if it is f-unitary. 
\end{lemma}

Similarly to the finite-dimensional case, $\ff$ multiplication by an $\ff$-unitary tensor is an isometry.

\begin{lemma}\label{lemma:ffuni.norm.preserve}
    Let $\tens{U} \in \cHs^{m \xx m}$ be an $\ff$-unitary tensor, then for all $\tens{X} \in \cHs^{m \xx p}$ we have that $\opNorm{\tens{U} \ff \tens{X}} = \opNorm{\tens{X}}$.
    Moreover, if $\tX \in \cH^{m \xx p}$, then $\cHNorm{\tens{U} \ff \tens{X}} = \cHNorm{\tens{X}}$.
 \end{lemma}
 \begin{proof}
    According to~\Cref{lem:ff.unitary.identities},  $\wh{\tens{U}}_{:,:,j}$ is a unitary matrix for all $j \in \ZZ$, hence $\TNorm{\wh{\tens{U}}_{:,:,j} \wh{\mat{Y}} } = \TNorm{\wh{\mat{Y}}}$ for any $\wh{\mat{Y}} \in \FFF^{m \xx p}$, and in particular,  $\TNorm{\wh{\tens{U}}_{:,:,j} \wh{\tens{X}}_{:,:,j}} = \TNorm{\wh{\tens{X}}_{:,:,j}}$ for all $j \in \ZZ$.
    Apply~\Cref{lem:qtopnorm.and.linf2norm} to obtain that $\opNorm{\tens{U} \ff \tens{X}} = \sup_{j \in \ZZ} \TNorm{\wh{\tens{U}}_{:,:,j} \wh{\tens{X}}_{:,:,j}} = \sup_{j \in \ZZ} \TNorm{\wh{\tens{X}}_{:,:,j}} = \opNorm{\tens{X}}$, and the result follows.

    In case $\tX \in \cH^{m \xx p}$, by~\Cref{eq:qt.frob.norm}
    we have that 
    \begin{align*}
        \cHNormS{\tens{X}}
        &= {\sum}_{j=1}^{m} {\sum}_{k=1}^{p} \cHNormS{\bs{x}_{kj}} \\
        &= {\sum}_{j=1}^{m} {\sum}_{k=1}^{p} \ltNormS{\wh{\bs{x}_{kj}}}\\
        &= {\sum}_{j=1}^{m} {\sum}_{k=1}^{p} {\sum}_{n \in \ZZ} |\wh{x}_{kj}(n)|^2 
    \end{align*}
    where $\wh{x}_{kj}(n)$ is the $n$-th element of the sequence $\wh{\bs{x}_{kj}}$.
    Since the series is absolutely convergent, we can rearrange the terms and obtain that $\cHNormS{\tens{X}} = {\sum}_{n \in \ZZ} \FNormS{\wh{\tens{X}}_{:,:,n}}$.
    And for any sequence $\{ \wh{\mat{U}}_{(n)} \}_{n \in \ZZ}$ of unitary $m \xx m$ matrices, we have that $\cHNormS{\tens{X}} = {\sum}_{n \in \ZZ} \FNormS{\wh{\mat{U}}_{(n)} \wh{\tens{X}}_{:,:,n}}$.
    Apply~\Cref{lem:ff.unitary.identities} to obtain that $\cHNormS{\tens{U} \ff \tens{X}} = \cHNormS{\tens{X}}$ for any $\ff$-unitary quasitubal tensor $\tens{U}$.
 \end{proof}


\subsection{Quasitubal SVD}\label{sec:quasitubal.svd}
\begin{definition}[q-SVD]\label{def:quasitubal.svd}
    Consider a quasitubal tensor $\tX \in \cHs^{m \xx p}$.
    A \textbf{quasitubal SVD}, or \textbf{q-SVD}$\boldsymbol{\ff}$ of $\tX$ is a factorization $\tX = \tU \ff \tS \ff \tV^*$ where $\tU \in \cHs^{m \xx m}$ and $\tV \in \cHs^{p \xx p}$ are $\ff$-unitary, and $\tS \in \cHs^{m \xx p}$ is f-diagonal with diagonal entries $\bs{s}_{1} \geq_{\cHs} \bs{s}_{2} \geq_{\cHs} \ldots \geq_{\cHs} \bs{s}_{\min(m,p)} \geq_{\cHs} 0$.
\end{definition}
Existence of a quasitubal SVD decomposition for any quasitubal tensor is established in the following theorem. 
\begin{theorem}\label{thm:quasitubal.svd.exists}
    Any $\tX \in \cHs^{m \xx p}$ admits a quasitubal SVD $\tX = \tU \ff \tS \ff \tV^*$ as defined in~\Cref{def:quasitubal.svd}.

    We call the `column slices' of $\tU$ and $\tV$ the left and right \textbf{singular quasitubal slices} of $\tX$, respectively, and the diagonal entries of $\tS$ the \textbf{singular quasitubes} of $\tX$.
\end{theorem}

\begin{proof}
    For existence, consider the $\uppi$-transform of $\tX$ as defined in~\Cref{def:pi-transform.quasitensor}, and apply the SVD to each frontal slice of $\wh{\tX}$, i.e., $\wh{\tX}_{:,:,k} = \wh{\matU}_k \wh{\matSigma}_k \wh{\matV}_k^\CT$.
    Set $\wh{\tens{U}},\wh{\tens{S}},\wh{\tens{V}}$ as the quasitubal tensor over $\ell_\infty$ with frontal slices $\wh{\matU}_k, \wh{\matSigma}_k, \wh{\matV}_k$, respectively, for all $k \in \ZZ$.
    Define $\tU = \wh{\tens{U}} \xx_3 \uppi^{-1}, \tS = \wh{\tens{S}} \xx_3 \uppi^{-1}, \tV = \wh{\tens{V}} \xx_3 \uppi^{-1}$, then by \Cref{lem:ff.unitary.identities} we have that $\tU, \tV$ are $\ff$-unitary.
    Since $\thS$ is f-diagonal, we have that $\tS$ is f-diagonal as well.
    Consider the quasitube $\hbs{s}_{j} \coloneqq \wh{\tS}_{j, j,:} \in \ell_\infty$ that is the $j$-th diagonal entry of $\thS$, whose $k$-th element is denoted by $\wh{s}^{(k)}_{j}$.
    By the properties of matrix SVD, we have that $\wh{s}^{(k)}_{j} \geq \wh{s}^{(k)}_{j+1}$.
    It follows that $\hbs{s}_{j} \geq_{\ell_\infty} \hbs{s}_{j+1}$,
    and thus, by~\Cref{lem:rel.invariance} we have that $\bs{s}_{1} \geq_{\cHs} \bs{s}_{2} \geq_{\cHs} \cdots \geq_{\cHs} \bs{s}_{\min(m,p)} \geq_{\cHs} 0$.
    Hence $\tX = \tU \ff \tS \ff \tV^*$ is a q-SVD$\ff$ of $\tX$.
\end{proof}

Having established a notion of singular value decomposition, we can further develop notions of ranks and rank truncation. Let $\tX \in \cHs^{m \xx p}$ be a quasitubal tensor, and $\tU \ff \tS \ff \tV^*$ be a q-SVD$\ff$ of $\tX$.
The \textbf{q-rank} of $\tX$ under $\ff$ is the number of non-zero singular quasitubes in the q-SVD of $\tX$.
The \textbf{multi-rank} of $\tX$ under $\ff$ is the sequence $\bs{\rho} $ of integers such that $\rho_k = \rnk(\wh{\tX}_{:,:,k})$ for all $k \in \ZZ$.

\begin{definition}[Rank Truncations]\label{def:qt.rank.truncations}
    Let $\tX \in \cHs^{m \xx p}$ be a quasitubal tensor, and $\tU \ff \tS \ff \tV^*$ be a q-SVD$\ff$ of $\tX$. The \textbf{q-rank-$\mathbf{r}$ truncation} of $\tX$ is the quasitubal tensor $[\tX]_{r} \coloneqq \tU_{r} \ff \tS_{r} \ff \tV_{r}^*$ where $\tU_{r} \coloneqq \tU_{:,1:r} \in \cHs^{m \xx r}, \tV_{r} \coloneqq \tV_{:,1:r} \in \cHs^{p \xx r}$ and $\tS_{r} \coloneqq \tS_{1:r,1:r} \in \cHs^{r \xx r}$.
    The \textbf{multi-rank-$\bs{\rho}$ truncation} of $\tX$ is the quasitubal tensor $[\tX]_{\bs{\rho}} $ such that $\wh{\tX}_{:,:,k} = \wh{\tU}_{:,1:\rho_k,k} \vartriangle \wh{\tS}_{1:\rho_k,1:\rho_k,k} \vartriangle (\wh{\tV}_{:,1:\rho_k,k})^\CT$.
\end{definition}

\subsection{Optimality of Multi-rank Truncations}
The success and usefulness of any rank truncations of tensors rely on Eckart-Young-like results that guarantee the optimality of low-rank approximations~\cite{KilmerPNAS}. 
Here we establish that a optimal low q(multi)-rank approximation of a quasitubal tensor $\tX$ under $\ff$ is obtained by the truncations defined in~\Cref{def:qt.rank.truncations}, providing our first Eckart-Young-like result for quasitubal tensors.


\begin{theorem}\label{thm:qt.best.lowrank.op}
    Let $\tX \in \cHs^{m \xx p}$ be a quasitubal tensor, and $[\tX]_{r}, [\tX]_{\bs{\rho}}$ be the q-rank $r$ and multi-rank $\bs{\rho}$ truncations of $\tX$ under $\ff$ (as defined in~\Cref{def:qt.rank.truncations}).
    
    Then, for any $\tens{Y} \in \cHs^{m \xx p}$ we have that 
    (1) if the multi-rank of $\tY$ is $\bs{\varphi} \leq_{\ell_\infty} \bs{\rho}$, then $\opNorm{\tX - [\tX]_{\bs{\rho}}} \leq \opNorm{\tX - \tens{Y}}$,  and 
    (2) if the q-rank of $\tY$ is $\varphi \leq r$, then $\opNorm{\tX - [\tX]_{r}} \leq \opNorm{\tX - \tens{Y}}$.
\end{theorem}
\begin{proof}
    For all $j \in \ZZ$ and $k = 1, \ldots, p$,
    write $\wh{\mat{v}}_{k}^{(j)} \coloneqq \wh{\tV}_{:,k,j}$, where $\tV$ is the right singular quasitubal slices of $\tX$.
 Let $\tens{Y} \in \cHs^{m \xx p}$ be a quasitubal tensor with multi-rank $\bs{\varphi} $  under $\ff$ and assume that $\bs{\varphi} \leq_{\ell_\infty} \bs{\rho}$, i.e., $\varphi_j \leq \rho_j$ for all $j \in \ZZ$.
    By the spectral norm formulation of the  Eckart-Young theorem for matrices, we have that $\TNorm{\wh{\tX}_{:,:,j} - \wh{\tY}_{:,:,j}} \geq \TNorm{\wh{\tX}_{:,:,j} - \wh{[\tX]_{\bs{\rho}}}_{:,:,j}}$ for all $j \in \ZZ$.
    Hence, by~\Cref{lem:qtopnorm.and.linf2norm} we have $\opNorm{\tX - \tY} = \sup_{j \in \ZZ} \TNorm{\wh{\tX}_{:,:,j} - \wh{\tY}_{:,:,j}} \geq \sup_{j \in \ZZ} \TNorm{\wh{\tX}_{:,:,j} - \wh{[\tX]_{\bs{\rho}}}_{:,:,j}} = \opNorm{\tX - [\tX]_{\bs{\rho}}}$.

    For the second part, suppose $\tY$ has q-rank $\varphi \leq r$ under $\ff$.
    Note that $\bs{\varphi}' \coloneqq \varphi \bs{1} \leq_{\ell_\infty} r \bs{1} \eqqcolon \bs{\rho}'$, and that $\tens{Y}_{\bs{\varphi}'} = \tens{Y}$. 
    By applying the same argument as above we have that $\opNorm{\tX - [\tX]_{r}} = \opNorm{\tX - [\tX]_{\bs{\rho}'}} \leq  \opNorm{\tX - \tens{Y}}$.
\end{proof}

\begin{corollary}\label{cor:qt.best.lowrank.frob}
    Let $\tX \in \cHs^{m \xx p}$ be a quasitubal tensor, and $\tX = \tens{U} \ff \tens{S} \ff \tens{V}^*$ be a q-SVD$\ff$ of $\tX$.
    Then $\opNorm{\tX} = \opNorm{\tS}$ and if $\tX \in \cH^{m \xx p}$ then $\cHNorm{\tX} = \cHNorm{\tS}$.
\end{corollary}

\section{Quasitubal Tensor Algebraic Framework}\label{sec:quasitubal.tensor.alg}

\Cref{thm:qt.best.lowrank.op} established an Eckart-Young-like theorem for quasitubal tensors. However, since it uses multi-rank truncation, when $\cHs$ is infinite dimensional then typically the approximation does not have a finite representation. 

Arguably, it is desirable to be able to approximate a quasitubal tensor with a tensor which posses a {\em finite} representation. We propose such an approximation, and prove an Eckart-Young-like result for it, in the next section. To be able to do so, we first need to elaborate on the algebraic structure and properties of quasitubal tensors.

\subsection{Commutative Unital C*-algebra Structures}
\paragraph{General observations.}
Let $\alg{A}$ be a commutative, unital C*-algebra over $\FFF$, and $\bs{x} \in \alg{A}$. 
Then for any $\lambda \in \FFF$ we have $(\bs{x} - \lambda \cdot \bs{e})^* = \bs{x}^* - \overline{\lambda} \cdot \bs{e}$. In general, an element $\bs{y} \in \alg{A}$ is invertible if and only if $\bs{y}^*$ is invertible. Thus, a scalar $\lambda \in \FFF$ is in the spectrum of $\bs{x}$ if and only if $\overline{\lambda} \in \spc(\bs{x}^*)$.

The following quantity is defined for elements in general Banach algebras. 
\begin{definition}[Spectral Radius]\label{def:spcrad}
    Let $\alg{A}$ be a Banach algebra and $\bs{a} \in \alg{A}$.
    The \textbf{spectral radius} of $\bs{a}$ is defined as 
    \begin{equation*}
        r(\bs{a}) = \sup \{|\alpha| :\, \alpha \in \spc(\bs{a}) \}
    \end{equation*}
\end{definition}
We use the standard notion of natural exponents to mark repeated multiplication, e.g., for $\bs{a} \in \alg{A}$ we write $\bs{a}^0 \coloneqq \bs{e}$ and $\bs{a}^k \coloneqq \bs{a} \odot \bs{a}^{k-1}$ for all $k \in \NN$. 
In particular, $\bs{a}^2 = \bs{a}\odot\bs{a}$.
\begin{lemma}[\texorpdfstring{\cite[Chapter VII, Proposition 3.8 and Chapter VIII, Proposition 1.11]{Conway2007}}{\cite[Chapter VII, Proposition 3.8 and Chapter VIII, Proposition 1.11]{Conway2007}}]\label{lem:spcrad.exists.norms}
    Suppose that $\alg{A}$ is a Banach algebra with identity, then for any $\bs{a} \in \alg{A}$, the spectral radius defined in \Cref{def:spcrad} exists and  $r(\bs{a}) = \lim_{n \to \infty} \GNorm{\bs{a}^n}_{\alg{A}}^{1/n}$.
    If, in addition, $\alg{A}$ is a C*-algebra, then for a self-adjoint $\bs{a} = \bs{a}^* \in \alg{A}$ it holds that 
    $\GNorm{\bs{a}}_{\alg{A}} = r(\bs{a})$.
\end{lemma}

\begin{lemma}[Uniqueness of Norms on C*-Algebras \texorpdfstring{\cite[Chapter VIII, Proposition 1.8]{Conway2007}}{\cite[Chapter VIII, Proposition 1.8]{Conway2007}}]\label{lem:general.cstar.unique.norm}

    Let $\alg{A}$ be a commutative, unital C*-algebra.
    The norm associated with $\alg{A}$ is uniquely determined by the algebra structure, and is given by
    \begin{align*}
        \GNorm{\bs{a}}_{\alg{A}} 
        &= \sup \{ \GNorm{\bs{a} \cdot \bs{x}}_{\alg{A}} : \bs{x} \in \alg{A} , \GNorm{ \bs{x}}_{\alg{A}} \leq 1 \} \\
        &= \sup \{ \GNorm{\bs{x} \cdot \bs{a}}_{\alg{A}} : \bs{x} \in \alg{A} , \GNorm{ \bs{x}}_{\alg{A}} \leq 1 \}
    \end{align*}
\end{lemma}
Note that the expression for $\GNorm{\cdot}_{\alg{A}}$ in \Cref{lem:general.cstar.unique.norm} is a `self-induced'  norm, i.e., the norm is defined as an upper bound of some constrained set under the image of the same norm. 
Nevertheless, this result is not circular:
for any $\bs{a} \in \alg{A}$ we have that $\bs{a} \cdot \bs{a}^* $ is self-adjoint and from \Cref{lem:spcrad.exists.norms} we have $\GNorm{\bs{a} \cdot \bs{a}^*}_{\alg{A}} = r(\bs{a} \cdot \bs{a}^*)$, hence, by the C* identity, $\GNorm{\bs{a}}_{\alg{A}} = \sqrt{\GNorm{\bs{a} \cdot \bs{a}^*}_{\alg{A}}} = \sqrt{r(\bs{a} \cdot \bs{a}^*)}$.
Since \Cref{def:spcrad} of the spectral radius is independent of the norm, we have that the seemingly self-inducing norm expressed in \Cref{lem:general.cstar.unique.norm} is well defined.


\begingroup
{

}
\endgroup

\begin{lemma}\label{lem:pos.geq.norm.bound}
    Let $\alg{A}$ be a unital C*-algebra, then 
    $\bs{a} \geq_{\alg{A}} \bs{b} \geq_{\alg{A}} 0$ implies $\GNorm{\bs{a}}_{\alg{A}} \geq \GNorm{\bs{b}}_{\alg{A}} $.
\end{lemma}
The following lemma will serve us in proving~\Cref{lem:pos.geq.norm.bound}.
\begin{lemma}\label{lem:helper.for.geqnormbound}
    Let $\alg{A}$ be a unital C*-algebra.
    For all $\bs{a} \in \alg{A}$ and $\alpha \in \CC$ it holds that 
    $\lambda \in \spc(\bs{a})$ if and only if $\alpha - \lambda \in \spc(\alpha \bs{e} - \bs{a})$.
    Equivalently, $\lambda \in \spc(\alpha \bs{e} - \bs{a})$ if and only if $\alpha - \lambda \in \spc(\bs{a}) $.
\end{lemma}
\begin{proof}
    For all $\lambda \notin \spc(\bs{a})$, we have that $(\bs{a} - \lambda \bs{e})$ is invertible, so there exists some $\bs{z} \in \alg{A}$ such that $(\bs{a} - \lambda \bs{e}) \bs{z} = \bs{e}$. 
    Then, for all $\alpha \in \RR$ we have 
    \begin{align*}
        (\alpha \bs{e} -  \bs{a} - (\alpha - \lambda) \bs{e} ) (-\bs{z})
        &= ( -  \bs{a}  + \lambda \bs{e} ) (-\bs{z}) \\
        &= (\bs{a} - \lambda \bs{e}) \bs{z} 
        = \bs{e}
    \end{align*}
    Hence, for all $\alpha\in\CC$, $\CC \smallsetminus (\alpha - \spc(\bs{a})) \subseteq \CC \smallsetminus \spc(\alpha \bs{e} - \bs{a})$ where $\alpha - \spc(\bs{a}) \coloneqq \{ \alpha - \lambda ~|~ \lambda \in  \spc(\bs{a}) \}$.
    As a consequence, for all $\beta \in \spc(\alpha \bs{e} - \bs{a}) $ there is a $\lambda \in \spc(\bs{a})  $ such that $\beta = \alpha - \lambda$. 

    For the reverse direction, suppose that $\lambda \in \spc(\bs{a})$ and $\alpha - \lambda \notin \spc(\alpha \bs{e} - \bs{a})$. There exists a $\bs{y}$ such that $(\alpha \bs{e} - \bs{a} - (\alpha - \lambda)\bs{e})\bs{y}=\bs{e}$. So, $\bs{y}^{-1}=\alpha \bs{e} - \bs{a} - (\alpha - \lambda)\bs{e} = (\lambda \bs{e} - \bs{a}$, so $(a-\lambda\bs{e})^{-1} = -\bs{y}$. This shows that $a-\lambda\bs{e}$ is invertible, so $\lambda \notin \spc(\bs{a})$ which is a contradiction.
    \begingroup
    {
    }
    \endgroup
\end{proof}

\begin{proof}[Proof of~\Cref{lem:pos.geq.norm.bound}]
    Let $\bs{a},\bs{b} \in \alg{A}$.
    Assume that  $\bs{b} \geq_{\alg{A}} 0 $ then it is self adjoint, and following from \Cref{lem:spcrad.exists.norms} $\GNorm{\bs{a}}_{\alg{A}} - \GNorm{\bs{b}}_{\alg{A}} = \GNorm{\bs{a}}_{\alg{A}} - r(\bs{b}) $
    holds for all $\bs{a} \in \alg{A}$. 
    That is, $\GNorm{\bs{a}}_{\alg{A}} - \GNorm{\bs{b}}_{\alg{A}} \geq \GNorm{\bs{a}}_{\alg{A}} - \beta$ for all $\beta \in \spc(\bs{b})$.
    By \Cref{lem:helper.for.geqnormbound}   $\beta \in \spc(\bs{b})$ if and only if $\GNorm{\bs{a}}_{\alg{A}} - \beta \in \spc(\GNorm{\bs{a}}_{\alg{A}}\bs{e} -\bs{b}) $ thus
    \begin{equation}\label{eq:genupbound.shift}
        \forall \beta' \in \spc(\GNorm{\bs{a}}_{\alg{A}}\bs{e} -\bs{b}):~~~ \GNorm{\bs{a}}_{\alg{A}} - \GNorm{\bs{b}}_{\alg{A}} \geq \beta'
    \end{equation}

    If  $\bs{a} \geq_{\alg{A}} \bs{b}$ then $\bs{a}$ is also non-negative, in particular $\bs{a} = \bs{a}^*$ and it clearly holds that $\GNorm{\bs{a}}_{\alg{A}} \bs{e} - \bs{a}$ is self adjoint.
    Let $ \lambda \in \spc(\GNorm{\bs{a}}_{\alg{A}} \bs{e} - \bs{a})$. 
    By \Cref{lem:helper.for.geqnormbound},  there exists $\lambda' \in \spc(\bs{a})$ such that $\lambda = \GNorm{\bs{a}}_{\alg{A}} - \lambda'$.
    Now, $\lambda = \GNorm{\bs{a}}_{\alg{A}} - \lambda' \geq \GNorm{\bs{a}}_{\alg{A}} - \sup_{\lambda' \in \spc(\bs{a})} \lambda' \geq 0$ where the last inequality follows from \Cref{lem:spcrad.exists.norms}, so we have that $\lambda \geq 0$. 
    Therefore $\spc(\GNorm{\bs{a}}_{\alg{A}} \bs{e} - \bs{a}) $ is contained in the non-negative half of the real line, and we have that $\GNorm{\bs{a}}_{\alg{A}} \bs{e} - \bs{a} \geq_{\alg{A}} 0$.
    Combined with ~\Cref{eq:genupbound.shift}, the above implies that $\GNorm{\bs{a}}_{\alg{A}} - \GNorm{\bs{b}}_{\alg{A}} \geq 0$.
    \begingroup{
    }\endgroup
    \begingroup
    {
    
    %
    }
    \endgroup
\end{proof}

Borrowing notation of \cite{Conway2007}, we denote the collection of positive elements in a C*-algebra $\alg{A}$ by $\alg{A}_+$. We say that $\bs{x} \in \alg{A}$ is a upper (respectively, lower) bound on $A \subseteq \alg{A}$ if $\bs{a}\leq\bs{x}$ (respectively, $\bs{x}\leq \bs{a}$) for all $\bs{a} \in A$. It is easy to verify that if $A \subset \alg{A}_+$ be a finite set of non-negative elements, then $\sum_{\bs{a}\in A} a$ is an upper bound $A$.

\begin{lemma}[Square Roots in \texorpdfstring{$\ell_\infty$}{ell infinity}]\label{def:ellinf.sqrt}
Let $\bs{a} \in \ellinf$ be a self-adjoint element. 
Then there exists a $\bs{b} \in \ellinf$ such that $\bs{b}^2 =\bs{b} \odot \bs{b} = \bs{a}$. 
Moreover, if $\bs{a} \geq_{\ellinf} 0$, then there is a unique, non-negative element $\sqrt{\bs{a}} \in \ellinf  $ such that $\bs{a} = \sqrt{\bs{a}}\odot \sqrt{\bs{a}}  = \sqrt{\bs{a}}^2 $ defined as the principal \textbf{square-root} of $\bs{a}$.
\end{lemma}
\begin{proof}
    Let $\bs{a} \in \ellinf$ and suppose that $\bs{a} = \bs{a}^*$, then $a_k \in \RR$ for all $k$.
    For all $k \in \ZZ$,
    if $a_k \geq 0$ we have $a_k = b_k^2$ where $b_k = \sqrt{a_k} \geq 0$ is the unique non-negative number such that $a_ = b_k^2$.  
    If $a_k < 0$ then  $a_k = - |a_k|$ hence we write $a_k = b_k^2$ where  $b_k = \pm i \sqrt{|a_k|}$.
    Clearly $\bs{a} = \bs{b} \odot \bs{b}$, and in case $\bs{a} \geq_{\ellinf} 0 $ we have $a_k \geq 0$ for all $k$ hence $\sqrt{\bs{a}} \coloneqq \bs{b}   \geq 0$ is the unique non-negative $\ellinf$ element such that $\bs{a} = \sqrt{\bs{a}} \odot \sqrt{\bs{a}}$. 
\end{proof}
\begin{corollary}[Square Roots in \texorpdfstring{$\cHs$}{H_*}]\label{def:cHs.sqrt}
    An element $\bs{x} \in \cHs$ is self-adjoint, if and only if $\bs{x} = \bs{y} \ff \bs{y}$ for some $\bs{y} \in \cHs$.
    If $\bs{x} \geq_{\cHs} 0$ then there exists a unique $\sqrt{\bs{x}} \geq_{\cHs} 0$ such that $\bs{x} = \sqrt{\bs{x}} \ff \sqrt{\bs{x}} = \sqrt{\bs{x}}^2$. 
\end{corollary}
\begin{proof}
    Trivial - follows *-iso $(\ellinf,\odot) \cong (\cHs, \ff)$.
\end{proof}

\begin{definition}[\texorpdfstring{$\cHs$}{H_*}-Absolute Value]\label{def:absval.chs.appendix}
    Let $\bs{x} \in \cHs$, the $\cHs$-absolute value of $\bs{x}$ is defined as $|\bs{x}| \coloneqq \sqrt{\bs{x}^* \ff \bs{x}}$. 
\end{definition}

\begin{lemma}\label{lem:absval.chs.ex.un.appendix}
    Let $\bs{x} \in \cHs$, then $| \bs{x} |$ is well defined and in $\cHs$, and we have $|\bs{x}|^2= \bs{x}^* \ff \bs{x}$.
\end{lemma}
\begin{proof}
    Following from \cite[Chapter VIII, Theorem 3.6]{Conway2007}, we have $\bs{x}^* \ff \bs{x} \geq_{\cHs} 0$.  \Cref{def:cHs.sqrt} implies that the square root of $\bs{x}^* \ff \bs{x}$ exists and in $\cHs$, and that $|\bs{x}|^2= \sqrt{\bs{x}^* \ff \bs{x}}^2= \bs{x}^* \ff \bs{x}$.
\end{proof}

\subsection{Hilbert C*-module structure}\label{sec:cstar.module}

First, we recall some basic definition from abstract algebra
\begin{definition}[Module over ring]\label{def:module.over.ring}
    Let $(\alg{R}, \odot)$ be a ring.  
    A \textbf{left $\bs{\alg{R}}$-module} $E$ is an Abelian group $(E,+)$ with additional operation $\cdot \colon \alg{R} \xx E \to E$ such that for all $\bs{q},\bs{p} \in \alg{R}$ and $\mat{X}, \mat{Y} \in E$ it holds that
    \begin{enumerate}
        \item $\bs{q}\cdot (\mat{X} + \mat{Y}) = \bs{q} \cdot \mat{X} + \bs{q} \cdot \mat{Y} $ \label{item:module.rmul.distributive.over.madd}
        \item $(\bs{q} + \bs{p}) \cdot \mat{X} = \bs{q}\cdot \mat{X} + \bs{p} \cdot \mat{X}$ \label{item:module.rmul.distributive.over.radd}
        \item $(\bs{q} \odot \bs{p}) \cdot \mat{X} = \bs{q}\cdot(  \bs{p} \cdot \mat{X})$ \label{item:module.associative.rmul}
        \item If $\alg{R}$ has a unity $\bs{1}$, then $\bs{1} \cdot \mat{X} = \mat{X}$ \label{item:module.unitary.rmul}
    \end{enumerate}

    A subset $B$ of an $\alg{R}$ module $E$ is called a \textbf{basis} for $E$ if it is an $\alg{R}$-linearly independent generating set:
    \begin{enumerate}
    \item Linearly independent: For every finite set $\bs{B}_1 , \ldots , \mat{B}_n \in B$, we have $\sum_{j=1}^n \bs{a}_j \cdot \mat{B}_j = 0$ if and only if $\bs{a}_1 = \cdots = \bs{a}_n = 0$.
    \item Generating: For all $\mat{X} \in E$, we have $\mat{X} = \sum_{\mat{B}_\beta \in B} \bs{a}_{\beta} \cdot  \mat{B}_\beta$ where the set of indices $\beta$ such that $\bs{a}_\beta \neq 0$ is finite. 
    \end{enumerate}
    An $\alg{R}$-module $E$ is said to be \textbf{free} if it contains a basis.
    A ring $\alg{R}$ has \textbf{invariant basis number (IBN) property} if for all $m,n\in \NN$ it holds that $\alg{R}^m \cong \alg{R}^n \iff m=n$. Note that any commutative ring statisifies the IBN property. 
    Suppose that $\alg{R}$ has the IBN property, then the \textbf{rank} of an $\alg{R}$ module $E$ is the cardinality of any basis of $E$ (all of which are of the same cardinality).
\end{definition}

An elementary result from linear algebra states there is an isomorphism between linear transformations and matrices of the underlying field. This also holds for modules over commutative rings. The proof is essentially the same as for matrices. 
\begin{lemma}\label{lem:Rlin.iff.matmul}
    Let $(\alg{R}, \odot)$ be a commutative ring and $m,p$ be positive integers.
    A mapping $T$ between the (free) $\alg{R}$-modules $\alg{R}^p$ and $\alg{R}^m$ is an $\alg{R}$-module homomorphism ($\alg{R}$-linear) if and only if there exists an $\alg{R}$-valued matrix $\tens{T} \in \alg{R}^{m \xx p}$ with entries $\{ \bs{t}_{hj} \}_{(h,j)\in[m]\xx[p]}$ such that 
    \begin{align*}
        T(\mat{X}) 
        &= \tens{T} \mat{X} 
        \coloneqq {\sum}_{h=1}^m {\sum}_{j=1}^p \bs{t}_{hj} \odot \bs{x}_j \cdot \mat{E}_{(m)}^{(h)}
    \end{align*}
    for all $\mat{X} \in \alg{R}^p$, where  $\mat{E}_{(M)}^{(K)}$ denotes the $K$-th element of the standard basis in $\alg{R}^M$.
    {That is, \newline $\Hom_{\alg{R}}(\alg{R}^p,\alg{R}^m) {\cong} \alg{R}^{m \xx p}$.}
\end{lemma}

\subsubsection{The module of quasitubal slices}
For our purposes, all underlying rings are commutative and all modules are both left and right modules unless stated otherwise. 

\begin{lemma}
    The set $\cHs^p$ of $p$-tuples $\mat{X} = (\bs{x}_1, \ldots, \bs{x}_p) $ where $\bs{x}_k \in \cHs$ for all $k \in [p]$, together with the operations $\mat{X} + \mat{Y} = (\bs{x}_1 + \bs{y}_1, \ldots, \bs{x}_p + \bs{y}_p)$ and $\bs{q} \ff \mat{X} = (\bs{q} \ff \bs{x}_1, \ldots, \bs{q} \ff \bs{x}_p)$ for all $\matX,\matY \in \cHs^p$ and $\bs{q} \in \cHs$,
    is a free $\cHs$-module of rank $p$. 
\end{lemma}
\begin{proof}
    It immediately follows from the definition of $\ff$-multiplication and the properties of $\cHs$ that $\cHs^p$ is an $\cHs$-module.
    Consider the set $\{ \mat{\Phi}_{(p)}^{(k)} \}_{k = 1}^{p}$ where 
    \begin{equation}\label{eq:cHs.std.bases}
        \mat{\Phi}_{(p)}^{(k)} = (\bs{0}, \ldots, \bs{0}, \underset{k\textnormal{'th coordinate}}{\bs{e}}, \bs{0}, \ldots, \bs{0})
    \end{equation}
    then clearly $\{ \mat{\Phi}_{(p)}^{(k)} \}_{k = 1}^{p} \subset \cHs^p$ is a generating set for $\cHs^p$ that is $\cHs$-linearly independent, therefore, a basis.
    Since $\cHs$ is a nonzero, commutative ring, it has the IBN property, thus all bases of $\cHs^p$ have the same cardinality - $p$, and $\cHs^p$ is a free module of rank $p$ over $\cHs$.
\end{proof}
In what follows, we will denote the $j$-th standard basis element of \smash{$\ellinf^{p}$} by \smash{$\mat{E}_{(p)}^{(j)}$} or \smash{$\mat{E}^{(j)}$} in case the module's rank is clear from context. 
Correspondingly, we denote by \smash{$\bs{\Phi}_{(p)}^{(j)} \in \cHs^p$}  the $j$-th element in the standard basis for \smash{$\cHs^p$}, and we have, $$\bs{\Phi}_{(p)}^{(j)} \coloneqq \matE_{(p)}^{(j)} \xx_3 \uppi^{-1}.$$  We write  $\bs{\Phi}^{(j)}$ if the dimension is clear from the context.

\begin{definition}[Pre Hilbert Module]\label{def:pre.hilbert.module.PHM}
    Let $\alg{A}$ be a C*-algebra over $\FFF$ with a norm denoted by $\GNorm{\cdot}_{\alg{A}}$.
    A \textbf{pre-Hilbert $\alg{A}$-module}, or, pre-Hilbert module (PHM) over $\alg{A}$, is a complex vector space $M$ that is also an $\alg{A}$-module, equipped with a mapping
    $\dotmp{\cdot,\cdot} \colon M \xx M \to \alg{A}$, called `$\alg{A}$\textbf{-valued inner product}' (or $\alg{A}$-inner-product) if for all $\bs{a},\bs{b} \in \alg{A}$ and $\mat{X},\mat{Y},\mat{Z} \in M$
    \begin{enumerate}
        \item $\dotmp{\mat{X}, \bs{a}  \mat{Y} + \bs{b}  \mat{Z}} = \bs{a} \dotmp{\mat{X}, \mat{Y}} + \bs{b} \dotmp{\mat{X}, \mat{Z}} $ \label{item:avalinnerp.2nd.arg.lin}
        \item $\dotmp{\mat{X},\mat{Y}} = \dotmp{\mat{Y},\mat{X}}^*$ \label{item:avalinnerp.adj.sym}
        \item $\dotmp{\mat{X},\mat{X}} \geq_{\alg{A}} 0$ for all $\mat{X} \in M$  and $\dotmp{\mat{X},\mat{X}} = 0 \iff \mat{X} = 0$.  \label{item:avalinnerp.posdef}
    \end{enumerate}

    Let $\alg{A}$ be a C*-algebra and $M$ be a PHM over $\alg{A}$ with inner product $\dotmp{\cdot , \cdot} \colon M \xx M \to \alg{A}$.
    The \textbf{$\alg{A}$-absolute value} of any $\mat{X} \in M$, is defined as 
    \begin{equation}\label{def:PHM.Aabs}
        | \mat{X} |_M  \coloneqq \sqrt{\dotmp{\mat{X},\mat{X}}}
    \end{equation}
    that is,  the unique, non-negative (hence self adjoint) element $| \mat{X} |_M \in \alg{A}$  such that $| \mat{X} |_M \cdot | \mat{X} |_M^* = | \mat{X} |_M^2 = \dotmp{\mat{X},\mat{X}}$, and the \textbf{real $M$-norm} of $\mat{X}$ is 
    \begin{equation}\label{def:PHM.Rnorm}
         \GNorm{\mat{X}}_{M} \coloneqq  \GNorm{| \mat{X} |_M}_{\alg{A}}
    \end{equation}
    Note that the identity $\GNorm{\mat{X}}_{M} = \GNorm{\dotmp{\mat{X},\mat{X}}}_{\alg{A}}^{1/2}$ holds since $\GNormS{| \mat{X} |_M}_{\alg{A}} = \GNorm{| \mat{X} |_M^2}_{\alg{A}} = \GNorm{\dotmp{\mat{X},\mat{X}}}_{\alg{A}}$.
\end{definition}

\begin{lemma}[Cauchy-Schwartz in PHM \texorpdfstring{\cite[Proposition 1.1]{Lance1995}}{\cite{Lance1995}}]\label{lem:PHM.cs.ineq}
    Let $\alg{A}$ be a C*-algebra and $M$ be a PHM over $\alg{A}$ with inner product $\dotmp{\cdot , \cdot} \colon M \xx M \to \alg{A}$.
    Then 
    $$
    |\dotmp{\mat{X},\mat{Y}}|^2 \leq_{\alg{A}} \GNormS{\mat{X}}_M \dotmp{\mat{Y}, \mat{Y}}
    $$
    holds for all $\mat{X},\mat{Y} \in M$.
\end{lemma}




By~\Cref{def:pre.hilbert.module.PHM} and~\Cref{def:PHM.Rnorm}, any PHM $M$ is a normed complex vector space. We have the following:
\begin{definition}[Hilbert C* Module]\label{def:Hcstar.module}
    A PHM $M$ over a C*-algebra $\alg{A}$, equipped with the $\alg{A}$-inner-product $\dotmp{\cdot,\cdot} \colon M \xx M \to \alg{A}$ is a \textbf{Hilbert C*-module} over $\alg{A}$ if $M$ is complete w.r.t the the real $M$-norm. 
\end{definition}

\begin{theorem}\label{thm:chsp.is.cstar.hilbert.module}
    The set $\cHs^p$ with the mapping $\dotmp{\mat{X}, \mat{Y}} \coloneqq \mat{X}^* \ff \mat{Y}$ is a Hilbert C*-module over $\cHs$. 
\end{theorem}
\begin{proof}
    To show that $\dotmp{\cdot, \cdot}$ is an $\cHs$-inner-product, 
    note that~\Cref{item:avalinnerp.2nd.arg.lin,item:avalinnerp.adj.sym} clearly hold.
    
    Let $\mat{X} \in \cHs^p$, then
     $\mat{X}^* \ff \mat{X} = \sum_{j=1}^p \bs{x}_j^* \ff \bs{x}_j$ is an upper bound on $\{ \bs{x}_j^* \ff \bs{x}_j \}_{j=1}^p$ and, in particular, a non-negative element. 
    If $\mat{X}^* \ff \mat{X} = 0$ then $0 \geq_{\cHs} \bs{x}_j^* \ff \bs{x}_j$ therefore $\bs{x}_j^* \ff \bs{x}_j = 0$ for all $j$ and it suffices to show that $\bs{x}^* \ff \bs{x} = 0$ if and only if $\bs{x} = 0$.
    Let $\bs{x} \in \cHs$.
    Then by \Cref{lem:pi.star.iso} $|\bs{x}|^2 \coloneqq \bs{x}^* \ff \bs{x} = \uppi^{-1}(\hbs{x})^* \ff \uppi^{-1}(\hbs{x}) = \uppi^{-1}(\hbs{x}^* \odot\hbs{x})$
    is nonzero if and only if $|\hbs{x}|^2 \coloneqq \hbs{x}^* \odot \hbs{x} \in \ell_\infty$ is nonzero.
    Clearly, $|\hbs{x}|^2 = 0$ if an only if $\hbs{x} = 0$, thus $|\bs{x}|^2 = 0$ is equivalent to $\bs{x} = \uppi^{-1}(0) = 0$. 

    Next, we need to show that the space $\cHs^p$ is complete w.r.t $\GNorm{\cdot}_{\cHs^p}$.
    Let $\{\mat{X}^{(j)}\}_{j \in \NN} \subset \cHs^p$ be a Cauchy sequence, then for any $n \in \NN$, there exists $N = N_n \in \NN$ such that for all $m > N$ and $k \in \NN$ we have $\GNorm{\mat{X}^{(m)} - \mat{X}^{(m+k)} }_{\cHs^p} < n^{-1/2}$. 
    By setting $\mat{Y}^{(m,k)} \coloneqq \mat{X}^{(m)} - \mat{X}^{(m+k)}$, and using the definitions in \Cref{def:PHM.Aabs,def:PHM.Rnorm} we get 
    $$
    \GNormS{\mat{Y}^{(m,k)}}_{\cHs^p}
    =
    \cHsNormSd{\sqrt{{\sum}_{j=1}^p {\bs{y}_j^{(m,k)}}^* \ff {\bs{y}_j^{(m,k)}} }} 
    = 
    \cHsNormSd{{\sum}_{j=1}^p {\bs{y}_j^{(m,k)}}^* \ff {\bs{y}_j^{(m,k)}}}
    $$
    where the last equality follows from the C*-identity. 
    Clearly, $\sum_{j=1}^p {\bs{y}_j^{(m,k)}}^* \ff {\bs{y}_j^{(m,k)}} \geq_{\cHs} {\bs{y}_r^{(m,k)}}^* \ff {\bs{y}_r^{(m,k)}}$ for all $r =1,\ldots,p$, and by \Cref{lem:pos.geq.norm.bound} we obtain that
    \begin{equation}\label{eq:opnorm.cauchy.bound}
    \cHsNormd{{\sum}_{j=1}^p {\bs{y}_j^{(m,k)}}^* \ff {\bs{y}_j^{(m,k)}}} \geq  \cHsNorm{{\bs{y}_r^{(m,k)}}^* \ff {\bs{y}_r^{(m,k)}}}
    \end{equation}
    By construction,
    $$
    \cHsNormd{{\sum}_{j=1}^p {\bs{y}_j^{(m,k)}}^* \ff {\bs{y}_j^{(m,k)}}} 
    = \GNormS{\mat{X}^{(m)} - \mat{X}^{(m+k)} }_{\cHs^p}  
    $$
    and 
    $$
    {\bs{y}_r^{(m,k)}}^* \ff {\bs{y}_r^{(m,k)}} 
    = 
    (\bs{x}_r^{(m)} - \bs{x}_r^{(m+k)})^* \ff (\bs{x}_r^{(m)} - \bs{x}_r^{(m+k)})
    $$
    thus \Cref{eq:opnorm.cauchy.bound} becomes 
    $\cHsNormS{\bs{x}_r^{(m)} - \bs{x}_r^{(m+k)}} \leq \GNormS{\mat{X}^{(m)} - \mat{X}^{(m+k)} }_{\cHs^p} \leq  n^{-1}$ for all $r \in [p]$, that is, $\{\bs{x}_r^{(m)} \}_{m \in \NN}$ is a Cauchy sequence with respect to $\cHsNorm{\cdot}$ for all $r \in [p]$.
    Combined with the fact that $\cHs$ is C*-algebra (in particular - Banach space), the above implies that for all $r \in [p]$ there exists $\bs{x}_r \in \cHs$ such that $\lim_{m \to \infty} \bs{x}_r^{(m)} = \bs{x}_r$. 
    Define $\mat{X} \in \cHs^p$ as the quasitubal slice whose $r$th entry is $\bs{x}_r$ and have that $\lim_{j \to \infty} \GNorm{\mat{X}^{(j)} - \mat{X}}_{\cHs^p} = 0$ as an immediate consequence. 
\end{proof}

\if00
\begingroup{
\subsubsection{The module of quasitubal tensors}
The following construction is meant to generalize \Cref{thm:chsp.is.cstar.hilbert.module} and establish a Hilbert C*-module structure over the space of $m \xx p$ quasitubal tensors. 
We begin with a matrix mimetic generalization of the trace operator. 
\begin{definition}[Quasitubal Trace]\label{def:quasitubal.tensor.trace}
    Let $\alg{R}$ be a C*-algebra. 
    The $\alg{R}$-trace of $\tens{X} \in \alg{R}^{p \xx p} $ is defined by $\trace{\tX} = {\sum}_{j=1}^{p} \bs{x}_{jj} \in \alg{R}$.
    In particular, the quasitubal trace of a quasitubal tensor $\tX \in \cHs^{p \xx p}$ is the $\cHs$-trace of $\tX$.
\end{definition}
Noting that for any commutative C*-algebra $\alg{A}$, the set $\alg{A}^{p \xx p}$ with the obvious (elementwise) addition and scalar multiplication by $\alg{A}$ elements is a free $\alg{A}$-module, and observe the following result:
\begin{lemma}\label{lemma:trace.is.module.hom}
    The operator $\trc \colon \alg{A}^{p \xx p} \to \alg{A}$ in \Cref{def:quasitubal.tensor.trace} is an $\alg{A}$-linear mapping, hence $\trc \in \operatorname{Hom}(\alg{A}^{p \xx p},\alg{A}) $ and in particular $\trc \in \alg{A}^{p \xx p}_*$. 
\end{lemma}
\begin{proof}
    Let $\tX, \tY \in \alg{A}^{p \xx p}$ and $\bs{a}, \bs{b} \in \alg{A}$ then 
    \begin{align*}
        \trace{\bs{a} \tX + \bs{b} \tY}
        &= {\sum}_{j=1}^p [\bs{a} \tX + \bs{b} \tY]_{jj} 
        = {\sum}_{j=1}^p \bs{a} \bs{x}_{jj} + \bs{b} \bs{y}_{jj} \\
        &=  \bs{a} {\sum}_{j=1}^p \bs{x}_{jj} + \bs{b} {\sum}_{j=1}^p \bs{y}_{jj} 
        = \bs{a}\trace{\tX} + \bs{b} \trace{\tY}
    \end{align*}
\end{proof}
\newcommand{\cHsmp}{\cHs^{m \xx p}}
\newcommand{\dotchsmp}[1]{\ensuremath{\dotmp{#1}_{\cHsmp}}}
\newcommand{\dotlmp}[1]{\ensuremath{\dotmp{#1}_{\ell_{\infty}^{m \xx p}}}}

Next, consider the $\cHs$-module $\cHs^{m \xx p}$ and define 
\begin{equation}\label{eq:trc.innerprod.mxxp}
    \dotchsmp{\tX , \tY} \coloneqq \trace{\tX^* \ff \tY}
\end{equation}
\begin{theorem}\label{thm:chs.mp.is.hilbert.cs.module}
    The $\cHs$-module  $\cHs^{m \xx p}$, together with  $\dotchsmp{\cdot , \cdot}$ in \Cref{eq:trc.innerprod.mxxp} is a  Hilbert C*-module over $\cHs$. 
    As a consequence, $\ell_{\infty}^{m \xx p}$ equipped with $\dotlmp{\thX , \thY} \coloneqq \trace{ {\thX^{\bf{H}}} \triangle \thY}$ is a Hilbert C*-module over $\ellinf$. 
\end{theorem}
\begin{proof}
    Trivial due to the isometric isomorphism between $\cHs^{m \xx p}$ and $\cHs^{mp}$ induced by flattening (vectorization), which respects the inner product $\dotchsmp{\cdot , \cdot}$.
\end{proof}
\begin{remark}\label{rem:chs.mp.absval.module}
    According to \Cref{def:PHM.Aabs}, the $\cHs$-absolute value induced by $\dotchsmp{\cdot , \cdot}$ is 
    $$
    |\tX|_{\cHs^{m \xx p}} \coloneqq \sqrt{\dotchsmp{\tX , \tX}} = \sqrt{\trace{\tX^* \ff \tX}}
    $$
\end{remark}
The next result establishes the relations between the adjoint mapping to $T$ and the $\ff$ conjugate transpose of the quasitubal tensor $\tens{T}$ that represents it.
\begin{lemma}\label{lemma:adjoint.map.and.transpose}
    For any $\cHs$-linear mapping $T\colon \cHs^p \to \cHs^m$, there exists a unique $\cHs$-linear mapping $T^* \colon \cHs^m \to \cHs^p$ such that
    $$
    \forall \mat{X} \in \cHs^p ,\mat{Y}\in \cHs^m :~~ 
    \dotmp{T(\mat{X}), \mat{Y}}_{\cHs^m} 
    = \dotmp{\mat{X}, T^*(\mat{Y})}_{\cHs^p}
    $$
    that is, the adjoint operator to $T$.
    Furthermore, let $\tens{T} \in \cHs^{m \xx p}$ denote the quasitubal tensor such that $T(\mat{X}) = \tens{T} \ff \mat{X}$ for all $\mat{X} \in \cHs^p$, then for all $\mat{Y} \in \cHs^m$ it holds that $T^*(\mat{Y}) = \tens{T}^* \ff \mat{Y}$ where $\tens{T}^*$ is the $\ff$-conjugate transpose of $\tens{T}$ (\Cref{eq:def.ff.conj.transpose}).
\end{lemma}
\begin{proof}
    By \Cref{lem:Rlin.iff.matmul}, there exists $\tens{T} \in \cHs^{m \xx p}$
    such that $T(\mat{X}) = \tens{T} \ff \matX$ 
    for all $\mat{X} \in \cHs^p$.
    Write 
    \begin{align*}
        \dotmp{T(\mat{X}), \mat{Y}}_{\cHs^m} 
        &= \dotmp{\tens{T} \ff \mat{X}, \mat{Y}}_{\cHs^m} \\
        &= (\tens{T} \ff \mat{X})^* \ff \mat{Y} \\
        &= \mat{X}^* \ff (\tens{T}^* \ff \mat{Y}) 
        = \dotmp{\mat{X}, \tens{T}^* \ff \mat{Y}}_{\cHs^p}
    \end{align*}

    By \Cref{lem:Rlin.iff.matmul}, the mapping defined by $T^* (\mat{Y}) = \tens{T}^* \ff \mat{Y}$ for all $\mat{Y} \in \cHs^m$ is a $\cHs$-linear mapping from $\cHs^m$ to $\cHs^p$.
    For uniqueness, suppose that  $T' \colon \cHs^m \mapsto \cHs^p$ is a $\cHs$-linear map such that $
    \dotmp{T(\mat{X}), \mat{Y}}_{\cHs^m} 
    = \dotmp{\mat{X}, T'(\mat{Y})}_{\cHs^p}$ for all $\mat{X} \in \cHs^p $ and $\mat{Y}\in \cHs^m$.
    Let $\tens{T}'$ denote the quasitubal tensor representing $T'$, then for all $\mat{X} \in \cHs^p, \mat{Y}\in \cHs^m$ it holds that 
    $\dotmp{\mat{X}, (\tens{T} - \tens{T}') \ff \mat{Y}}_{\cHs^p} = 0$ hence $\tens{T} = \tens{T}'$ and the adjoint is uniquely defined.
\end{proof}

Note that $ \bs{\Phi}^{(j)}_{(m)} \ff \bs{\Phi}^{(h)}_{(p)}{}^*$ is the $m \xx p$ quasitubal tensor whose $j,h$ entry is the unit element $\bs{e}$ and zero for all other entries.
For any $\tX \in \cHs^{m \xx p}$ we have $\tX = \sum_{j=1}^m \sum_{j'=1}^p \bs{x}_{j,j'} \ff \bs{\Phi}^{(j)}_{(m)} \ff \bs{\Phi}^{(j')}_{(p)}{}^*$, therefore $\{ \bs{\Phi}^{(j)}_{(m)} \ff \bs{\Phi}^{(j')}_{(p)}{}^* | j \in [m] ,j' \in [p] \}$ is a spanning set for $\cHs^{m \xx p}$, and since the modules $\cHs^{m \xx p}$ and $\cHs^{mp}$ are isomorphic they have equal ranks that is $mp$, i.e., the spanning set is a minimal one, hence a basis. 

Observe that given any bases \smash{$\{ \mat{A}^{(j)} \}_{j=1}^m \subset \cHs^m, \{ \mat{B}^{(j)} \}_{j=1}^p \subset \cHs^p$}, 
we can write $\bs{\Phi}^{(j)}_{(m)} = \sum_{j'=1}^m \dotmp{\mat{A}^{(j')},\bs{\Phi}^{(j)}_{(m)}} \ff \mat{A}^{(j')} $ and 
$\bs{\Phi}^{(h)}_{(p)} = \sum_{h'=1}^p \dotmp{\mat{B}^{(h')},\bs{\Phi}^{(h)}_{(p)}} \ff \mat{B}^{(h')} $.
Thus, 
$$\bs{\Phi}^{(j)}_{(m)} \ff \bs{\Phi}^{(h)}_{(p)}{}^* = \sum_{j'=1}^m  \sum_{h'=1}^p \dotmp{\mat{A}^{(j')},\bs{\Phi}^{(j)}_{(m)}}  \ff \dotmp{\bs{\Phi}^{(h)}_{(p)}, \mat{B}^{(h')}} \ff \mat{A}^{(j')} \ff \mat{B}^{(h')}{}^*$$
therefore,
the set \smash{$\{ \tens{C}^{(j')}_{(j)} ~|~ j\in [m],j'\in[p] \}$} where $\tens{C}^{(j')}_{(j)} = {\matA^{(j)}} \ff {\matB^{(j')}}^*$ is a basis for $\cHs^{m \xx p}$.
Importantly, if $\{ \mat{A}^{(j)} \}_{j=1}^m$ and $\{ \mat{B}^{(j)} \}_{j=1}^p$ are unit normalized, $\ff$-orthogonal bases for $\cHs^m$ and $\cHs^p$ respectively, i.e., $\dotmp{\mat{A}^{(j)}, \mat{A}^{(j')}}_{\cHs^m} = \delta_{jj'} \bs{e}, \dotmp{\mat{B}^{(h)}, \mat{B}^{(h')}}_{\cHs^p} = \delta_{hh'} \bs{e}$ , then 
\begin{align*}
    \dotchsmp{\tens{C}^{(h)}_{(j)}, \tens{C}^{(h')}_{(j')}} 
    &= \trace{{\matB^{(j)}} \ff {\matA^{(h)}}^* \ff \matA^{(h')} \ff {\matB^{(j')}}^* } 
    = \delta_{h,h'} \trace{{\matB^{(j)}} \ff {\matB^{(j')}}^* }
\end{align*}
and since $\trace{{\matB^{(j)}} \ff {\matB^{(j')}}^* } = \sum_{k=1}^p {\matB^{(j')}_k}^* \ff \matB^{(j)}_k = {\matB^{(j')}}^* \ff \matB^{(j)} $, we have \smash{$\trace{{\matB^{(j)}} \ff {\matB^{(j')}}^* } = \delta_{j,j'} \bs{e}$} so {$ \dotchsmp{\tens{C}^{(h)}_{(j)}, \tens{C}^{(h')}_{(j')}}  = \delta_{h,h'} \delta_{j,j'} \bs{e} $}, that is, {$\{ \tens{C}^{(j')}_{(j)} \}$} is a unit normalized, $\ff$-orthogonal basis for $\cHsmp$.
Furthermore, if $\{ \tens{C}^{(h)}_{(j)} \}_{h \in [m],j\in[p]}$ is a unit normalized, $\ff$-orthogonal basis for the module $\cHsmp$, then $\{ \bs{\phi}^{(k)} \ff \tens{C}^{(h)}_{(j)} \}_{h \in [m],j\in[p],k \in \ZZ}$ is an orthonormal basis for the Hilbert space $\cH^{m \xx p}$ whose inner product is 
\begin{equation}\label{eq:def:chmp.hilbert.inner.product}
    \dotph{\tX, \tY} \coloneqq \dotp{\tX, \tY}_{\cH^{m \xx p}} 
    \coloneqq \sum_{j=1}^p \dotph{\tX_{:,j}, \tY_{:,j} } 
    =  \sum_{j=1}^p \sum_{h=1}^m \dotph{\bs{x}_{h,j}, \bs{y}_{h,j}}
\end{equation}
which clearly induces the norm defined in \Cref{def:quasitubal.tensor.norm}.
To see this, note that 
\begin{align*}
    \dotchsmp{\bs{\phi}^{(k)} \ff \tens{C}^{(h)}_{(j)} , \bs{\phi}^{(k')} \ff \tens{C}^{(h')}_{(j')} } 
    &= \bs{\phi}^{(k)} \ff \bs{\phi}^{(k')} \ff \dotchsmp{ \tens{C}^{(h)}_{(j)} ,  \tens{C}^{(h')}_{(j')} }  \\
    &= \delta_{kk'} \bs{\phi}^{(k)} \ff (\delta_{hh'} \delta_{jj'} \bs{e}) \\
    &= \delta_{kk'}  \delta_{hh'} \delta_{jj'} \bs{\phi}^{(k)} 
\end{align*}
Hence 
\begin{align*}
    \dotp{\bs{\phi}^{(k)} \ff \tens{C}^{(h)}_{(j)} , \bs{\phi}^{(k')} \ff \tens{C}^{(h')}_{(j')}}_{\cH^{m \xx p}}
    &= \sum_{t=1}^p \dotph{\bs{\phi}^{(k)} \ff [\tens{C}^{(h)}_{(j)}]_{:,t} , \bs{\phi}^{(k')} \ff [\tens{C}^{(h')}_{(j')}]_{:,t}} \\
    &=  \dotph{\bs{\phi}^{(k)} \ff  , \bs{\phi}^{(k')} \ff \sum_{t=1}^p ([\tens{C}^{(h)}_{(j)}]_{:,t})^* \ff [\tens{C}^{(h')}_{(j')}]_{:,t}} \\
    &= \dotph{\bs{\phi}^{(k)} \ff  , \bs{\phi}^{(k')} \ff \trace{(\tens{C}^{(h)}_{(j)})^* \ff \tens{C}^{(h')}_{(j')} } } \\
    &= \dotph{\bs{\phi}^{(k)} \ff  , \bs{\phi}^{(k')} \ff (\delta_{jj'} \delta_{hh'} \bs{e})}\\
    &= \delta_{jj'} \delta_{hh'} \dotph{\bs{\phi}^{(k)}, \bs{\phi}^{(k')} }\\
    &= \delta_{jj'} \delta_{hh'} \delta_{kk'}
\end{align*}

In light of the definition in \Cref{eq:def:chmp.hilbert.inner.product} for the inner product on the Hilbert space $\cH^{m \xx p}$ which covers the particular case of $\cH^p$, we revisit \Cref{lemma:adjoint.map.and.transpose} with the following observation.
\begin{lemma}\label{lemma:adjoint.map.and.transpose2}
    Let $T \colon \cHs^{p} \to \cHs^m$ be a $\cHs$-linear mapping, and denote by $\tens{T} \in \cHs^{m \xx p}$  the quasitubal tensor such that $T(\mat{X}) = \tens{T} \ff \mat{X}$ for all $\mat{X} \in \cHs^{p}$.
    Then $T$ defines (via restriction) a bounded linear operator from $\cH^p$ to $\cH^m$ and we have 
    $$
    \forall \mat{X} \in \cH^p, \mat{Y} \in \cH^m,: ~~\dotph{\tens{T} \ff \mat{X} , \matY} = \dotph{\mat{X}, \tens{T}^* \ff \mat{Y}} 
    $$
    where the inner product above is given by \Cref{eq:def:chmp.hilbert.inner.product}. 
    That is, the adjoint operator to $T$, is defined by the $\ff$-conjugate transpose of the quasitubal tensor associated with $T$. 
\end{lemma}
\begin{proof}
    \begin{align*}
        \dotph{T(\mat{X}), \mat{Y}} 
        &= \dotph{\tens{T} \ff \mat{X}, \mat{Y}} \\
        &= {\sum}_{h=1}^m \dotph{\tens{T}_{h,:} \ff \mat{X}, \bs{y}_{h}} \\
        &= {\sum}_{h=1}^m   \dotph{{\sum}_{j=1}^p \bs{t}_{hj} \ff \bs{x}_j , \bs{y}_{h}} \\
        &=   {\sum}_{j=1}^p \dotph{ \bs{x}_j , {\sum}_{h=1}^m \bs{t}_{hj}^* \ff \bs{y}_{h}} \\
        &=    \dotph{ {\sum}_{j=1}^p \bs{x}_j \ff \bs{\Phi}_{(p)}^{(j)} , {\sum}_{j'=1}^p  {\sum}_{h=1}^m \bs{t}_{hj'}^* \ff \bs{y}_{h} \ff \bs{\Phi}_{(p)}^{(j')}  } \\
        &= \dotph{ \mat{X}, \tens{T}^* \ff \mat{Y}}
    \end{align*}
\end{proof}


\subsubsection{The C*-algebra of square quasitubal tensors}\label{sec:square.qttensors.Cstar.alg}

\begin{lemma}\label{lem:chs.pp.is.cstar}
    The set $\cHs^{p \xx p}$ together with the $\ff$ product and the operator norm is a non commutative, unital, C*-algebra over $\FFF$. 
\end{lemma}
\begin{proof}
    The fact that $\cHs^{p \xx p}$ is a vector space over $\FFF$, combined with the $\cHs$-linearity of the operation $\tY \mapsto \tX^* \ff \tY$ which makes it also $\FFF$-linear, directly imply that $\cHs^{p \xx p}$ is an algebra.
    Clearly $\tI_p$ is a unit element.
    Let $\tX \in \cHs^{p \xx p}$ and write the qtSVD $\tX = \tU \ff \tS \ff \tV^*$. Then, 
    \begin{align*}
        \opNorm{\tX^* \ff \tX} 
        &= \opNorm{\tS^* \ff \tS} 
        = \max_{j \in [p]} \GNorm{|\bs{s}_j|_{\cHs}^2}_{\cHs} \\
        &= \max_{j \in [p]} \GNormS{|\bs{s}_j|_{\cHs}}_{\cHs} 
        =\max_{j \in [p]} \GNormS{\bs{s}_j}_{\cHs} 
    \end{align*}
    where the last transition follows from the non negativity of $\bs{s}_j$ (\Cref{def:quasitubal.svd}), and the second to last follows the C*-identity for the C*-algebra $\cHs$. 
    By \Cref{cor:qt.best.lowrank.frob},  $\opNorm{ \tX} = \opNorm{\tS} = \max_{j \in [p]} \GNorm{\bs{s}_j}_{\cHs}$, hence $\opNorm{\tX^* \ff \tX} = \opNorm{\tS}^2 = \opNorm{ \tX}^2$.
    Also, note that
    \begin{align*}
        \GNormS{\tX}_{\cHs^{p \xx p}} 
        &= \cHsNormd{|\tX|^2_{\cHs^{p \xx p}}}
        = \cHsNormd{|\tS|^2_{\cHs^{p \xx p}}}
        = \cHsNormd{ \sum_{j=1}^p |\bs{s}_j|^2_{\cHs} }
    \end{align*}
    On one hand, we have $\GNorm{ \sum_{j=1}^p |\bs{s}_j|^2_{\cHs} }_{\cHs}  \geq \GNormS{\bs{s}_1}_{\cHs} = \opNorm{\tX}^2$ and on the other, $\GNorm{ \sum_{j=1}^p |\bs{s}_j|^2_{\cHs} }_{\cHs}  \leq \sum_{j=1}^p  \GNormS{\bs{s}_j}_{\cHs}$. 
     Hence 
    $$
    \opNorm{\tX} \leq \GNorm{\tX}_{\cHs^{p \xx p}} \leq \sqrt{r} \opNorm{\tX}
    $$
    where $r$ denotes the q-rank of $\tX$ under $\ff$. 
    
    As a result, for any sequence $ = \{ \tX^{(n)} \}_{n \in \NN} \subset \cHs^{p \xx p}$ we have that 1) if $\{ \tX^{(n)} \}_{n \in \NN}$  is a Cauchy sequence w.r.t $\opNorm{\cdot}$, then it is also a Cauchy sequence w.r.t $\GNorm{\cdot}_{\cHs^{p \xx p}} $, and 2) if $\{ \tX^{(n)} \}_{n \in \NN}$ converges in $\GNorm{\cdot}_{\cHs^{p \xx p}}$ to limit $\tX$ then also $\lim_{n \to \infty } \opNorm{\tX^{(n)} - \tX} = 0$.
    Combined together with \Cref{thm:chs.mp.is.hilbert.cs.module} we get that any Cauchy sequence in $\cHs^{p \xx p}$ w.r.t $\opNorm{\cdot}$ converges, in $\opNorm{\cdot}$ to a limit in $\cHs^{p \xx p}$. 
\end{proof}

Given $\tX \in \cHs^{m \xx p}$, with q-SVD $\tX = \tU \ff \tS \ff \tV^*$, it is clear that $\bs{a} \in \cHs$ is a singular quasitube of $\tX$, i.e., $\bs{a} = \bs{s}_{j,j}$ for some $j \in [p]$, only if $| \bs{a} |^2$  is a singular quasitube of $\tX^* \ff \tX$.
To see this, we simply write the q-SVD of $\tX^* \ff \tX$, which, turns out to be an `eigendecomposition', $\tX^* \ff \tX = \tV \ff \tS^* \ff \tS \ff \tV^*$, and note that for all $j \in [p]$
\begin{align*}
    [\tS^* \ff \tS]_{j,j} 
    &= {\sum}_{j'=1}^p \bs{s}_{jj'}^* \ff \bs{s}_{j'j} \\
    &= \bs{s}_{jj}^* \ff \bs{s}_{jj} 
\end{align*}
where the last transition is due to that $\tS$ is f-diagonal.
Hence, a singular quasitube of $\tX$ is a (principal) square-root of a singular quasitube of $\tX^* \ff \tX $ (and of $\tX \ff \tX^* \in \cHs^{m \xx m}$ due to the same arguments).
\Cref{lem:chs.pp.is.cstar} tells that $\tX^* \ff \tX$ is an element of a unital C*-algebra, thus admits a spectrum, and we have the following lemma.
\begin{lemma}\label{lem:singvals.in.spc}
    Let $\tA $ be a self-adjoint element in the C*-algebra of square quasitubal tensors $\cHs^{p \xx p}$.
    Then, if $\tA \ff \mat{Z} =\lambda \mat{Z}$ for $\lambda \in \FFF$ and $\mat{Z} \in \cHs^p$, then $\lambda \in \spc(\tA)$. 
\end{lemma}
\begin{proof}
    Write $\tA = \tV \ff \tG \ff \tV^*$. 
    Let $\mat{Z} $ be a nonzero quasitubal slice in $\cHs^p$ and suppose that $\tA \ff \mat{Z} =\lambda \mat{Z}$ for some $\lambda \in \FFF$, then 
    \begin{align*}
        \lambda \mat{Z}
        &= \tV \ff {\sum}_{k \in \ZZ} {\sum}_{j=1}^p \wh{g}_{j,j}(k) \bs{\phi}^{(k)} \ff \bs{\Phi}^{(j)} \ff \bs{\Phi}^{(j)}{}^*  \ff \tV^* \ff \mat{Z} 
    \end{align*}
    Write $\mat{Z} = \tV \ff \mat{Y}$ (where $\mat{Y} = \tV^* \ff \mat{Z} \in \cHs^p$), and have 
    \begin{align*}
        \lambda  \mat{Y}
        &= {\sum}_{k \in \ZZ} {\sum}_{j=1}^p \wh{g}_{j,j}(k) \bs{\phi}^{(k)} \ff \bs{\Phi}^{(j)} \ff \bs{\Phi}^{(j)}{}^*  \ff \mat{Y} \\
        &= {\sum}_{k \in \ZZ} {\sum}_{j=1}^p \wh{g}_{j,j}(k) \wh{y}_{j}(k) \bs{\phi}^{(k)} \ff \bs{\Phi}^{(j)}  
    \end{align*}
    Recall that $\mat{Y} = \sum_{k \in \ZZ } \sum_{j=1}^p \wh{y}_{j}(k) \bs{\phi}^{(k)} \ff \bs{\Phi}^{(j)}$ hence 
    $$
    {\sum}_{k \in \ZZ } {\sum}_{j=1}^p (\lambda - \wh{g}_{j,j}(k))  \wh{y}_{j}(k) \bs{\phi}^{(k)} \ff \bs{\Phi}^{(j)} = 0
    $$
    and since $\bs{\phi}^{(k)} \ff \bs{\Phi}^{(j)}$ are independent, it must hold that $\wh{g}_{j,j}(k) = \lambda$ for all $j,k$ such that $\wh{y}_{j}(k) \neq 0$. 
    Assuming that $\mat{Z} \neq 0$ implies $\mat{Y} \neq 0$, therefore, a pair of indices $j,k$ such that $\wh{g}_{j,j}(k) = \lambda$ exists. 

    Assume by contradiction that $\lambda \notin \spc(\tA)$. Then there exists $\tB \in \cHs^{p \xx p}$ such that $\tI = (\tA - \lambda \tI) \ff \tB$. 
    Moreover, we have that $\tB = \tV \ff \sum_{k \in \ZZ} \sum_{j=1}^p (\wh{g}_{j,j}(k) - \lambda)^{-1} \bs{\phi}^{(k)} \ff \bs{\Phi}^{(j)} \ff \bs{\Phi}^{(j)}{}^*  \ff \tV^*$ is a bounded linear operator, and in particular,  $|(\wh{g}_{j,j}(k) - \lambda)^{-1} | \leq \opNorm{\tB}$ for all $j,k$.
    As a result, $|\wh{g}_{j,j}(k) - \lambda| > \opNorm{\tB}^{-1} $ abrogating the possibility for the existence of a nonzero $\mat{Z} \in \cHs^p$ such that $\tA \ff \mat{Z} = \lambda \mat{Z}$. 
    In conclusion, $\lambda$ is an eigenvalue of $\tA$ only if $\lambda \in \spc(\tA)$. 
\end{proof}
\Cref{lem:singvals.in.spc} establishes a connection between the singular values of $\tX \in \cHs^{m \xx p}$ and the spectra of $\tX^* \ff \tX \in \cHs^{p \xx p}$ and $\tX \ff \tX^* \in \cHs^{m \xx m}$, that is, $\operatorname{Image} \sigmaset_{\tX}^2 \subseteq  \spc (\tX^* \ff \tX)  \cap \spc (\tX \ff \tX^*)$ where $\sigmaset_{\tX}^2 \colon [\min(p,m)] \xx \ZZ \to \RR$ is the mapping
\begingroup{
}\endgroup
\begin{equation}\label{eq:sigmaset2.chmp}
    \sigmaset_{\tX}^2(j,k) = |\thS_{j,j,k}|^2 \quad \text{for all } (j,k) \in [\min(m,p)] \xx \ZZ
\end{equation}

Adhering to standard terminology, we say that $\alpha \in \RR$ is an accumulation point of the mapping $\sigmaset_{\tX}^2$, if for any neighborhood $A$ of $\alpha$, there are infinitely many $(j,k)$ such that $\sigmaset_{\tX}^2(j,k) \in A$.  
Importantly, for $\tX \in \cH^{m \xx p}$ we have the following lemma: 
\begin{lemma}\label{lem:sigmaset.X.in.H.has.no.acc}
    Let $\tX \in \cH^{m \xx p}$, then $0$ is the only accumulation point of the mapping $\sigmaset_{\tX}^2$ defined in \Cref{eq:sigmaset2.chmp}. 
    As a corollary, the mapping $\sigmaset_{\tX}^2$ attains a maximum on any non-empty subset of indices $K \subseteq [\min(p,m)] \xx \ZZ$. 
\end{lemma}
\begin{proof}
    Write $\tX = \tU \ff \tS \ff \tV^*$. 
    If $\tX = 0$ then $\wh{\tS}_{j,j,k} = 0$ for all $j,k$ and the statement is clear.

    Otherwise, there are $j,k$ such that $\wh{\tS}_{j,j,k} \neq 0$. 
    Let $\alpha \geq 0$ be an accumulation point of $\sigmaset_{\tX}^2$.
    Given $\varepsilon > 0$, denote by $\{(j_{a}^{(\varepsilon)}, k_{a}^{(\varepsilon)})\}_{a = 1}^\infty \subset [\min(m,p)] \xx \ZZ$ the set of indices such that $| |\wh{\tS}_{j_{a}^{(\varepsilon)}, j_{a}^{(\varepsilon)}, k_{a}^{(\varepsilon)}}|^2 - \alpha | < \varepsilon/2$, or, 
    $$
    \alpha -\varepsilon/2 < |\wh{\tS}_{j_{a}^{(\varepsilon)}, j_{a}^{(\varepsilon)}, k_{a}^{(\varepsilon)}}|^2  < \alpha + \varepsilon/2~~.
    $$
    For all $\varepsilon > 0$ we have
    \begin{align*}
        \cHNormS{\tX} = \cHNormS{\tS}
        &= {\sum}_{j = 1}^p {\sum}_{k \in \ZZ}|\wh{\tS}_{j,j,k}|^2 \\
        &\geq {\sum}_{a =1}^\infty |\wh{\tS}_{j_{a}^{(\varepsilon)}, j_{a}^{(\varepsilon)}, k_{a}^{(\varepsilon)}}|^2
    \end{align*}
    By taking $\varepsilon = \alpha$ we obtain 
    \begin{align*}
        \cHNormS{\tX} 
        &\geq
        {\sum}_{a =1}^\infty |\wh{\tS}_{j_{a}^{(\alpha)}, j_{a}^{(\alpha)}, k_{a}^{(\alpha)}}|^2 \\
        &\geq {\sum}_{a =1}^\infty \alpha / 2
    \end{align*}
    thus $\alpha = 0$ (otherwise $\tX \notin \cH^{m \xx p}$), and we have shown that the only possible accumulation point of $\sigmaset_{\tX}^2$ is 0.
    Since $\cHNormS{\tS}$ is finite,
    we have that the series $\sum_{n=1}^\infty |\wh{\tS}_{c_n^{(\operatorname{lat})}, c_n^{(\operatorname{lat})}, c_n^{(\operatorname{front})}} |^2$ converges for any ordering 
    $c \colon n \mapsto (c_n^{(\operatorname{lat})},c_n^{(\operatorname{front})}) \in [\min(m,p)]\xx \ZZ$ of $[\min(m,p)]\xx \ZZ$.
    Therefore, given any ordering (bijection)  $c \colon \NN \to [\min(m,p)]\xx \ZZ$ we have $\lim_{n\to \infty} |\wh{\tS}_{c_n^{(\operatorname{lat})}, c_n^{(\operatorname{lat})}, c_n^{(\operatorname{front})}} |^2 = 0$ thus $0$ is necessarily an accumulation point of $\sigmaset_{\tX}^2$.

    For the corollary, let $K \subseteq [\min(p,m)] \xx \ZZ$ be a subset, and let $\alpha = \sup_{(j,k) \in K} |\wh{\tS}_{j,j,k}|^2 $.
    Our goal is to show that $\alpha = |\wh{\tS}_{j',j',k'}|^2$ for some $(j',k') \in  K$.
    
    If $\alpha$ is an accumulation point of $\sigmaset_{\tX}^2$, then by the above $\alpha = 0$.
    Since $\sigmaset_{\tX}^2$ is non-negative, we get that $\sigmaset_{\tX}^2(j,k) = 0 $ for all $(j,k) \in K$, i.e., $0 = \alpha = \max_{(j,k) \in K} \sigmaset_{\tX}^2(j,k) $ and the result follows. 
    Otherwise, according to the definition of accumulation point of a mapping, there exists a neighborhood $A_0$ of $\alpha$, such that the set of indices $(j,k) \in [\min(m,p)] \xx \ZZ$ for which $\sigmaset_{\tX}^2(j,k) \in A_0$ is of finite cardinality. 
    In particular, we have that the set  $C_0 \coloneqq \{ (j,k) \in K ~:~  \sigmaset_{\tX}^2(j,k) \in A_0 \}$ is finite.

    From the definition of supremum it follows that for all $n \in \NN$ there exists $(j_n,k_n) \in K$ such that 
    $$
    \alpha - 1/n < \beta_n \leq \alpha ~~ \text{where }  \beta_n \coloneqq |\wh{\tS}_{j_n,j_n,k_n}|^2 = \sigmaset_{\tX}^2(j_n,k_n)
    $$
    Assume without loss of generality that $\beta_n$ is monotonically increasing (otherwise replace it by a monotonically increasing subsequence that surely exists), then $\beta_n$ converges to $\alpha$. 
    As a result, there exists $N_0 \in \NN$ such that $\beta_n \in A_0$ for all $n \geq N_0$. 
    By construction $\beta_n = \sigmaset_{\tX}^2(j_n,k_n) = |\wh{\tS}_{j_n,j_n,k_n}|^2 $ therefore $(j_n,k_n) \in C_0$ for all $n \geq N_0$. 
    Observe that 
    \begin{equation}
        \alpha = \limsup_{n \to \infty} \beta_n 
        = \limsup_{n \to \infty} \sigmaset_{\tX}^2(j_n,k_n) 
        \leq \sup \left\{ \sigmaset_{\tX}^2(j,k) ~:~ (j,k) \in C_0 \right\}  \label{eq:no.accpoints.eq1}
    \end{equation}
    Following from the finiteness of $C_0$, we obtain that  $\max \{ \sigmaset_{\tX}^2(j,k) ~:~ (j,k) \in C_0 \} $ exists, and since $C_0 \subseteq K$ we get that 
    \begin{equation} \label{eq:no.accpoints.eq2}
    \max \left\{ \sigmaset_{\tX}^2(j,k) ~:~ (j,k) \in C_0 \right\} 
    \leq \sup \left\{  \sigmaset_{\tX}^2(j,k) ~:~ (j,k) \in  K \right\} 
    = \alpha
    \end{equation}
    Combining the inequalities of \Cref{eq:no.accpoints.eq1,eq:no.accpoints.eq2}, we get
    $$
    \alpha \leq \max \left\{ \sigmaset_{\tX}^2(j,k) ~:~ (j,k) \in C_0 \right\} \leq \alpha
    $$
    that is, there exists $(j',k') \in K$ such that $\alpha = \sigmaset_{\tX}^2(j',k')$.
    
    In conclusion, the mapping $\sigmaset_{\tX}^2$ attains a maximum on any subset $K \subset [\min(p,m)] \xx \ZZ$.
\end{proof}
Note that for any $\tX = \tU \ff \tS \ff \tV^* \in \cHs^{m \xx p}$, we have that $\tS_{j,j} \geq_{\cHs} 0$, i.e., $\wh{\tS}_{j,j,k} \geq 0$ for all $j,k$, and the mapping $\wh{\tS}_{j,j,k} \mapsto |\wh{\tS}_{j,j,k}|^2$ is a one to one correspondence. 
Thus, for any $\tX \in \cH^{m \xx p}$, the set $\{ \wh{\tS}_{j,j,k} ~:~  (j,k) \in [\min(m,p)] \xx \ZZ \}$ has a single accumulation point that is $0$.

Another direct consequence of~\Cref{lem:sigmaset.X.in.H.has.no.acc} is the following result
\begin{lemma}\label{lem:levelsets.are.finite}
    Let $\tX \in \cH^{m \xx p}$ with q-SVD $\tX = \tU \ff \tS \ff \tV^*$. 
    Define 
    \begin{equation}\label{eq:def.sublevel.set.sigma}
        S_{\tX}(\alpha) \coloneqq \{ \wh{\tS}_{j,j,k} ~:~  \wh{\tS}_{j,j,k} > \alpha \} \qquad \text{for all } \alpha \geq 0
    \end{equation}
    Then $S_{\tX}(\alpha)$ is finite for any $\alpha > 0$.
\end{lemma}
\begin{proof}
    Let $\alpha > 0$ and note that $S_{\tX}(\alpha) \subseteq [\alpha, \opNorm{\tX}]$ thus an infinite sequence of values contained in $S_{\tX}(\alpha)$ must have a convergent subsequence $\{ \beta_n \}_{n=1}^\infty $ whose limit $\beta $ is such that $\beta \geq \alpha >0$ and in particular, $\beta$ is an accumulation point of the mapping $(j,k) \mapsto \wh{\tS}_{j,j,k}$ therefore $\beta^2 > 0$ is  an accumulation point of $\sigmaset_{\tX}^2$, in contradiction to~\Cref{lem:sigmaset.X.in.H.has.no.acc}.
\end{proof}


\begin{definition}\label{def:descending.order.seq}
    Let $\tX = \tU \ff \tS \ff \tV^*  \in \cH^{m \xx p}$ be non-zero.
    Define $\{ \sigma_n \}_{n=1}^\infty $ with $\sigma_n = \wh{\tS}_{l_n,l_n,t_n}$ where 
    \begin{equation}\label{eq:def.lt_1.max}
        (l_1,t_1) = \argmax_{(l,t) \in [\min(p,m)] \xx \ZZ} \wh{\tS}_{l,l,t} 
    \end{equation}
    and for all $n=2,3,\ldots$
    \begin{align}
        C_n 
        &=  ([\min(p,m)] \xx \ZZ) \smallsetminus \{ (l_{n'},t_{n'}) ~:~ n' =1, \ldots ,n-1 \} 
        = C_{n-1} \smallsetminus \{ (l_{n-1},t_{n-1}) \} \label{eq:def.C_n.subset} \\ 
        (l_n,t_n) 
        &= \argmax_{(l,t) \in C_n} \wh{\tS}_{l,l,t} \label{eq:def.lt_n.max}
    \end{align}

    In case of multiple maximizers for \Cref{eq:def.lt_1.max,eq:def.lt_n.max}, we take $(l_n,t_n)$ that is smallest with respect to lexicographic order (with the second coordinate preceding).
\end{definition}
\begin{lemma}\label{lemma:ordering.arbitrary.small}%
    Let $\tX \in \cH^{m \xx p}$ with q-SVD $\tX = \tU \ff \tS \ff \tV^* $ and $\{ \sigma_n \}_{n=1}^{\infty} $ the sequence whose construction is given by~\Cref{def:descending.order.seq}.
    Then 
    \begin{equation}\label{eq:lem.sigma_n.dec.order}
        \sigma_1 \geq \sigma_2 \geq \cdots 
    \end{equation}
    and for a all $\varepsilon > 0$, there exists $N \in \NN$ such that $\sigma_{n} < \varepsilon$ for all $n \geq N$.
\end{lemma}
\begin{proof}
    First, note that the maximum in \Cref{eq:def.lt_1.max} and \Cref{eq:def.lt_n.max} both exist due to \Cref{lem:sigmaset.X.in.H.has.no.acc}.
    If multiple indices yield the same positive maximum, this set must be finite, and any ordering will work, and in particular the lexicographic ordering we use. 

    By construction, $\sigma_{n} \in \{  \wh{\tS}_{l,l,t} :  (l,t) \in C_n \} \subseteq \{  \wh{\tS}_{l,l,t} :  (l,t) \in C_{n-1} \}$ hence, following from \Cref{eq:def.C_n.subset} we have $\sigma_n \leq \max_{(l,t) \in C_{n-1} }  \wh{\tS}_{l,l,t} = \sigma_{n-1}$, therefore \Cref{eq:lem.sigma_n.dec.order} holds.
    For the second part, note that $\sum_{n=1}^\infty |\sigma_n|^2 = \cHNormS{\tX}$ is a convergent series thus for any $\varepsilon > 0$ there exists $N_0 \in \NN$ such that $\sum_{n = N_0}^\infty |\sigma_n|^2 \leq \varepsilon^2/4$ and for all $n \geq N_0$ it is clear that $|\sigma_n|^2 \leq \sum_{{n'} = N_0}^\infty |\sigma_{n'}|^2$ thus $\sigma_n < \varepsilon$.
\end{proof}



\paragraph{Implicit ranks.}
Previous works, e.g., \cite{KilmerPNAS}, define the implicit rank of a finite dimensional tubal tensor $\tX \in \FFF^{m \xx p \xx n}$, as  the number of non-zero entries in $\thS$ where $\tX = \tU \mm \tS \mm \tV^{\bf{H}}$  is the t-SVD of $\tX$ under $\mm$.
As a straightforward extension, we define the implicit rank of a quasitubal tensor $\tX = \tU \ff \tS \ff \tV^* \in \cHs^{m \xx p}$ as the cardinality of the set $\{ (l,t) \in [\min(p,m)] \xx \ZZ ~|~ \wh{s}_l^{(t)} \neq 0 \}$ which may be infinite, such as the case where $\tX = \tI$ in which all diagonal entries of the transformed singular quasitubal tensor are 1.


Recall that by \Cref{lem:Rlin.iff.matmul} a quasitubal tensor represents an $\cHs$-linear mapping (and vice versa), and that $\cHs$-linear mappings are, in particular, linear maps between vector spaces over $\FFF$.
 
\begin{definition}[Rank of a mapping]\label{def:chs.linmap.rank}
    Let $T \colon \cHs^p \to \cHs^m$ be a $\cHs$-linear mapping, and denote by $\operatorname{Image}_{\FFF} T \subseteq \cHs^m$ the image of $T$ when considered as a vector space over $\FFF$.
    The rank of $T$ is defined as $\rank_{\FFF,\ff} T \coloneqq  \dim \operatorname{Image}_{\FFF} T$. 
\end{definition}
We will use $\rank_{\FFF,\ff} T$ and $\rank_{\FFF,\ff} \tX$ interchangeably for referring to the rank of the mapping $T$ whose quasitubal tensor representation is given by $\tX$. 
The following lemma shows that both the rank of $\tX$ and the $\cHs$-linear mapping it defines are, in fact, the implicit rank of $\tX$ under $\ff$. 

\begin{lemma}\label{lem:Frank.is.imprank}
    Let $\tX \in \cHs^{m \xx p}$ be a quasitubal tensor and $T = T_{\tX} \colon \cHs^p \to \cHs^m$ be the $\cHs$-linear mapping defined by $T(\matY) = \tX \ff \matY$ for all $\matY \in \cHs^p$. 
    Then the implicit rank of $\tX$ is $\rank_{\FFF,\ff} T$.
\end{lemma}
\begin{proof}
    Write $\tX = \tU \ff \tS \ff \tV^*$ and  
    $I = \{(j,k) \in [\min(p,m)] \xx \ZZ ~|~  \thS_{j,j,k} \neq 0 \}$. 
    Note that $I$ is of at most countable cardinality, and write $I = \{ (j_n,k_n) \}_{n=1}^q$ where $q$ may be finite or infinite.
    Then $\tX = {\sum}_{n=1}^q \thS_{j_n,j_n,k_n} \bs{\phi}^{({k_n})} \ff \tU_{:,{j_n}} \ff \tV_{:,{j_n}}^*$.
    Note that $\tI = \tV \ff \tV^* = \sum_{k \in \ZZ} \sum_{j =1}^p \bs{\phi}^{({k})} \ff \tV_{:,{j}} \ff \tV_{:,{j}}^* $, therefore, 
    $$\matY = \tV \ff \tV^* \ff \matY  = {\sum}_{n=1}^q  \bs{\phi}^{({k_n})} \ff (\tV_{:,{j_n}}^* \ff \matY) \ff \tV_{:,{j_n}} +  {\sum}_{(j',k') \notin I}  \bs{\phi}^{({k'})} \ff (\tV_{:,{j'}}^* \ff \matY) \ff \tV_{:,{j'}}$$ 
    for all $\matY \in \cHs^p$.
    As a result, 
    \begin{align*}
        \tX \ff \matY 
        &= {\sum}_{n=1}^q \wh{s}_{j_n}^{({k_n})} \bs{\phi}^{({k_n})} \ff \tU_{:,{j_n}} \ff \tV_{:,{j_n}}^* \ff {\sum}_{n'=1}^q  \bs{\phi}^{({k_{n'}})} \ff (\tV_{:,{j_{n'}}}^* \ff \matY) \ff \tV_{:,{j_{n'}}} \\
        &= {\sum}_{n=1}^q \wh{s}_{j_n}^{({k_n})} \bs{\phi}^{({k_n})} \ff (\tV_{:,{j_{n}}}^* \ff \matY) \ff \tU_{:,{j_n}}
    \end{align*}
    Note that $\bs{\phi}^{({k_n})} \ff (\tV_{:,{j_{n}}}^* \ff \matY) = \alpha_n \bs{\phi}^{({k_n})}$ for some $\alpha_n \in \FFF$, hence the above becomes
    \begin{align*}
        \tX \ff \matY 
        &= {\sum}_{n=1}^q \alpha_n \thS_{j_n,j_n,k_n}  \bs{\phi}^{({k_n})} \ff \tU_{:,{j_n}}
    \end{align*}
    which implies that $\operatorname{Image}_{\FFF} T = \operatorname{span}(\{ \bs{\phi}^{({k_n})}  \ff \tU_{:,{j_n}} \}_{n=1}^{q})$.
    Since $(\bs{\phi}^{({k_n})}  \ff \tU_{:,{j_n}})^* \ff (\bs{\phi}^{({k_{n'}})}  \ff \tU_{:,{j_{n'}}}) \neq 0 $ if and only if $n =n'$ the elements in $\{ \bs{\phi}^{({k_n})}  \ff \tU_{:,{j_n}} \}_{n=1}^{q}$ are linearly independent, thus $\dim \operatorname{Image}_{\FFF} T = \dim \operatorname{span}(\{ \bs{\phi}^{({k_n})}  \ff \tU_{:,{j_n}} \}_{n=1}^{q}) = q$.
\end{proof}

\begin{lemma}
    Let $\tX = \tX^{(1)} \ff \tX^{(2)} \in \cHs^{m \xx p}$ where $\tX^{(1)} \in \cHs^{m \xx r}$ and $ \tX^{(2)} \in \cHs^{r \xx p} $. 
    Then $\rank_{\FFF,\ff} \tX \leq \min(\rank_{\FFF,\ff} \tX^{(1)},\rank_{\FFF,\ff} \tX^{(2)})$
\end{lemma}
\begin{proof}
    For all $k \in \ZZ$, the matrix rank of $\thX_{:,:,k}$ is bounded by $\min(\rank \thX^{(1)}_{:,:,k}, \rank \thX^{(2)}_{:,:,k})$. 
    We have 
    $\rank \thX = \sum_{k \in \ZZ} \rank \thX_{:,:,k} \leq \sum_{k \in \ZZ} \min(\rank \thX^{(1)}_{:,:,k}, \rank \thX^{(2)}_{:,:,k}) \leq \min(\sum_{k \in \ZZ} \rank \thX^{(1)}_{:,:,k}, \sum_{k \in \ZZ} \rank \thX^{(2)}_{:,:,k}) = \min(\rank_{\FFF,\ff} \tX^{(1)},\rank_{\FFF,\ff} \tX^{(2)})$.
\end{proof}

}\endgroup
\else
Commented out the construction of a module structure on $\cHs^{m \xx p}$
\fi

\newcommand*\ruleline[1]{\par\noindent\raisebox{.8ex}{\makebox[\linewidth]{\hrulefill\hspace{1ex}\raisebox{-.8ex}{#1}\hspace{1ex}\hrulefill}}}


\section{Optimal Finitely Representable Rank Truncations}\label{sec:finite.rank.approximation}

In this section we are interested in finding a best, finite, implicit rank approximation.
We prove our second Eckart-Young-like result, showing that explicitly truncating at a given rank budget, gives the optimal approximation with the specified implicit rank. 
Importantly, quasitubal tensors of finite implicit rank are finitely representable, i.e., these are finite dimensional tubal tensors embedded in $\cHs^{m \xx p}$ via inclusion. 
Therefore, explicit rank truncation are finitely representable as well. 

The problem of finding a best approximation for $\tX \in \cHs^{m \xx p}$ with rank up to $q$, is stated as 
\begin{equation}\label{eq:optprob.lowrank.approx}
    \text{minimize} \GNorm{\tX - \tY}_{\square} \text{ subjected to } \rank_{\FFF,\ff} \tY \leq q
\end{equation}
where $\GNorm{\cdot}_{\square}$ may be any norm defined on a subspace of $\cHs^{m \xx p}$ that contains $\tX$ and $\tY$.

For general $\tX \in \cHs^{m \xx p}$, the only guarantee is that $\opNorm{\tX}$ and the equivalent $\cHsNorm{\tX}$ are finite. 
In this case, best low-rank approximations, i.e., minimizers of \Cref{eq:optprob.lowrank.approx}, may be meaningless. 
For example, consider the identity operator $\bs{e} \in \cHs$. Then for any finite rank quasitube $\bs{y} = \sum_{k=1}^q \wh{y}_{t_k} \bs{\phi}^{(t_k)} \in \cHs$, write $P_{\bs{y}} = \sum_{k=1}^q \bs{\phi}^{(t_k)}$ and have that
\begin{align*}
    \opNorm{\bs{e} - \bs{y}}
    &= \opNorm{P_{\bs{y}} \ff \bs{e} - \bs{y} + (\bs{e} - P_{\bs{y}} ) \ff \bs{e} } \\
    &\geq \max(\opNorm{P_{\bs{y}} \ff \bs{e} - \bs{y}}, \opNorm{\bs{e} - P_{\bs{y}}}) \\
    &=  \max(\opNorm{P_{\bs{y}} \ff \bs{e} - \bs{y}}, 1)
\end{align*}
ergo, the minimum of $\opNorm{\bs{e} - \bs{y}}$ for any finite rank $\bs{y}$ is bounded from below by 1. We see that even when $q\to\infty$, we do not have a rank-$q$ approximation of $\bs{e}$ that is arbitrarily good.

For $\tX \in \cH^{m \xx p}$ we expect a finite residual $\cHNorm{\tX - \tY}$ for any $\tY $ with $\rank_{\FFF,\ff} \tY \leq q$. In particular, when $q\to\infty$, we can have a rank-$q$ approximation that is arbitrarily good, as the following theorem shows.

\begin{theorem}\label{thm:best.finite.approx.cH}
    Let $\tX \in \cH^{m \xx p}$, with a q-SVD $\tX = \tU \ff \tS \ff \tV^*$ and let $q \in \NN$ be a finite integer.
    Define the \textbf{explicit rank $q$ truncation} of $\tX$ as the quasitubal tensor
    \begin{equation}\label{def:eq.best.irank.q.trunk}
        \tX_{[q]} \coloneqq {\sum}_{n=1}^q \wh{s}^{(t_{n})}_{l_{n}} \bs{\phi}^{(t_{n})} \ff \tU_{:,l_{n}} \ff \tV_{:,l_{n}}^*
    \end{equation}
    where the mapping $ n \mapsto (l_{n},t_{n})$ is one-to-one mapping from $\NN$ onto a subset of $[p] \xx \ZZ$ such that
    \begin{equation}\label{prop:eq.descending.order}
        \wh{s}^{(t_{1})}_{l_{1}} \geq \wh{s}^{(t_{2})}_{l_{2}} \geq \cdots \geq \wh{s}^{(t_{q})}_{l_{q}} \geq \wh{s}^{(t_{q+1})}_{l_{q+1}} \geq \cdots \geq 0 \quad \text{with } \wh{s}^{(t)}_{l} \coloneqq \thS_{l,l,t}
    \end{equation}
    Then $\rank_{\FFF,\ff} \tX_{[q]} \leq q$.
    Moreover, it holds that 
    \begin{equation}\label{eq:state.best.irank}
        \forall \tY \in \cH^{m \xx p}, \rank_{\FFF,\ff} \tY \leq q \qquad \cHNorm{\tX - \tX_{[q]}} \leq \cHNorm{\tX - \tY} 
    \end{equation}
    That is, $\tX_{[q]}$ is a optimal implicit rank $q$ approximation of $\tX$ in $\cHNorm{\cdot}$.

    Importantly, $\lim_{q \to \infty} \opNorm{\tX - \tX_{[q]}}= \lim_{q \to \infty} \cHNorm{\tX - \tX_{[q]}} = 0$.
\end{theorem}

Before proving \Cref{thm:best.finite.approx.cH}, recall the statement above about the finite representability of quasitubal tensors whose implicit rank is finite.
Given $\tX \in \cH^{m \xx p}$ with $\rank_{\FFF,\ff} \tX = q$, \Cref{thm:best.finite.approx.cH} shows that $\tX = \sum_{n=1}^q \wh{s}_{l_n}^{(t_n)} \bs{\phi}^{(t_{n})} \ff \tU_{:,l_n} \ff \tV_{:,l_n}^*$.
Note that since $\{ t_n \}_{n=1}^q$ is finite set of integers, it must be bounded, and there are integers $b,B$ such that   $-b \leq t_n \leq B $ for all $n=1,\ldots,q$.
Therefore, there exists a \textbf{finite dimensional, tubal} tensor $\wh{\tA} \in \FFF^{m \xx p}_{B+b+1}$ such that $\thX \in \ellinf^{m \xx p}$ is obtained by embedding  $\wh{\tA}$ in $\ellinf^{m \xx p}$ 
using a trivial inclusion map: first, by padding both sides of the frontal axis with an infinite sequence of zeros, then shifting them by  $-b$.
An appealing consequence of this property is that, in applied contexts, finite-rank truncations can be computed, expressed, stored, and manipulated.

\begin{proof}
    The mapping $n \mapsto (l_n,t_n)$ given in \Cref{def:descending.order.seq} is constructed such that \Cref{prop:eq.descending.order} holds and we have 
    $$
    \tX = {\sum}_{n=1}^{\infty} \wh{\sigma}_{n} \bs{\phi}^{(t_n)} \ff \tU_{:,l_n} \ff \tV_{:, l_n}^*
    $$
    with $\wh{\sigma}_{n} = \wh{s}_{l_n}^{(t_n)} = \wh{\tS}_{l_n,l_n,t_n}$ being non-negative reals  such that $\wh{\sigma}_{1} \geq \wh{\sigma}_{2} \geq \cdots$.


    Let $Q$ denote (the possibly infinite quantity) $\rank_{\FFF, \ff} \tX$.
    For any $q \geq Q$ we have $\tX_{[q]} = \tX $ which is the best possible approximation. 
    Otherwise, 
    let $\tY \in \cH^{m \xx p}$ be a quasitubal tensor with $\rank_{\FFF, \ff} \tY = q' \leq q$, and write 
    \begin{equation}\label{eq:tenY.blr.proof1}
        \tY = \sum_{k=1}^{q'} \wh{g}_{\lambda_k}^{(\tau_k)} \bs{\phi}^{(\tau_k)} \ff \tens{R}_{:,\lambda_k} \ff \tW_{:,\lambda_k}^*
    \end{equation}
    where $\tY  = \tens{R} \ff \tens{G} \ff \tens{W}^*$ is the full q-SVD of $\tY$.
    
    Since $\tens{R}$ and $ \tens{W}$ are $\ff$-unitary we have that $\{ \tens{R}_{:,j} \ff \tens{W}_{:,j'}^* ~|~  j \in [m], j' \in [p] \}$ is a unit normalized $\ff$-orthogonal basis in $\cHs^{m \xx p}$,
    therefore, it is clear that 
    $\{ \bs{\phi}^{(k)} \ff  \tens{R}_{:,j} \ff \tens{W}_{:,j'}^* ~|~ k \in \ZZ, j \in [m], j' \in [p] \}$
    is an orthonormal basis for the Hilbert space $\cH^{m \xx p}$ considered
    with the inner product defined in \Cref{eq:def:chmp.hilbert.inner.product}

    For  $k = 1,\ldots,q'$, denote the $k$-th rank-1 quasitubal tensor in \Cref{eq:tenY.blr.proof1}  by 
    $\tB^{(k)} = \bs{\phi}^{(\tau_k)} \ff \tens{R}_{:,\lambda_k} \ff \tW_{:,\lambda_k}^* \in \cH^{m \xx p}$.
    Consider the mapping $T_{\tY} (\tZ) = \tY \ff \tZ$ for all $\tZ \in  \cHs^{p \xx p} $. 
    Since $\cH$ is a *-ideal in $\cHs$ we get that 
    $T_{\tY} \colon \cHs^{p \xx p} \to \cH^{m \xx p}$ is a $\cHs$-linear, thus in particular linear.
    Note that $\{ \tB^{(k)}  \}_{k=1}^{q'}$ is an orthonormal basis for the image of $T_{\tY}$ in $\cH^{m \xx p}$, and let  
    $\{ \tB^{(k)}  \}_{k=q'+1}^{\infty}$ be be an orthonormal basis for the orthogonal complement of the image of $T_{\tY}$ in $\cH^{m \xx p}$. 
    Then, for any $\tA \in \cH^{m \xx p}$ it is possible to write $\tA = \sum_{k=1}^{\infty} \alpha_k \tB^{(k)}$ where $\alpha_k \in \FFF$.

    Write $\tX = \sum_{k=1}^{\infty} \alpha_k \tB^{(k)}$ then
    \begin{align*}
        \cHNormS{\tX - \tY }
        &=\cHNormS{{\sum}_{k= 1}^{q'} (\alpha_k - \wh{g}_{\lambda_k}^{(\tau_k)} ) \tB^{(k)} + {\sum}_{k= q'+ 1}^{\infty} \alpha_k \tB^{(k)}  } \\
        &= {\sum}_{k= 1}^{q'} |\alpha_k - \wh{g}_{\lambda_k}^{(\tau_k)} |^2 + {\sum}_{k= q' + 1}^{\infty} |\alpha_k|^2 
    \end{align*}

    Observe that if $q'<q$, then the quasitubal tensor $\wt{\tY} = \tY + \alpha_{q'+1}  \tB^{(q'+1)}$ is of rank at most $q$ and that
    \begin{align*}
        \cHNormS{\tX - \wt{\tY} }
        &= {\sum}_{k= 1}^{q'} |\alpha_k - \wh{g}_{\lambda_k}^{(\tau_k)} |^2 + {\sum}_{k= q' + 2}^{\infty} |\alpha_k|^2 \\
        &\leq \cHNormS{\tX - \tY }
    \end{align*}
    therefore $\wt{\tY}$ exhibits smaller approximation error.

    Suppose, $q'=\rank_{\FFF, \ff} \tY=q$ and let $\rrho$ denote the multi-rank of $\tY$. 
    There are at most $q$ frontal slices such that $\thY_{:,:,k} \neq 0$.
    Let $R = \{ k_1,\ldots,k_t \} \subset \ZZ$ denote the set of nonzero frontal slices of $\thY$.
    Then for all $n \in R$ we have $\cHNormS{\thX_{:,:,n} - \thY_{:,:,n}} \geq \cHNormS{\thX_{:,:,n} - [\thX_{:,:,n}]_{\rho_{n}} }$ where $[\thX_{:,:,n}]_{\rho_{n}}$ denotes the optimal  rank $\rho_{n}$ approximation of the matrix $\thX_{:,:,n}$. 
    
    Define $\wt{\tY} = \sum_{n\in R}  (\bs{\phi}^{(n)} \ff \tX)_{[\rho_{n}]}$ then 
    \begin{align*}
        \cHNormS{\tX - \wt{\tY}} 
        &= \sum_{n \in R} \cHNormS{\thX_{:,:,n} - [\thX_{:,:,n}]_{\rho_{n}}} + \sum_{n \notin R} \cHNormS{\thX_{:,:,n}} \\
        &\leq \sum_{n \in R} \cHNormS{\thX_{:,:,n} - \thY_{:,:,n}} + \sum_{n \notin R} \cHNormS{\thX_{:,:,n}} \\
        &= \cHNormS{\tX - \tY}
    \end{align*}
    and note that 
    \begin{align*}
        \cHNormS{\tX - \wt{\tY}}
        &= {\sum}_{k \in \ZZ} {\sum}_{j=1}^{\min(m,p)}  (1 - \upchi_{_{R}}(k) \upchi_{[\rho_k]}(j)) |\wh{s}_{j}^{(k)}|^2 \\
        &= {\sum}_{n=1}^{\infty} (1 - \upchi_{_{R}}(t_n) \upchi_{[\rho_{t_n}]}(l_n))  |\wh{s}_{l_n}^{(t_n)}|^2
    \end{align*}
    Let $N_q \subset \NN$ denote the set of numbers $n\leq q$ such that $\upchi_{_{R}}(t_n) \upchi_{[\rho_{t_n}]}(l_n) = 1$. 
    For each $n \in [q] \smallsetminus N_q$ there exists $n' > q$ such that $\upchi_{_{R}}(t_{n'}) \upchi_{[\rho_{t_{n'}}]}(l_{n'}) = 1$ and since $\wh{s}_{l_n}^{(t_n)} \geq  \wh{s}_{l_{n'}}^{(t_{n'})}$ we have 
    \begin{align*}
        \cHNormS{\tX - \wt{\tY}}
        &= {\sum}_{n=1}^{\infty}   |\wh{s}_{l_n}^{(t_n)}|^2
        - {\sum}_{n \in N_q} |\wh{s}_{l_n}^{(t_n)}|^2 
        -{\sum}_{n \notin N_q} \upchi_{_{R}}(t_n) \upchi_{[\rho_{t_n}]}(l_n) |\wh{s}_{l_n}^{(t_n)}|^2 \\
        &\geq {\sum}_{n=1}^{\infty}   |\wh{s}_{l_n}^{(t_n)}|^2
        - {\sum}_{n \in N_q} |\wh{s}_{l_n}^{(t_n)}|^2 
        -{\sum}_{n = 1}^q  (1 - \upchi_{_{N_q}}(n)) |\wh{s}_{l_n}^{(t_n)}|^2 \\
        &= {\sum}_{n=1}^{\infty}   |\wh{s}_{l_n}^{(t_n)}|^2 - {\sum}_{n = 1}^q |\wh{s}_{l_n}^{(t_n)}|^2 
        = \cHNormS{\tX - \tX_{[q]}} 
    \end{align*}
    So, we have shown that $\cHNormS{\tX - \tY} \geq \cHNormS{\tX - \tX_{[q]}}$ as desired.
    
    Let $\varepsilon > 0$, then by \Cref{lemma:ordering.arbitrary.small} there exists an $N \in \NN$ such that for all $n > N$ we have $\wh{s}_{l_n}^{(t_n)} \leq \varepsilon$ then
    \begin{align*}
        \opNorm{\tX - \tX_{[q]}} 
        &= \opNorm{\tU \ff \sum_{n=q+1}^{\infty} \wh{s}_{l_n}^{(t_n)} \bs{\phi}^{(t_n)} \ff \bs{\Phi}_{(m)}^{(l_n)} \ff \bs{\Phi}_{(p)}^{(l_n)}{}^* \ff \tV^*} \\
        &= \opNorm{\sum_{n=q+1}^{\infty} \wh{s}_{l_n}^{(t_n)} \bs{\phi}^{(t_n)} \ff \bs{\Phi}_{(m)}^{(l_n)} \ff \bs{\Phi}_{(p)}^{(l_n)}{}^*} \\ 
        &= \sup_{n > q} \wh{s}_{l_n}^{(t_n)}
    \end{align*}
    therefore $\opNorm{\tX - \tX_{[q]}} \leq \varepsilon$.
    This shows that $\lim_{q \to \infty} \opNorm{\tX - \tX_{[q]}} = 0 $. 
    Next, write
    \begin{align*}
        \cHNormS{\tX - \tX_{[q]}} 
        &= \sum_{n=q+1}^{\infty} |\wh{s}_{l_n}^{(t_n)}|^2 \\
        &= \cHNormS{\tS} - \sum_{n=1}^{q} |\wh{s}_{l_n}^{(t_n)}|^2
    \end{align*}
    then since $\lim_{q \to \infty }\sum_{n=1}^{q} |\wh{s}_{l_n}^{(t_n)}|^2 = \cHNormS{\tS}$ it is clear that $\lim_{q \to \infty}\cHNorm{\tX - \tX_{[q]}} = 0$.

\end{proof}

\section{Computational Point of View}
Note the obvious fact that no computer is capable of explicitly representing or storing  infinite dimensional objects. 
Yet, many (routine) computations involve infinite dimensional objects. 
Such tasks are attainable whenever an object can be approximated by representable (finite) approximations to an arbitrary precision, thus hopefully keeping the error of the resulting computational task bellow some desired error budget.

For quasitubal tensors over separable Hilbert space, \Cref{thm:best.finite.approx.cH} provides just that, since given $\tX \in \cH^{m \xx p}$,  $\{ \tX_{[q]} \}_{q \in \NN}$ converges to $\tX$ uniformly, with the fastest rate of convergence achieved for a sequence whose $k$'th element's rank is bounded by $k$. 
Unfortunately, the existence of a best rank $q$ approximation for a quasitubal tensor over separable Hilbert space given by~\Cref{thm:best.finite.approx.cH} does not imply that such approximation is easy to compute.

Given $\tX \in \cH^{m \xx p}$ and a target rank $q \in \NN$, worst case bounds on the number of arithmetic operations (e.g., flops) involved in computing $\tX_{[q]}$ are subjected to the context and premise in which we operate.
In the infinite dimensional case, we have no way to store entire transform domain images, so one must assume that $\thX$ is given as a mapping $k \mapsto \thX_{:,:,k}$ rather than an array of matrices which can be stored in memory (or disk). 
Hence, computing the transform domain SVD for all frontal slices is impossible.

In cases where the operations $\matZ \mapsto \tX \ff \matZ$ and $\matY \mapsto \tX^* \ff \matY$ are computationally inexpensive, then it might be worthwhile to directly apply iterative methods in order to compute the $q$ leading components. 

When the evaluation of $\thX_{:,:,k}$ is significantly more efficient than $\tX \ff \matY$ for arbitrary $\matY \in \cH^p$, it is possible to find a ``bandlimited'' approximation of $\tX$ that shares the $q$ leading components using the following procedure.

Define $\thY^{(1)} \in \ellinf^{m \xx p}$ such that $\thY^{(1)}_{:,:,k} = \sum_{k=-B_1}^{B_1} \bs{e}^{(k)} \odot \thX_{:,:,k}$ then $\bs{\phi}^{(k)} \ff \tY^{(1)} = \bs{\phi}^{(k)} \ff \tX$ for all $k=-B_1,\ldots,B_1$ and zero otherwise, where $B_1 \in \NN$ is such that
$\opNorm{\tY^{(1)}}^2 > \cHNormS{\tX} - \cHNormS{\tY^{(1)}}  $.
Let $k_1 \in \ZZ$ be (the unknown) integer such that $k_1 = \argmax_{k \in \ZZ} \TNorm{\thX_{:,:,k}} $, that is, $\TNorm{\thX_{:,:,k_1}} = \opNorm{\tX}$.
Note that 
\begin{align*}
    \cHNormS{\tX} - \cHNormS{\tY^{(1)}} 
    &= \sum_{k \in \ZZ} \FNormS{\thX_{:,:,k}} - \sum_{k = -B_1}^{B_1} \FNormS{\thX_{:,:,k}} 
    = \sum_{k \in \ZZ}  \FNormS{(1 - \chi_{|k| \leq B_1}(k)) \thX_{:,:,k}}  \\
    &= \sum_{k \in \ZZ}  \FNormS{\thX_{:,:,k} - \chi_{|k| \leq B_1}(k) \thX_{:,:,k}}  
    = \sum_{k \in \ZZ}  \FNormS{\thX_{:,:,k} - \wh{\tY}^{(1)}_{:,:,k}}  \\
    &\geq  \sum_{k \in \ZZ} \GNorm{\thX_{:,:,k} - \wh{\tY}^{(1)}_{:,:,k}}_2^2 
    \geq \sup_{k \in \ZZ} \GNorm{\thX_{:,:,k} - \wh{\tY}^{(1)}_{:,:,k}}_2^2  \\
    &= \opNormS{\thX - \thY^{(1)}} = \opNormS{\tX - \tY^{(1)}}
\end{align*}
Since $\opNorm{\tY^{(1)}}^2 > \cHNormS{\tX} - \cHNormS{\tY^{(1)}}$ we get that  $\opNorm{\tY^{(1)}} > \opNorm{\tX - \tY^{(1)}}$ holds, i.e., $\opNorm{\tX} \geq  \opNorm{\tY^{(1)}} > \max_{|k| > B_1} \TNorm{\thX_{:,:,k}}$, hence $k_1 \leq B_1$ and we have $\tY^{(1)}_{[1]} = \tX_{[1]}$.

For $j = 2,3,\ldots,q$, define $\tX^{(j)} = \tX - \tX_{[j-1]}$, and $\thY^{(j)} \in \ellinf^{m \xx p}$ such that $\thY^{(j)}_{:,:,k} = \thX^{(j)}_{:,:,k}$ for all $k=-B_j,\ldots,B_j$ where $B_j \in \NN$ is such that
$\opNorm{\tY^{(j)}}^2 > \cHNormS{\tX^{(j)}} - \cHNormS{\tY^{(j)}}$, following the explanation above, $\tY^{(j)}_{[1]} = \tX^{(j)}_{[1]} $, thus $\tX_{[j]} = \tX_{[j-1]} + \tY^{(j)}_{[1]}$.

\subsection{Numerical Demonstration}
As a numerical illustration, we consider an instance of the problem of approximating a family of curves in $\RR^{20}$ by a finite dimensional tubal tensor.

We start with a base curve $\bs{\gamma} \colon [-1,1] \to \RR^3$ given by
$$
\bs{\gamma}(t) = \begin{bmatrix}
\cos(\tau_1 + 3t) / (1.5 - 0.5\cos(2t)) \\
\sin(\tau_2 + 2t) / (1.5 - 0.5\cos(3t)) \\
\cos(\tau_3 + 2t)
\end{bmatrix}
$$
where $\tau_1, \tau_2, \tau_3 \in \RR$ are randomly chosen parameters. We then generate two families of curves, $\bs{x}^{(1)},\bs{x}^{(2)}$ by applying the following transformations to $\bs{\gamma}$:
\begin{align*}
\bs{x}^{(1)}_j(t) &= \mat{W}^{(1)}_2 (\mat{W}^{(1)}_1 \bs{\gamma}(t) + \mat{a}_j ) +  \mat{b}_j \\
\bs{x}^{(2)}_j(t) &= \mat{W}^{(2)}_2 (\mat{W}^{(2)}_1 \bs{\gamma}(t) + \mat{a}_j) +  \mat{b}_j
\end{align*}
where the matrices $\mat{W}^{(1)}_1, \mat{W}^{(2)}_1 \in \RR^{200 \times 3}$ and $\mat{W}^{(1)}_2, \mat{W}^{(2)}_2 \in \RR^{20 \times 200}$ are randomly generated and are consistent across all $j$. 
The vectors $\mat{a}_j \in \RR^{200}$ and $\mat{b}_j \in \RR^{20}$ are randomly generated for each $j$. 

We generate 40 random instances of $\bs{x}^{(1)}$ and $\bs{x}^{(2)}$ and stack them into a $80 \xx 20$ quasitubal tensor $\tX$.
\Cref{fig:exp.3dproj,fig:exp.mpcurves} illustrate the structure of the original families of curves.

\begin{figure}[H]
    \centering
    \subcaptionbox{\label{fig:exp.3dproj}}[3in]
    {\includegraphics[width=2in,height=2in]{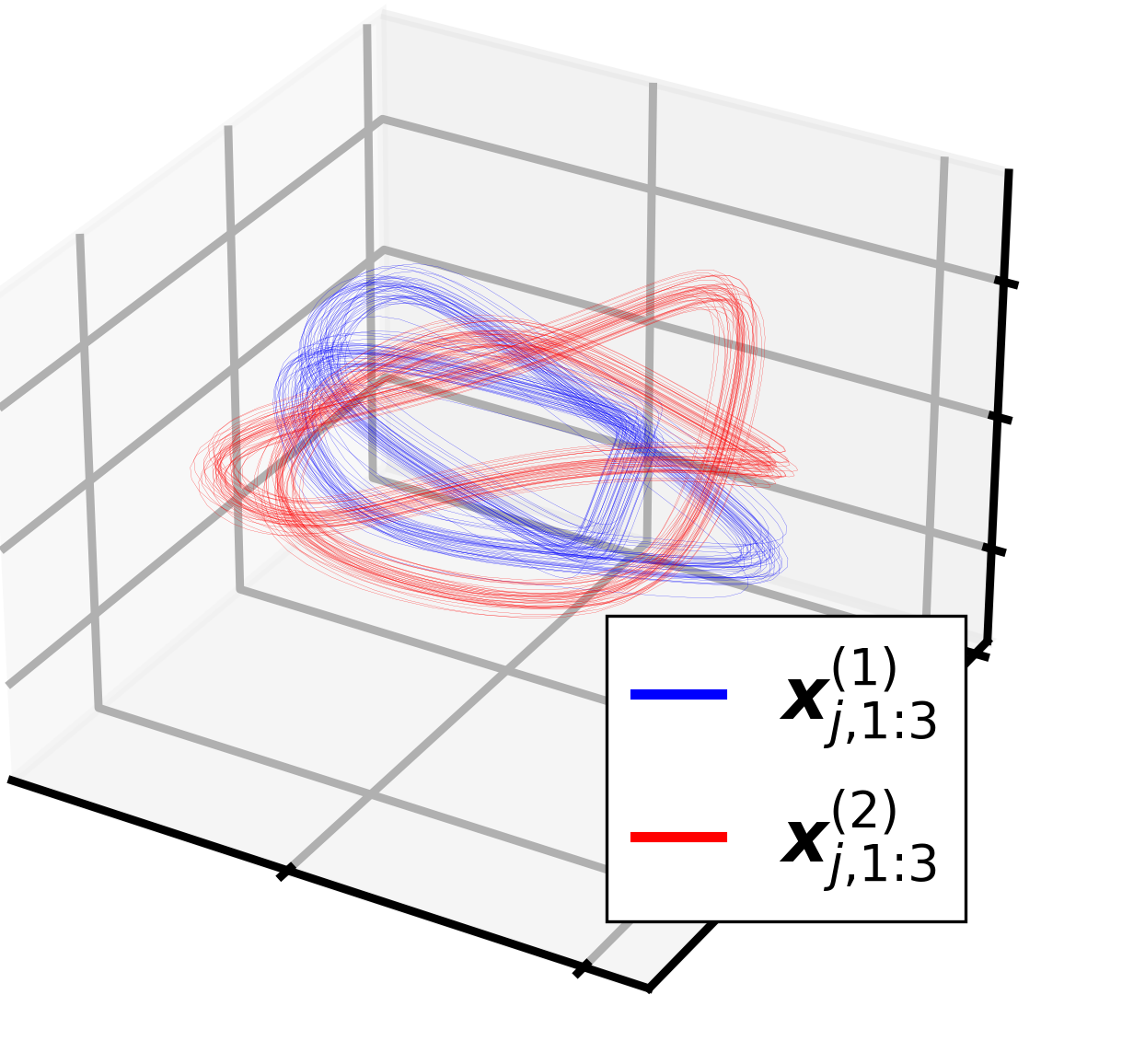}\vspace{-0.5cm}}
    \subcaptionbox{\label{fig:exp.mpcurves}}[3in]
    {\includegraphics[width=2in,height=2in]{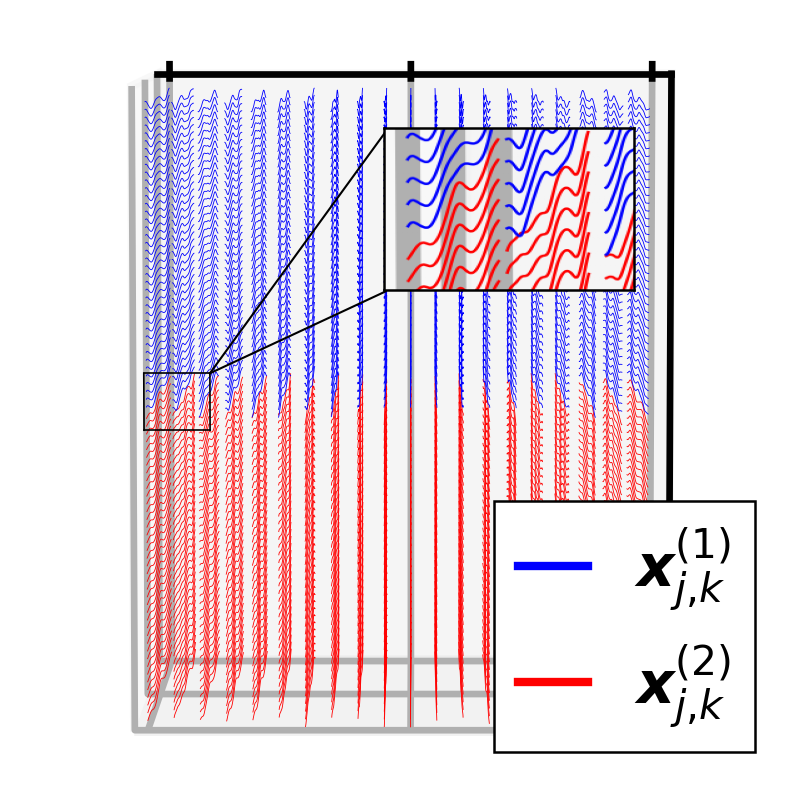}\vspace{-0.5cm}}
    \caption{Description of the original families of multidimensional curves. \Cref{fig:exp.3dproj} shows the trajectory of the first 3 coordinates of the each curve. \Cref{fig:exp.mpcurves} represents each curve as a horizontal collection of 20 ``inward'' facing, 1d functions hence highlighting the `matrix of tubes' structure. Inset shows a close-up view of the first two components of selected representative functions in each of the two curve families. In both graphs, each curve is a single quasitube $\tX_{j,k}$ for $j=1,\ldots,80$ and $k=1,\ldots,20$.}
    \label{fig:illustration.curve}
\end{figure}

The univariate components in $\bs{\gamma}$ are continuous on the compact $[-1,1]$, as a result, the entries $\bs{x}_{hj}$ that are defined by (pointwise) linear combinations of $\bs{\gamma}$-components,  are continuous on $[-1,1]$ and as such are square integrable. 
Therefore we assume that $\tX \in \cH^{80 \xx 20}$, where $\cH$ is a Hilbert space of square integrable functions on $[-1,1]$ endowed with some inner product.



Taking $\cH$ to be the space of square integrable functions equipped with $\dotph{f,g} = \frac{1}{2}\int_{-1}^{1} \bar{f} g $ and the convolution operation as binary multiplication, determines the orthonormal basis $F = \{ \bs{\phi}^{(k)} \}_{k \in \ZZ}  $ where $\bs{\phi}^{(k)} (t) = \exp(i \pi k t) $.
Consequently, the $F$-transform is a mapping between a function $f \in \cH$ and the coefficients in its Fourier series. 
In this particular case, the univariate components in $\bs{\gamma}$ are smooth and periodic functions, thus, the entries of $\tX$ which are (pointwise) linear combinations of $\bs{\gamma}$ are also smooth periodic functions, hence the decrease of Fourier coefficients is expected to be rapid. 
However, for general, not neccessarily periodic, $C^\infty$ functions, this particular construction is prone to exhibit some undesired phenomena (e.g., Gibbs), thus we turn in our demonstration to a less restrictive choice, that is, Chebychev polynomials.


Recall that the Chebyshev polynomials (of the first kind) are defined as $T_{n+1}(x) = 2x T_n(x) - T_{n-1}(x)$ with $T_0(x) = 1, T_1(x) = x$. 
A well-known property of Chebyshev polynomials is that they are orthogonal with respect to the weight function ${t} \mapsto 1/\sqrt{1-t^2}$ on $[-1,1]$, that is $\dotph{T_k,T_j} = \int_{-1}^{1} T_k(t) T_j(t) / \sqrt{1-t^2} \, dt = 0$ for $k \neq j$.
Using the normalization factors $w_0 = \sqrt{\pi}$ and $w_k = \sqrt{\pi/2}$ for all $k \geq 1$, we have $\dotph{T_k/ w_k, T_j/ w_j} = \delta_{kj}$, hence the weighted Chebyshev polynomials form an orthonormal basis for square integrable functions on $[-1,1]$ with respect to the inner product $\dotph{f,g} = \int_{-1}^{1} f(t) g(t) / \sqrt{1-t^2} \, dt$, hence we define $\bs{\phi}^{(n)}(t) = T_n(t) / w_n$ for $n \in \NN$.
Standard, high-level software packages for numerical computing such as Chebfun~\cite{Battles2004} (for MATLAB), or the Python's numeric computing echo-system \cite{Harris2020,Virtanen2020}, provide efficient implementations for working with Chebyshev polynomials.
In this case, we used numpy's chebfit function for computing the coefficients in a Chebyshev polynomial expansion of a function given its values at sample points (which we set as the Chebyshev nodes).

Recall that $\cHNormS{\tX} = \sum_{h=1}^m \sum_{j=1}^p \cHNormS{\bs{x}_{hj}}$. 
$$
\cHNormS{\tX} = \ltNormS{\wh{\tX}} = {\sum}_{n=1}^\infty \FNormS{\wh{\tX}_{:,:,n}}.
$$
Therefore, estimation of the required number of frontal slices needed to approximate $\tX$ with relative error bounded by $\varepsilon$ amounts to finding the minimal $N$ such that 
$$
\frac{ \cHNorm{\tX} -  \sqrt{{\sum}_{n=1}^N \ltNormS{\wh{\tX}_{:,:,n}}}}{ \cHNorm{\tX}} \leq \varepsilon
$$

In our experiment, the setup allows for applying symbolic routines (and even analytic methods) for computing the exact value of the integrals $\cHNormS{\bs{x}_{hj}}$. 
Nevertheless, we have used a common, general purpose implementation (quad function in scipy) with setting both the absolute and relative epsilon parameters to $10^{-12}$, to simulate real life situations where access to analytic expressions of the functions may not be available.
Therefore, we do not expect our estimation of $\cHNormS{\tX}$ to be perfect, and as a consequence, the error of the low rank approximation of $\tX$, that is measured relatively to the estimation of $\cHNormS{\tX}$ based on this numeric integration,  is not expected to be smaller than $10^{-12}$.

\Cref{fig:scre} shows that most of the energy is captured by the first 100 frontal slices, and the improvement stops at around 100 frontal slices. 
We see that our relative error does not drop below $10^{-13}$, as expected, due to the possible errors in the numerical integration scheme.
Computing the qSVD of $\tX$ with respect to first 80 Chebyshev polynomials is equivalent to computing the facewise tubal SVD of a finite tubal tensor  $\wt{\tX} \in \RR^{80 \xx 20 \xx 80}$, whose implicit rank is (trivially) bounded by $20 \cdot 80 $. 

In \Cref{fig:scre.rank1}, we see that the number of rank-1 components needed for  reconstruction  of the signal is no more than 180, thus, the implicit rank is significantly lower than the trivial upper bound.

\begin{figure}[H]
    \centering
    \subcaptionbox{\label{fig:scre}}{\includegraphics[height=2in]{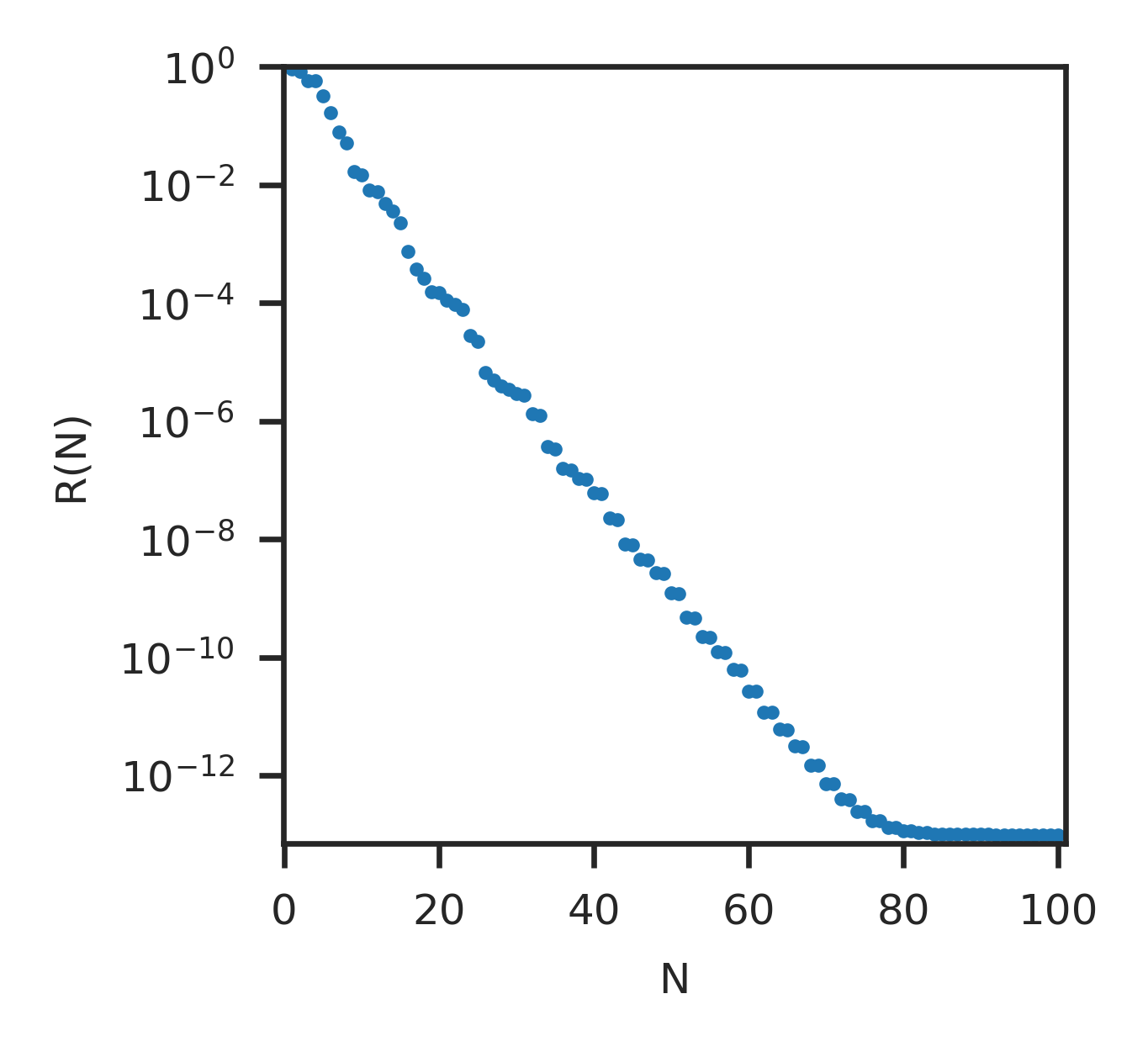}}\hfill
    \subcaptionbox{\label{fig:scre.rank1}}{\includegraphics[height=2in]{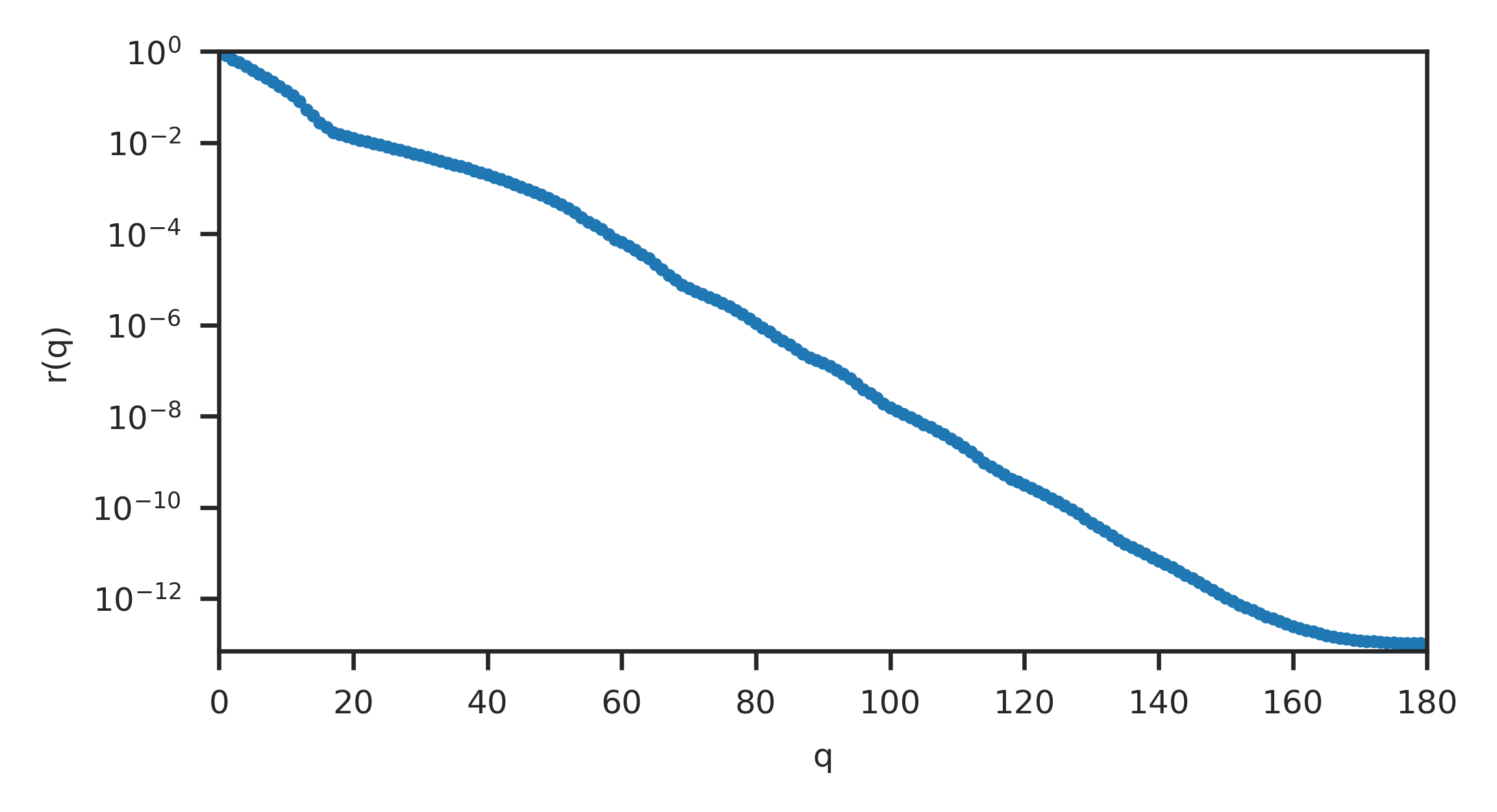}}
    
    \caption{Approximation quality as a function of truncation size. \Cref{fig:scre} depicts the error obtained by truncation of the signal to the first $N$ frontal slices, relative to the original signal, $R(N) = (\cHNormS{\tX} - \FNormS{\thX_{:,:,1:N}}) / \cHNormS{\tX}$.
    \Cref{fig:scre.rank1} describes the relative approximation error of the best, rank-$q$ truncation,  $r(q) = (\cHNormS{\tX} - \FNormS{\thX_{[q]}}) / \cHNormS{\tX}$.}\label{fig:scre.rec}
\end{figure}

Visual inspection of the reconstruction affirms this the case, as we clearly see that best approximations of rank 10 \Cref{fig:rec10}  and 20 \Cref{fig:rec20}, corresponding to relative approximation error greater than $10^{-1}$ and $10^{-2}$ respectively \Cref{fig:scre.rank1},  do not produce a reasonable reconstruction of the signal (not even periodic), whereas rank-160 approximation \Cref{fig:rec.all} for which the relative error is around $10^{-12}$ result in a visually satisfactory reconstruction.   

\begin{figure}[H]
    \centering
    \subcaptionbox{\label{fig:rec10}}{\includegraphics[height=1.9in]{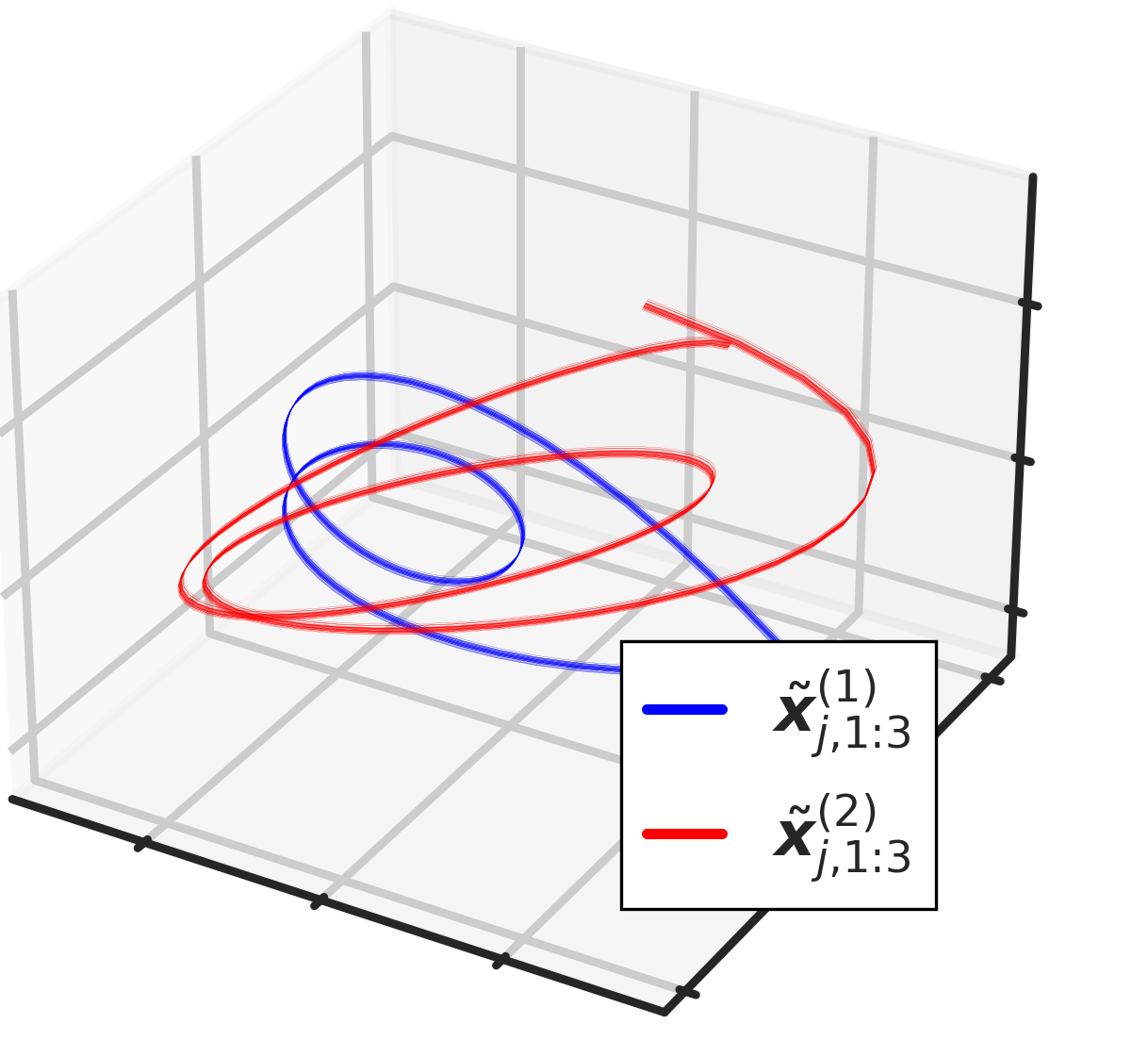}}\hfill
    \subcaptionbox{\label{fig:rec20}}{\includegraphics[height=1.9in]{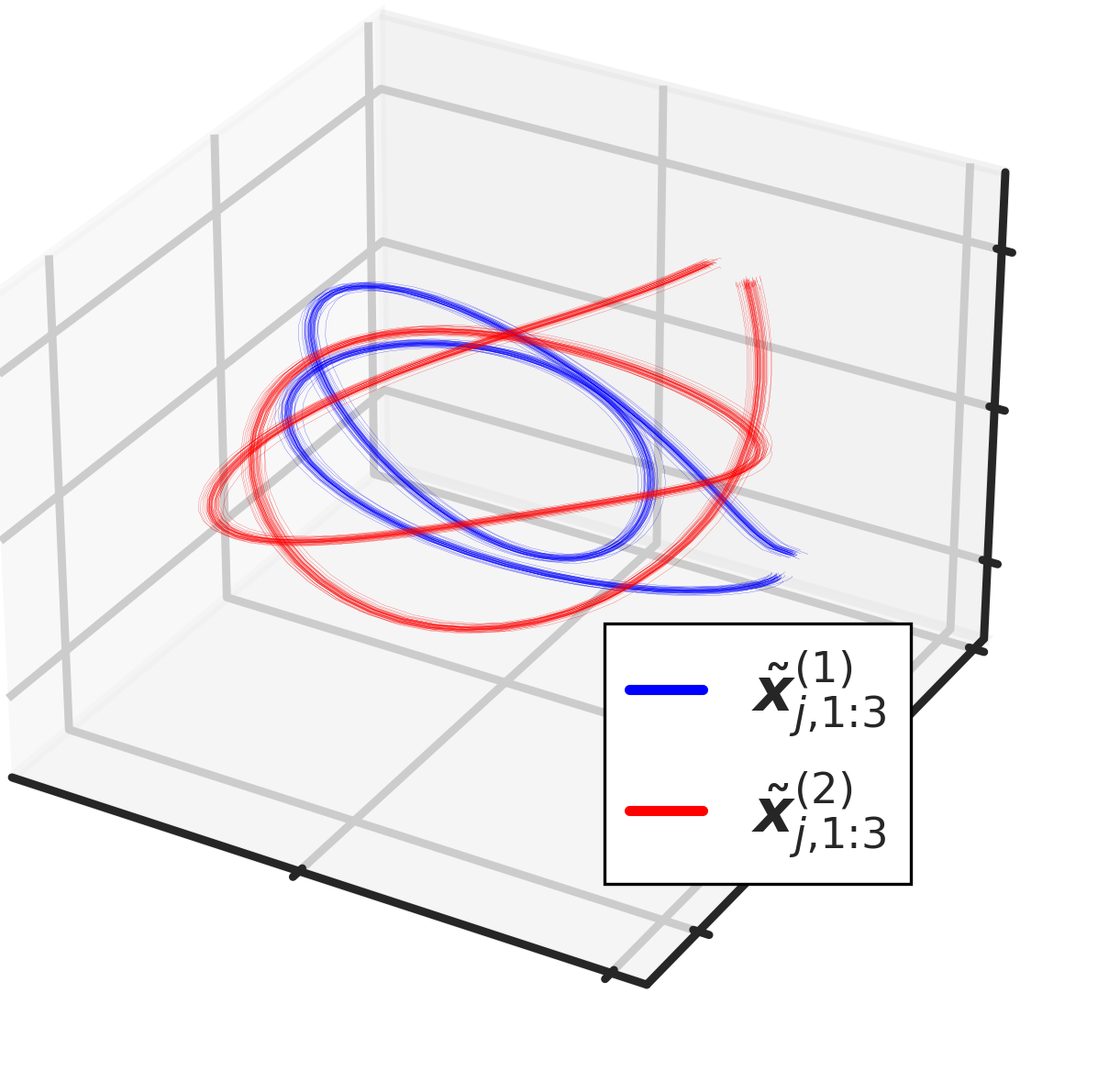}}\hfill
    \subcaptionbox{\label{fig:rec.all}}{\includegraphics[height=1.9in]{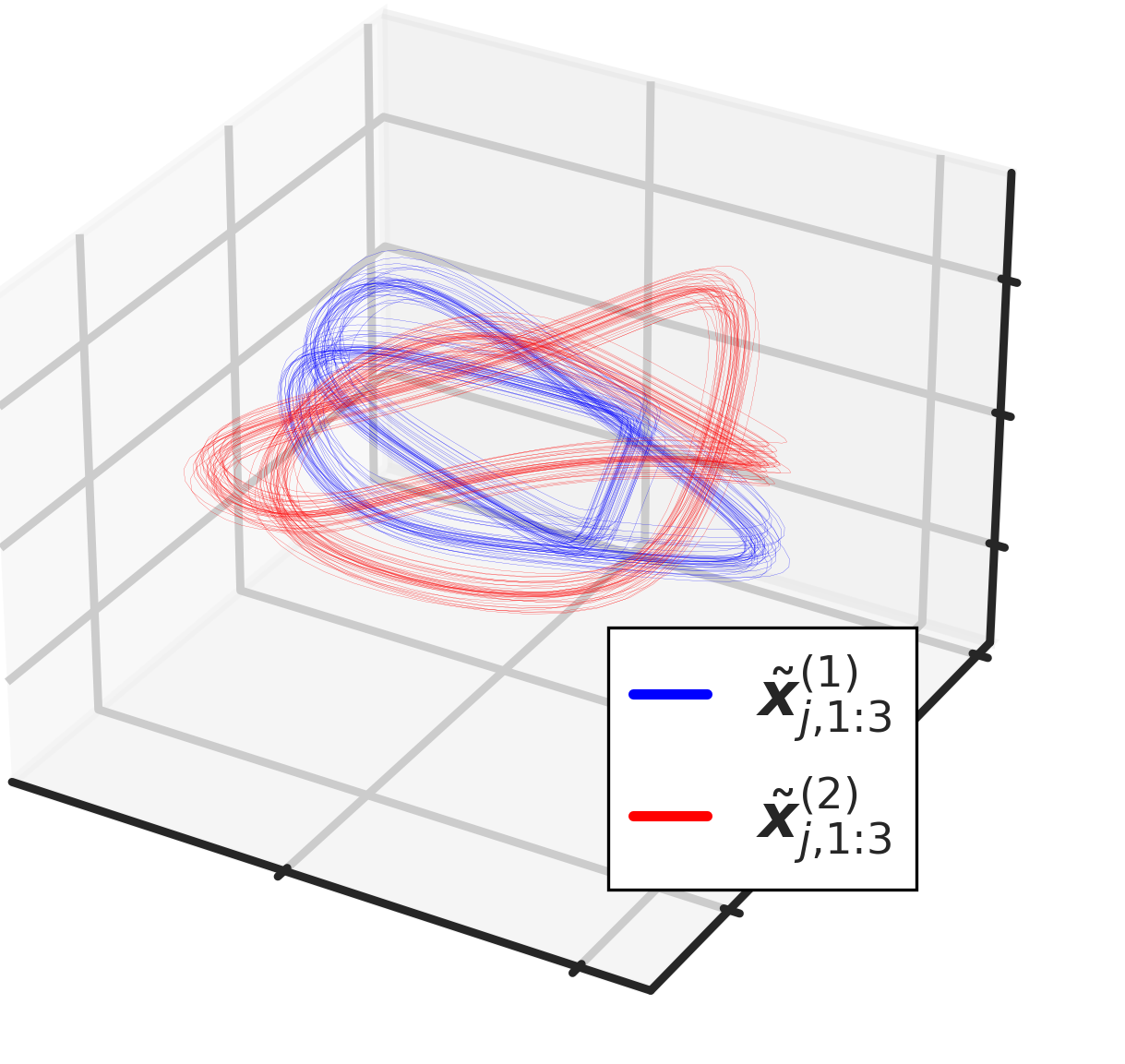}}
    
    \caption{
    Three dimensional curves, corresponding to the first 3 components in the reconstructed signal, using 10 components (\Cref{fig:rec10}), 20 (\Cref{fig:rec20}) and 160 (\Cref{fig:rec.all}).
    }\label{fig:rec.curves}
\end{figure}


Works such as~\cite{Mor2022,Han2023} have used low dimensional  representations to form an informative latent space representation of data.
Concretely, the TCAM method in~\cite{Mor2022} used the first columns of a matrix $\wh{\matZ} $ obtained by mode 1 unfolding of the tubal tensor $\tZ = \tU \Mprod \tS$.
Doing the same for the qSVD, we obtain a matrix $\wh{\matZ}_{[q]} = \thU_{:,1:l_q,1:t_q} \triangle \thU_{1:l_q,1:l_1,1:t_q}$.
We have noticed, for example, that the first and fourth components in of the domain transform, separate the two families of curves as shown in \Cref{fig:exp1.scatter}.
This interpretation for \Cref{fig:exp1.scatter} and the conditions under which it holds are beyond the scope of this work, and will be addressed in future work.

\begin{figure}[H]
    \centering
    \includegraphics[width=2in,height=2in]{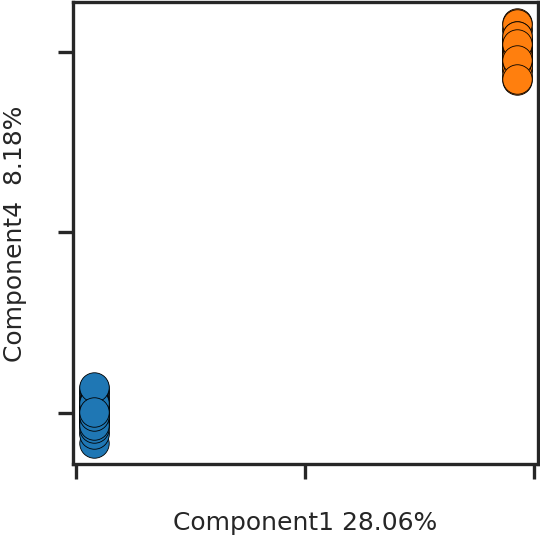}
    \caption{ A scatter plot of the first and fourth components of the domain transform, i.e., $\wh{\tZ}_{:,l_1,1:t_1}$ and $\wh{\tZ}_{:,l_3,1:t_3}$. The two families of curves are clearly separated.} \label{fig:exp1.scatter}
\end{figure}


\section{Conclusions and Future Direction}
This work lays the theoretical foundation for a quasitubal tensor framework, which extends matrix-mimetic tubal tensor framework to infinite dimensional, separable Hilbert spaces, i.e. quasitensor.

Since no Hilbert space can form a unital algebra, we set our construction in a larger, that is, the \textit{quasitubal} algebra which is a commutative, unital algebra, containing the separable Hilbert space as an ideal. 
Once endowed with a commutative algebra structure, a separable Hilbert space can be considered as a module over the (non unital) ring that is itself.
Then, the quasitubal algebra is the dual module of the separable Hilbert space.
We show that the quasitubal algebra is a commutative, unital C*-algebra, and quasitubal tensors are, therefore, homomorphisms between Hilbert C*-modules over this algebra.
As such, they realize the algebraic and geometric notions of the finite dimensional tubal tensor framework, namely, identity, adjoint and orthogonality, which are the key ingredients for the existence of a quasitubal singular value decomposition.
As for the qSVD, we established that for any separable Hilbert space (and implied inner product), and a tubal multiplication $\ff$ for which certain properties hold, 
the optimal low rank approximation of a quasitubal tensor is given by the truncations of its qSVD, i.e., prove Eckart-Young-like results. 


 The Eckart-Young-like results are tied directly to how we define the qausitubal product, which, in turn, is defined by an orthonormal basis $F$. Like the finite-dimension tubal case, i.e., based on tSVD and tSVDM, the quality of  of the resulting approximations depends on  the choice of tubal multiplication (equivalently, the choice of the orthonormal basis for the domain transform). 
Finding the most suitable orthonormal basis (and inner product) for the input tensor is equivalent to finding a global sense optimal low rank approximation for it (optimal CP decomposition), a problem that is generally very difficult to solve, even in the finite dimensional case.
In the infinite dimensional case, the situation is similarly difficult. 
Nevertheless, the theoretical framework we provide allows us to reason about the proper choice of the domain transform, in terms of the properties one expects from the input, as we showed in our numerical example. 

Looking forward, there are several directions to explore.
From the practical perspective, there are several immediate applications of the quasitubal tensor framework that we believe are worth exploring: modeling and analyzing linear-time-invariant systems, data driven identification, control and model order reduction of continuous time dynamical systems, and so on. From a theoretical point of view, several questions remain. For example, how to define and establish optimality results for non separable Hilbert spaces, or Banach spaces with semi-definite inner product.

\phantomsection{}

\bibliographystyle{abbrv}

\bibliography{./refs}

\newpage

\end{document}